\documentclass[fleqn,reqno,11pt,a4paper,final]{amsart}

\usepackage[a4paper,left=20mm,right=20mm,top=20mm,bottom=20mm,marginpar=20mm]{geometry}
\usepackage[english]{babel}
\usepackage{amsmath}
\usepackage{amssymb}
\usepackage{amsthm}
\usepackage{amscd}
\usepackage[utf8]{inputenc}
\usepackage{color}
\usepackage[english=american]{csquotes}
\usepackage[final]{graphicx}
\usepackage{hyperref}
\usepackage{calc}
\usepackage{mathptmx}
\usepackage{mathtools}
\usepackage{textcmds} 
\usepackage{tikz}
\usepackage{graphicx}
\usepackage{float}
\usepackage{stix}
\usepackage{enumitem}
\usepackage{dsfont}
\usepackage{oldgerm}
\usepackage{makecell}
\usepackage{caption}
\usepackage{pgfplots}
\pgfplotsset{compat=1.18}
\usepackage{tikz}
\usetikzlibrary{fadings,shadings}
\usetikzlibrary{decorations.markings}
\usepackage{caption}
\usepackage{subcaption}
\usepackage[percent]{overpic}

\usepackage[backend=biber,bibencoding=utf8,giveninits,maxnames=4,style=alphabetic,isbn=false, doi=false, eprint= true, url= false,date=year]{biblatex}
\usepackage[english=american]{csquotes}
\bibliography{ResonantInstabilities.bib}

\DeclareRobustCommand{\gobblefour}[5]{}
\newcommand*{\SkipTocEntry}{\addtocontents{toc}{\gobblefour}}

\definecolor{luh-dark-blue}{rgb}{0.0, 0.313, 0.608}

\graphicspath{{../Pictures/}}

\makeatletter
\renewcommand\paragraph{\@startsection{paragraph}{4}{\z@}%
	{1ex \@plus1ex \@minus.2ex}%
	{-1em}%
	{\normalfont\normalsize\bfseries}}
\makeatother

\numberwithin{equation}{section}

\newtheoremstyle{thmlemcorr}{10pt}{10pt}{\itshape}{}{\bfseries}{.}{10pt}{{\thmname{#1}\thmnumber{ #2}\thmnote{ (#3)}}}
\newtheoremstyle{thmlemcorr*}{10pt}{10pt}{\itshape}{}{\bfseries}{.}\newline{{\thmname{#1}\thmnumber{ #2}\thmnote{ (#3)}}}
\newtheoremstyle{remexample}{10pt}{10pt}{}{}{\bfseries}{.}{10pt}{{\thmname{#1}\thmnumber{ #2}\thmnote{ (#3)}}}
\newtheoremstyle{ass}{10pt}{10pt}{}{}{\bfseries}{.}{10pt}{{\thmname{#1}\thmnumber{ A#2}\thmnote{ (#3)}}}

\theoremstyle{thmlemcorr}
\newtheorem{theorem}{Theorem}
\numberwithin{theorem}{section}
\newtheorem{lemma}[theorem]{Lemma}

\newtheorem{proposition}[theorem]{Proposition}

\theoremstyle{thmlemcorr*}
\newtheorem*{theorem*}{Theorem}
\newtheorem{lemma*}[theorem]{Lemma}
\newtheorem{corollary*}[theorem]{Corollary}
\newtheorem{proposition*}[theorem]{Proposition}
\newtheorem{problem*}[theorem]{Problem}
\newtheorem{conjecture*}[theorem]{Conjecture}
\newtheorem{definition*}[theorem]{Definition}
\newtheorem{assumption*}[theorem]{Assumption}

\theoremstyle{remexample}
\newtheorem{remark}[theorem]{Remark}

\theoremstyle{ass}

\newcommand{\Acal}{\mathcal{A}}
\newcommand{\Bcal}{\mathcal{B}}
\newcommand{\Ccal}{\mathcal{C}}

\newcommand{\Fcal}{\mathcal{F}}

\newcommand{\Ical}{\mathcal{I}}

\newcommand{\Lcal}{\mathcal{L}}
\newcommand{\Mcal}{\mathcal{M}}
\newcommand{\Ncal}{\mathcal{N}}
\newcommand{\Ocal}{\mathcal{O}}

\newcommand{\Rcal}{\mathcal{R}}
\newcommand{\Scal}{\mathcal{S}}
\newcommand{\Tcal}{\mathcal{T}}

\newcommand{\Wcal}{\mathcal{W}}
\newcommand{\Xcal}{\mathcal{X}}
\newcommand{\Ycal}{\mathcal{Y}}
\newcommand{\Zcal}{\mathcal{Z}}

\renewcommand{\Re}{\operatorname{Re}}
\renewcommand{\Im}{\operatorname{Im}}
\newcommand{\ee}{\mathrm{e}}
\newcommand{\ii}{\mathrm{i}}

\newcommand{\norm}[1]{\|#1\|}

\newcommand{\abs}[1]{|#1|}

\newcommand{\dd}{\;\mathrm{d}}

\newcommand{\N}{\mathbb{N}}
\newcommand{\R}{\mathbb{R}}
\newcommand{\C}{\mathbb{C}}

\newcommand{\Z}{\mathbb{Z}}

\newcommand{\eps}{\varepsilon}

\newcommand{\Res}{\operatorname{Res}}

\def\XXint#1#2#3{{\setbox0=\hbox{$#1{#2#3}{\int}$}
		\vcenter{\hbox{$#2#3$}}\kern-.5\wd0}}

\renewcommand{\eps}{\varepsilon}

\newcommand{\kt}{k_\mathrm{T}}
\newcommand{\kth}{k_\mathrm{TH}}
\newcommand{\lt}{\lambda_\mathrm{T}}
\newcommand{\lth}{\lambda_\mathrm{TH}}

\newcommand{\Aav}{A_\mathrm{av}}
\newcommand{\Bav}{B_\mathrm{av}}
\newcommand{\Acalav}{\Acal_\mathrm{av}}

\begin{document}
	
	
	\title[]{Pattern formation and nonlinear waves close to a 1:1 resonant Turing and Turing--Hopf instability}
	
	\author{Bastian Hilder}
	\address{\textit{Bastian Hilder:}  Department of Mathematics, Technische Universität München, Boltzmannstraße 3, 85748 Garching b.\ München, Germany}
	\email{bastian.hilder@tum.de}
	
	\author{Christian Kuehn}
	\address{\textit{Christian Kuehn:}  Department of Mathematics, Technische Universität München, Boltzmannstraße 3, 85748 Garching b.\ München, Germany}
	\email{ckuehn@ma.tum.de}
	
	\begin{abstract}
		In this paper, we analyse the dynamics of a pattern-forming system close to simultaneous Turing and Turing–Hopf instabilities, which have a 1:1 spatial resonance, that is, they have the same critical wave number. For this, we consider a system of coupled Swift–Hohenberg equations with dispersive terms and general, smooth nonlinearities. Close to the onset of instability, we derive a system of two coupled complex Ginzburg–Landau equations with a singular advection term as amplitude equations and justify the approximation by providing error estimates. We then construct space-time periodic solutions to the amplitude equations, as well as fast-travelling front solutions, which connect different space-time periodic states. This yields the existence of solutions to the pattern-forming system on a finite, but long time interval, which model the spatial transition between different patterns. The construction is based on geometric singular perturbation theory exploiting the fast travelling speed of the fronts. Finally, we construct global, spatially periodic solutions to the pattern-forming system by using centre manifold reduction, normal form theory and a variant of singular perturbation theory to handle fast oscillatory higher-order terms.
	\end{abstract}
	\vspace{4pt}
	
	\maketitle
	

	
	
	\setcounter{tocdepth}{1}
	\tableofcontents

	\section{Introduction}
	
	The dynamical behaviour of pattern-forming systems is often organised by linear instabilities. One of the most famous examples of which is the Turing, or diffusion-driven, instability \cite{turing1952}, which typically leads to the formation of stationary, spatially periodic patterns and was used to explain pattern-formation in different dynamical systems arising in science and engineering. A related instability is the Turing–Hopf instability, which typically leads to the formation of travelling, spatially periodic wavetrains and occurs usually in pattern-forming systems which include advective effects, see e.g.~\cite{rovinsky1992}.
	
	Beyond dynamics driven by a single instability, the simultaneous occurrence of multiple instabilities can also be found in many applications. Examples are single and multilayer thermal convections, where the interaction of multiple Turing instabilities can occur, see e.g.~\cite{jones1987,proctor1988,cox1996,echebarria1997,mercader2001,prat2002} as well as magneto-hydrodynamics \cite{fujimura1998}. Specifically, the simultaneous occurrence of a Turing and a Turing–Hopf instability has been found in the Brusselator and Oregonator models, which describe chemical reaction dynamics, see \cite{yang2002b}, as well as in the Taylor–Couette problem, see \cite{chossat1994}. These higher-codimension instabilities can lead to a large variety of complex patterns.
	
	In this paper, we rigorously discuss the dynamics in a pattern-forming system close to resonant Turing and Turing–Hopf instabilities. Specifically, we consider a phenomenological model of two coupled Swift-Hohenberg-type equations
	\begin{equation}
		\label{eq:toy-model}
		\begin{split}
			\partial_t u &= -(\kt^2 + \partial_x^2)^2 u + \varepsilon^2 \alpha_u u + f(u,v), \\
			\partial_t v &= -(\kth^2 + \partial_x^2)^2 v + c_d \partial_x^3 v + \varepsilon^2 \alpha_v v + g(u,v),
		\end{split}
	\end{equation}
	Here, $u(t,x), v(t,x) \in \R$, $\alpha_u, \alpha_v, c_d \in \R$, $0 < \varepsilon \ll 1$ and $x \in \R$ denotes the spatial coordinate, whereas $t \geq 0$ denotes time. In addition, $f$ and $g$ are smooth nonlinerities in $(u,v)$, in particular,
	\begin{equation*}
		f(0,0) = \mathrm{D}f(0,0) = g(0,0) = \mathrm{D}g(0,0) = 0.
	\end{equation*}
	Therefore, the system \eqref{eq:toy-model} has the trivial steady state $(u,v) = (0,0)$ for all $\varepsilon > 0$. Furthermore, since $f$ and $g$ are smooth, we can consider a Taylor expansion around $(u,v) = (0,0)$, which is given by
	\begin{equation}\label{eq:taylor-expansion-nonlinearities}
		\begin{split}
			f(u,v) &= f_{20} u^2 + f_{11} uv + f_{02} v^2 + f_{30} u^3 + f_{21} u^2 v + f_{12} u v^2 + f_{03} v^3 + \Ocal(\norm{(u,v)}^4), \\ 
			g(u,v) &= g_{20} u^2 + g_{11} uv + g_{02} v^2 + g_{30} u^3 + g_{21} u^2 v + g_{12} u v^2 + g_{03} v^3 + \Ocal(\norm{(u,v)}^4)
		\end{split}
	\end{equation}
	with coefficients $f_{ij}, g_{ij} \in \R$ and $\Ocal(\norm{(u,v)}^4)$ are fourth-order terms understood in the usual asymptotic order notation in the limit $\norm{(u,v)}\rightarrow 0$.
	
	\begin{remark}
		Although we restrict the discussion to the model \eqref{eq:toy-model}, we expect that the results apply also to other systems featuring the same type of resonant instability. In fact, it is very typical to find that the dynamics close to the onset of an instability are similar for a classes of system with the same instability since the dynamics close to onset is described by universal amplitude equations, see e.g.~\cite{schneider2017a}.
	\end{remark}
	
	Linearising the system \eqref{eq:toy-model} about the trivial state $(u,v) = (0,0)$ and solving the resulting linear system with the ansatz $(U,V) = \exp(\lambda(k) t + ikx) \varphi$ with wave number $k \in \R$ yields the dispersion relations or eigenvalue curves
	\begin{equation*}
		\lt(k) = -\left(\kt^2-k^2\right)^2 + \varepsilon^2 \alpha_u, \quad \lth(k) = -\left(\kth^2-k^2\right)^2 - i c_d k^3 + \varepsilon^2 \alpha_v,
	\end{equation*}
	see Figure \ref{fig:dispersionRel}, with associated eigenvectors $\varphi_\mathrm{T} = (1,0)^2$ and $\varphi_\mathrm{TH} = (0,1)^T$, respectively. Therefore, the system \eqref{eq:toy-model} undergoes a Turing instability with wave number $k = \kt$ at $\alpha_u = 0$ and a Turing-Hopf instability with wave number $k = \kth$ at $\alpha_v = 0$.
	
	\begin{figure}
		\centering
		\includegraphics[width=0.6\textwidth]{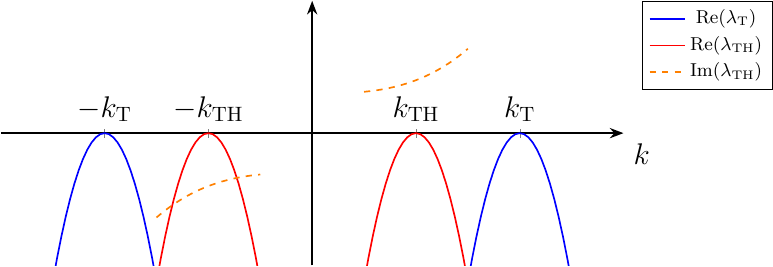}
		\caption{Plot of the eigenvalue curves $\lt$ and $\lth$ for $\kt \neq \kth$ at $\alpha_u = \alpha_v = 0$.}
		\label{fig:dispersionRel}
	\end{figure}
	
	Close to the onset of instability one typically can expect that any non-decaying solution to \eqref{eq:toy-model} will, to leading order, be given by the critical Fourier modes $\ee^{\ii\kt x} \varphi_\mathrm{T}$ and $\ee^{\ii \kth (x-c_p t)} \varphi_\mathrm{TH}$ with the phase velocity $c_p = \Im(\lth(\kth))/\kth$. If the wave numbers $\kt$ and $\kth$ are \emph{non-resonant}, that is, neither $\kt/\kth \in \N$ nor $\kth/\kt \in \N$, any cross-interaction of the critical Fourier modes will be exponentially damped. Therefore, the case of \emph{resonant} wave numbers, that is,
	\begin{equation*}
		\dfrac{\kt}{\kth} = N \in \N \text{, or } \dfrac{\kth}{\kt} = N \in \N
	\end{equation*}
	is of special interest.	In the following, we restrict to the case that $\kt = \kth = k^* = 1$. Therefore, both instabilities occur at the same wave number, which motivates the notation of a \emph{1:1 resonant Turing and Turing–Hopf instability}. Note that the choice of $k^* = 1$ is done for pure notational convenience and does not change the following analysis. In fact, the case of $k^* \neq 1$ can be recovered by a spatial rescaling.
	
	\begin{remark}\label{rem:different-resonances}
		We expect that the approach pursued in this paper can be used to other pairs of resonant wave numbers. The main difference compared to the present paper lies in the different nonlinearities occurring in the amplitude equations. For example, a 1:2 resonance also allows for a critical quadratic interaction, cf.~also \cite{gauss2021} for the case of a 1:2-resonant Turing and Turing instability. We leave the detailed discussion of these other resonances to future research.
	\end{remark}
	
	\subsection{Main results of the paper}
	
	The goal of this paper is to understand the dynamics of \eqref{eq:toy-model} close to the onset of a 1:1 resonant Turing and Turing–Hopf instability. For this, we use a similar strategy as in \cite{gauss2021}, where the dynamics close to two Turing instabilities with a spatial 1:2 resonance was studied. That is, we use modulation theory and amplitude equations to study the dynamics close to the onset of instability. Next, we summarise the main results and techniques used in this paper.
	
	\begin{itemize}
		\item We derive the amplitude equations \eqref{eq:final-amplitude-eq} close to the onset of instability and give a rigorous justification, see Theorem \ref{thm:full-approx-result}.
		\item We prove the existence of space-time periodic solutions to the amplitude equations \eqref{eq:final-amplitude-eq} as well as the existence of fast-moving front solutions connecting these space-time periodic solutions, see Theorem \ref{thm:persistence-heteroclinic-orbits}. Additionally, we obtain solutions to the full pattern-forming system \eqref{eq:toy-model}, which are close to spatio-temporal patterns and pattern interfaces on a long time interval, see Theorems \ref{thm:pure-patterns}, \ref{thm:patterns-space-time} and \ref{thm:pattern-interface} and Figures \ref{fig:T-to-trivial}--\ref{fig:TH-to-superpos}.
		\item We prove the existence of globally bounded, spatially periodic solutions to the pattern-forming system \eqref{eq:toy-model}, which are close to pure Turing and Turing–Hopf patterns, see Theorem \ref{thm:pure-patterns-global}, or resemble a superposition of Turing and Turing–Hopf patterns, see Theorem \ref{thm:superposition-pattern-global}.
	\end{itemize}
	
	\paragraph{Justification of amplitude equations}
	
	We formally derive a system of amplitude equation close to the onset of instability by making the ansatz
	\begin{equation}\label{eq:ansatz-amplitude-intro}
		u(t,x) = \varepsilon A(\varepsilon^2 t, \varepsilon x) \ee^{\ii x} + c.c. + h.o.t., \quad v(t,x) = \varepsilon B(\varepsilon^2 t, \varepsilon(x - c_g t)) \ee^{\ii (x-c_p t)} + c.c. + h.o.t.,
	\end{equation}
	where  $c.c.$ denote complex conjugated terms of the first summand, $A(T,X_1), B(T,X_2) \in \C$ are slow amplitude modulations of the critical Fourier modes as $\varepsilon$ is a small positive parameter, $X_1 = \varepsilon x$, $X_2 = \varepsilon (x - c_g t) = X_1 - \tfrac{c_g}{\varepsilon}T$ with $c_g \neq 0$, and $h.o.t.$ denote higher-order terms. By inserting this ansatz into the pattern-forming system \eqref{eq:toy-model} and equating different powers of $\varepsilon$ to zero, we find the linear phase velocity $c_p = c_d$ and the linear group velocity $c_g = 3 c_d$. In addition, we find a system of amplitude equations for $A$ and $B$, see \eqref{eq:full-amplitude-eq}. 
	
	Although these formal calculations suggest that the dynamics of the pattern-forming system \eqref{eq:toy-model} can be described by the set of amplitude equations obtained from this multiscale approach, this is not obvious and can fail. Indeed, there is a number of counter-examples, where an amplitude equation can be derived formally, but makes wrong predictions, see e.g.~\cite{schneider2015,haas2020,baumstark2020}. Therefore, the rigorous justification is crucial. In the case at hand, the formal derivation can indeed be made rigorous using a standard approach as in \cite{schneider1997}, see also \cite{schneider2017a}. For this, we prove error estimates between solutions of the full pattern-forming system \eqref{eq:toy-model} and the ansatz \eqref{eq:ansatz-amplitude-intro}, where the amplitude modulations are given as solutions to the amplitude equations \eqref{eq:full-amplitude-eq}. 
	
	However, the amplitude system \eqref{eq:full-amplitude-eq} not only depends on the slow time $T$, but also on the fast time $t$. This is reflected in two ways. First, fast oscillating terms of the form $\ee^{\ii n c_p t}$ with $n \in \Z$ arise through the interaction of $A$ and $B$ since the critical Fourier modes belonging to both amplitudes have different phase velocities. Second, since $A$ and $B$ are given in different spatial coordinate frames $X_1 = \varepsilon x$ and $X_2 = \varepsilon (x - c_g t) = X_1 - \tfrac{c_g}{\varepsilon}T$ with $c_g \neq 0$, their interaction implicitly contains a fast shift by $\pm c_g t$. 
	
	We deal with this in the following way. The fast-oscillating terms can be moved to higher order terms by spatial averaging following the strategy in \cite{schneider1997}, see also Lemma \ref{lem:averaging-amplitude-equations}. In addition, we make the fast shift due to the different spatial coordinate frames more explicit by shifting $B$ into the $X_1$-coordinate frame, which leads to the addition of a singular advection term $\tfrac{c_g}{\varepsilon} \partial_{X_1} B$ in the $B$-equation. This yields the amplitudes equations
	\begin{equation*}
		\begin{split}
			\partial_T \Aav &= 4\partial_{X_1}^2 \Aav + \alpha_u \Aav + \gamma_1 \Aav |\Aav|^2 + \gamma_2 \Aav |\Bav|^2, \\
			\partial_T \Bav &= (4+3ic_v) \partial_{X_1}^2 \Bav - \frac{c_g}{\eps} \partial_{X_1} \Bav + \alpha_v \Bav + \gamma_7 \Bav |\Aav|^2 + \gamma_8 \Bav |\Bav|^2,
		\end{split}
	\end{equation*}
	see also \eqref{eq:final-amplitude-eq}. Notice that, since we do not prescribe any localisation or periodicity in $A$ and $B$, the amplitude equations are fully coupled. This is in contrast to, e.g.~\cite{schneider1997}, where additionally assuming periodicity yields amplitude equations which are only coupled through the mean values of $A$ and $B$. Finally, assuming sufficient localisation for $A$ and $B$ would yield to a decoupled system, since nonlinear interactions can be moved into higher order terms exploiting that the difference velocity between $X_1$ and $X_2$ is singular in $\varepsilon$, see e.g.~\cite[Remark 10.7.6]{schneider2017a}.
	
	Finally, using the approximation result Theorem \ref{thm:approximation-full-amplitude-eq}, we also obtain that for any sufficiently regular solution to the amplitude equations exists a solution to the pattern-forming system \eqref{eq:toy-model} of the form
	\begin{equation*}
		\begin{split}
			u(t,x) = \varepsilon \Aav(T,X_1) \ee^{\ii x} + c.c. + \Ocal{(\varepsilon^2)}, \qquad v(t,x) = \varepsilon \Bav(T,X_1) \ee^{\ii (x-c_p t)} + c.c. + \Ocal{(\varepsilon^2)},
		\end{split}
	\end{equation*}
	where the error terms are small in $(H_{l,u}^1)^2$ asymptotically as $\varepsilon\rightarrow 0$, see Theorem \ref{thm:full-approx-result}. Note that since $H^1_{l,u}$ embeds into the space of continuous functions, the error is also small uniformly in space.
	
	\paragraph{Modulating dynamics}
	
	After justifying the amplitude system \eqref{eq:final-amplitude-eq}, we first study the existence of time-periodic solutions 
	\begin{equation*}
		\Aav(T) = r_A \ee^{\ii\omega_A T}, \quad \Bav(T) = r_B \ee^{\ii \omega_B T},
	\end{equation*}
	with temporal wave numbers $\omega_A$, $\omega_B$ and radii $r_A$, $r_B$. Under suitable parameter conditions, we obtain three types of solutions: (i) the trivial solutions with $r_A = r_B = 0$, (ii) the semi-trivial solutions with either $r_A = 0$ or $r_B = 0$, but not both, and (iii) fully nontrivial solutions with $r_A \neq 0$ and $r_B \neq 0$. By an implicit function theorem argument, we can extend the existence result to solutions with a small, but non-zero spatial wave number, see Proposition \ref{prop:space-time-periodic-waves}. Using the rigorous approximation result Theorem \ref{thm:full-approx-result}, we can relate these solutions to the amplitude equations to solutions in the pattern-forming system \eqref{eq:toy-model}. The semi-trivial solutions correspond to a pure Turing pattern if $r_B = 0$, that is, a spatially periodic solution, which is stationary to leading order, and a pure Turing–Hopf pattern if $r_A = 0$, that is, a spatially periodic solution with phase velocity close to $c_p$, see Theorem \ref{thm:pure-patterns}. Additionally, the fully nontrivial solutions correspond, to leading order, to a superposition of a stationary, periodic solution and a spatially periodic wavetrain with non-zero phase velocity, see Theorem \ref{thm:pure-patterns}.
	
	Next, we consider fast-moving fronts between the space-time periodic solutions, that is, we show the existence of solutions of the form
	\begin{equation*}
		\Aav(T,X_1) = r_A(\tilde{\xi}) \ee^{\ii(\phi_A(\tilde{\xi}) + \omega_A T)}, \quad \Bav(T,X_1) = r_B(\tilde{\xi}) \ee^{\ii(\phi_B(\tilde{\xi}) + \omega_B T)}
	\end{equation*}
	with a fast co-moving frame $\tilde{\xi} = X_1 - \tfrac{c_0}{\varepsilon}T$, $c_0 \neq 0$. The radii of the solutions are given by $r_A$, $r_B$, while the phases are given by $\phi_A$, $\phi_B$. In addition, we prescribe a given pair of temporal wave numbers $\omega_A$, $\omega_B$. While we are, in principle, free to choose these temporal wave numbers, we restrict to the temporal wave numbers of the semi-trivial time-periodic solutions. For other choices, the temporal wave number of the asymptotic states connected by the front does not change, however, there is an additional spatial wave number, see Remark \ref{rem:different-omegas}.
	
	Inserting the ansatz into the amplitude equations \eqref{eq:final-amplitude-eq}, we obtain a singularly perturbed travelling wave ODE, see \eqref{eq:travelling-fronts-ode}, due to the fast speed $c = \tfrac{c_0}{\varepsilon}$ of the front. To leading order, the equations for the radii $r_A$, $r_B$ are decoupled from the phases $\phi_A$, $\phi_B$. After rescaling, the radii satisfy the system
	\begin{equation*}
		\begin{split}
			\dot{r}_A &= -r_A + r_A^3 - \tilde{\gamma}_A r_A r_B^2, \\
			\dot{r}_B &= \dfrac{1}{\tilde{c}}\left(r_B - r_B^3 + \tilde{\gamma}_B r_B r_A^2\right),
		\end{split}
	\end{equation*}
	see also \eqref{eq:radii-dynamics-rescaled}, which we analyse using phase plane analysis. We identify parameter conditions for the existence of four equilibrium points corresponding to the time-periodic solutions in the amplitude equations: the trivial equilibrium point $T = (0,0)$, the semi-trivial equilibrium points $ST_A$ and $ST_B$ with $r_B = 0$ and $r_A = 0$, respectively, as well as the nontrivial equilibrium point $NT$. We then analyse the existence of heteroclinic orbits between the equilibrium points using stability properties of the equilibrium points to identify different parameter regimes with qualitatively different phase plane dynamics, which we study numerically, see Section \ref{sec:fronts} and analytically, see Section \ref{sec:proof-of-het-orbits}.
	
	To lift the existence from the $(r_A,r_B)$-system to the singularly perturbed travelling wave ODE, we use geometric singular perturbation theory. We obtain the persistence of heteroclinic orbits, which are structurally stable in the $(r_A,r_B)$-system. That is, after potentially restricting to an appropriate invariant subspace, specifically $\{r_A = 0\}$ or $\{r_B = 0\}$, the dimension of the unstable manifold of the left equilibrium point $n_u$ and the dimension of the stable manifold of the right equilibrium point $n_s$ satisfy $n_u + n_s > 2$. In this case, we prove that the corresponding centre-unstable and centre-stable manifolds in the full $(r_A, r_B, \phi_A, \phi_B)$-phase plane intersect transversally on the critical manifold, which persists under perturbation, see Theorem \ref{thm:persistence-heteroclinic-orbits}. Here, we also exploit that centre directions correspond to the invariance of solutions to the amplitude equations \eqref{eq:final-amplitude-eq} under phase shift. Together with the rigorous approximation results for the amplitude system, we thus obtain the existence of solutions to the pattern-forming system \eqref{eq:toy-model}, which describe pattern interfaces, that is, a moving front connecting two different patterns (including the trivial state), see Figures \ref{fig:T-to-trivial}--\ref{fig:TH-to-superpos}.
	
	\paragraph{Spatially periodic solutions}
	
	Since we can only guarantee that the amplitude equations give a good approximation to the full pattern-forming system \eqref{eq:toy-model} on a large, but finite, time interval, we consider the existence of global-in-time, spatially periodic solutions. Using centre manifold theory and a normal form transformation, we recover to leading order the amplitude equations \eqref{eq:final-amplitude-eq} for spatially constant $\Aav$, $\Bav$ as the reduced system on the centre manifold. These possess nontrivial, time-periodic solutions in appropriate parameter regimes, see Propositions \ref{prop:semi-trivial-solutions} and \ref{prop:fully-nontrivial-solutions}, which correspond to spatially periodic, slowly travelling wavetrains in the pattern-forming system \eqref{eq:toy-model}.
	
	The main challenge to obtain a persistence result for these time-periodic solutions on the centre manifold are fast oscillating terms $\ee^{\ii \varepsilon^{-2} n c_d T}$, which again arise from interactions of the amplitude modulations $A$ and $B$ since the phase velocities are different. Using a variation of singular perturbation theory due to Hale \cite[Chap.~VII]{hale1969}, see also Appendix \ref{app:persistence}, we still obtain invariant manifolds parametrised by the phase of the amplitude modulations and the fast angle appearing in the fast oscillating terms, see Theorems \ref{thm:pure-patterns-global} and \ref{thm:superposition-pattern-global}. This gives the existence of global, spatially periodic solutions to the pattern-forming system \eqref{eq:toy-model}. However, we cannot, in general, make statements about the temporal behaviour of the solutions. In particular, we cannot guarantee that the solutions are time-periodic, see Remark \ref{rem:complicated-dynamics}.
	
	\subsection{Related results}
	
	 The main topic of this paper is the unfolding of an instability with higher codimension in a PDE system using modulation theory and invariant manifolds. We point out that similar results have been obtained in the case of two Turing instabilities with a spatial 1:2 resonance in \cite{gauss2021}, where the system of amplitude equations consists of two coupled real Ginzburg–Landau equations. Additionally, similar results have been obtained in the case of a simultaneous Turing and long-wave Hopf instability in a general reaction-diffusion system, see \cite{schneider2022}. Here, a long-wave Hopf instability is characterised by an eigenvalue curve $\lambda(k,\mu)$, for which the real part vanishes quadratically at $k = 0$ and a critical value $\mu_c$, but the imaginary part is non-zero. Note that in real-valued reaction-diffusion systems, these curves appear as complex conjugated pairs. The authors then derive a system of amplitude equations consisting of a real and a complex Ginzburg–Landau equation and obtain a rigorous approximation result.
	
	In addition to the recent approaches using amplitude equations to describe nonlinear waves in PDEs on spatially extended domains, there is an extensive list of results on solutions on bounded domains with periodic boundary conditions arising from interacting instabilities. The case where both instabilities have the same wave number is discussed in \cite{guckenheimer1986}, whereas the case of a 1:2 resonance is considered in \cite{dangelmayr1986,armbruster1988,porter2001,porter2005}, see also \cite{porter2000} for a 1:3 resonance. The typical approach here is to consider the dynamics of the leading order equation on the centre manifold. Then, one can study the \emph{temporal} dynamics of the amplitude modulations of the critical Fourier modes, similar to Section \ref{sec:periodic-solutions}. In contrast, Section \ref{sec:fronts} considers the \emph{spatio-temporal} dynamics in a co-moving frame. We also point out that the instabilities considered are typically steady-state, i.e. Turing, instabilities. In particular, the temporal dynamics of spatially periodic solutions close to a Turing and Turing–Hopf instability seems to be open and, as Section \ref{sec:periodic-solutions} shows, is indeed nontrivial due to the presence of highly oscillatory higher-order terms.
	
	Finally, the front solutions connecting different patterns, which are constructed in this paper, are similar to modulating fronts. These are solutions of the form $u(t,x) = U(x-ct,x-c_p t)$, which are periodic with respect to its second argument and satisfy $\lim_{\xi \rightarrow \pm \infty} U(\xi,p) = u_\pm(p)$, where $u_\pm(p)$ are periodic solutions. Therefore, modulating fronts model a spatial transition between two periodic states. They have been constructed for different pattern-forming systems which destabilise through a single instability such as the cubic Swift–Hohenberg equation \cite{eckmann1991}, the Taylor–Couette problem \cite{haragus-courcelle1999}, a nonlocal reaction-diffusion equation close to a Turing instability \cite{faye2015}, as well as Swift–Hohenberg-type equations with an additional conservation law close to a Turing \cite{hilder2020} and a Turing–Hopf instability \cite{hilder2022}. The construction of modulating fronts is typically based on a spatial dynamics and centre manifold approach. Notably, this yields global solutions to the pattern-forming systems as opposed to solutions which only exist for a large but finite interval, which are constructed in this paper. However, we point out that there are substantial difficulties in applying this approach to the situation of a resonant Turing and Turing–Hopf instability as discussed in Section \ref{sec:discussion}.
	
	\subsection{Outline}
	
	The paper is structured as follows: In Section \ref{sec:amplitude-equations} we prove the validity of a formal derivation of the amplitude equations close to a 1:1 resonant Turing and Turing–Hopf instability. Since the justification proof follows the standard strategy, we only provide an overview of the proof in Appendix \ref{app:justification}.  Following this, we discuss the dynamics of the amplitude equations with a focus on space-time periodic solutions in Section \ref{sec:time-periodic-solutions} and fast-moving fronts connecting different space-time periodic solutions in Section \ref{sec:fronts}. This section contains a collection of numerically observed heteroclinic orbits, large classes of these orbits are then established rigorously in Section \ref{sec:proof-of-het-orbits}. Afterwards, we then translate the solutions constructed for the amplitude equations back to the full pattern-forming system \eqref{eq:toy-model} in Section \ref{sec:modulated-waves-in-full-system}. Finally, Section \ref{sec:periodic-solutions} contains the construction of global, spatially periodic solutions using centre manifold theory. Here, we use a variant of singular perturbation theory, which is outlined in Appendix \ref{app:persistence}. Concluding the paper, we discuss related open questions in Section \ref{sec:discussion} and give the explicit expressions for the coefficients in the amplitude equation in Appendix \ref{app:coefficients}.
	
	\section{Amplitude equations and their justification}\label{sec:amplitude-equations}
	
	We first start our analysis by a formal derivation of the amplitude equations governing the dynamics close to the onset of instability. Following the linear stability analysis, we make the ansatz
	\begin{equation}\label{eq:ansatz-amplitude-eq}
		\begin{pmatrix} u \\ v \end{pmatrix}(t,x) = \varepsilon A(T,X_1) \ee^{\ii x} \begin{pmatrix}
		    1 \\ 0
		\end{pmatrix} + \varepsilon B(T,X_2) \ee^{\ii (x-c_p t)} \begin{pmatrix}
		    0 \\ 1
		\end{pmatrix} + c.c + \Psi_\mathrm{hot} =: \Psi_\mathrm{GL}(A,B) + \Psi_\mathrm{hot},
	\end{equation}
	where $c.c.$ denotes the complex conjugated terms, $T = \varepsilon^2 t$ is the slow time scale, $X_1 = \varepsilon x$ is the slow spatial scale, and $X_2 = \varepsilon(x-c_g t)$ is the slow co-moving frame variable moving with the group velocity $c_g$ to be determined later. Additionally, $\Psi_\mathrm{hot}$ denotes the higher order residual terms, which are located at Fourier modes, which are not in the kernel of $L_0$, the linearisation of \eqref{eq:toy-model} about the trivial state. The residual terms are used to equilibrate lower-order terms in $\varepsilon$ to zero and we make the ansatz
	\begin{equation*}
        \begin{split}
    		\Psi_\mathrm{hot} &= \varepsilon^2 \left(\dfrac{1}{2}A_0(T,X_1,t) + A_2(T,X_1,t) \ee^{2\ii x} + c.c.\right) \begin{pmatrix}
    		    1 \\ 0
    		\end{pmatrix} \\
            &\qquad+ \varepsilon^2 \left(\dfrac{1}{2}B_0(T,X_2,t) + B_2(T,X_2,t) \ee^{\ii(x-c_p t)} + c.c.\right) \begin{pmatrix}
		    0 \\ 1
		\end{pmatrix}.
        \end{split}
	\end{equation*}
	We point out that compared to the usual ansatz for a pure Turing or Turing-Hopf instability, the amplitude modulations $A_0$, $A_2$, $B_0$ and $B_2$ not only depend on the slow temporal and spatial scales $T$, $X_1$ and $X_2$ but also on the fast time scale $t$. This is done to compensate for fast oscillating terms, which are induced by the fact that Turing and Turing-Hopf modes have different phase velocities, see also \cite{schneider1997}.
	
	Inserting the ansatz \eqref{eq:ansatz-amplitude-eq} into the system \eqref{eq:toy-model} and equating different powers of $\varepsilon$ to zero, we find at $\varepsilon^1$-balance that $c_p = c_d$, which determines the linear phase velocity. Calculating the balance at $\varepsilon^2$, we additionally need to separate the different Fourier modes. At $\varepsilon^2 \ee^{\ii x}$ we find that $c_g = 3c_d$, which determines the linear group velocity. At $\varepsilon^2 \ee^{0\ii x}$ we find that $A_0$ and $B_0$ satisfy the differential equations
	\begin{equation*}
		\begin{split}
			\partial_t A_0 &= - A_0 + \ee^{\ii c_p t} f_{11} A \bar{B} + 2(f_{20} |A|^2 + f_{02} |B|^2) + \ee^{-\ii c_p t} f_{11} \bar{A} B, \\
			\partial_t B_0 &= - B_0 + \ee^{\ii c_p t} g_{11} A \bar{B} + 2(g_{20} |A|^2 + g_{02} |B|^2) + \ee^{-\ii c_p t} g_{11} \bar{A} B.
		\end{split}
	\end{equation*}
	Solving this explicitly yields that 
	\begin{equation}\label{eq:correction-terms-0}
		\begin{split}
			A_0 &= \dfrac{\ee^{\ii c_p t}}{1 + ic_p} f_{11}  A \bar{B} + 2(f_{20} |A|^2 + f_{02} |B|^2) + \dfrac{\ee^{-\ii c_p t}}{1 - ic_p} f_{11} \bar{A} B, \\
			B_0 &= \dfrac{\ee^{\ii c_p t}}{1 + ic_p} g_{11}  A \bar{B} + 2(g_{20} |A|^2 + g_{02} |B|^2) + \dfrac{\ee^{-\ii c_p t}}{1 - ic_p} g_{11} \bar{A} B
		\end{split}
	\end{equation}
	Similarly, balancing at $\varepsilon^2 \ee^{2\ii x}$, we obtain a differential equation for $A_2$ and $B_2$, respectively. Solving this then leads to 
	\begin{equation}\label{eq:correction-terms-2}
		\begin{split}
			A_2 &= \dfrac{1}{9} f_{20} A^2 + \dfrac{\ee^{-\ii c_p t}}{9-i c_p} f_{11} AB + \dfrac{\ee^{-2\ii c_p t}}{9-2ic_p} f_{02} B^2, \\
			B_2 &= \dfrac{\ee^{2\ii c_p t}}{9 + 8ic_p} g_{20} A^2 + \dfrac{\ee^{\ii c_p t}}{9 + 7 ic_p} g_{11} AB + \dfrac{1}{9 + 6ic_p} g_{02} B^2.
		\end{split}
	\end{equation}
	
	Finally, equilibrating at $\varepsilon^3 \ee^{\ii x}$ we that find the formal amplitude equations for the system \eqref{eq:toy-model} are, to leading order, given by
	\begin{equation}\label{eq:full-amplitude-eq}
		\begin{split}
			\partial_T A &= 4\partial_{X_1}^2 A + \alpha_u A + \gamma_1 A \abs{A}^2 + \gamma_2 A \abs{B}^2 +  \gamma_3 A^2 \bar{B}\ee^{\ii c_p t} + (\gamma_4 B \abs{A}^2 + \gamma_5 B\abs{B}^2) \ee^{-\ii c_p t} \\
            &\qquad + \gamma_6 B^2 \bar{A} \ee^{-2\ii c_p t}, \\
			\partial_T B &= (4 + 3ic_v) \partial_{X_2}^2 B + \alpha_v B + \gamma_7 B \abs{A}^2 + \gamma_8 B \abs{B}^2 + (\gamma_9 A\abs{A}^2 + \gamma_{10} A \abs{B}^2) \ee^{\ii c_p t} \\
            &\qquad+ \gamma_{11} A^2 \bar{B} \ee^{2\ii c_p t} + \gamma_{12} B^2 \bar{A} \ee^{-\ii c_p t}
		\end{split}
	\end{equation}
	with coefficients $\gamma_{j} \in \C$ for $1 \leq j \leq 12$. The local existence and uniqueness of solutions to \eqref{eq:full-amplitude-eq} follows from \cite[Lemma 2.5]{schneider1997}. In particular, the solutions can be bounded uniformly in $\eps \in (0,1)$.
	
	We point out that, unlike typical amplitude equations appearing close to pure Turing- or Turing–Hopf instability, the amplitude equations \eqref{eq:full-amplitude-eq} still depend on the fast time $t$. This dependence is explicit via the appearance of the fast oscillatory exponential functions $\ee^{\beta \ii c_p t}$ with $\beta = \pm 1, \pm 2$. However, the equations also depends implicitly on $t$ through the interaction terms of $A$ and $B$, which leads to an interaction of different co-moving frames. This can for example be seen by writing 
	\begin{equation*}
		(A\abs{B}^2)(T,X_1,t) = A(T,X_1) \abs{B(T,X_1 + \eps c_g t)}^2.
	\end{equation*}
	In particular, shifting both equations to the same co-moving frame generates a singular advection term in one of the equations, see e.g.~\eqref{eq:final-amplitude-eq}.
	
	One should note that the fast oscillating terms appear due to the difference in phase velocity between the different instabilities, while the interaction of the two different co-moving frames $X_1$ and $X_2$ is due to different group velocities. As we will see, the fast oscillating terms can be moved to higher orders of $\eps$. However, this is not the case for the interactions of $A$ and $B$, unless both $A$ and $B$ are sufficiently localised \cite[Remark 10.7.6]{schneider2017a}.

    We now turn to the rigorous justification of the amplitude equations \eqref{eq:full-amplitude-eq}. For this, we introduce $H^\ell_{l,u}$, $\ell \geq 0$, the space of uniformly local Sobolev functions, equipped with the norm
    \begin{equation*}
        \|w\|_{H^{\ell}_{l,u}} := \sup_{x \in \R} \|w\|_{H^\ell((x,x+1))},
    \end{equation*}
    where $H^\ell((x,x+1))$ denotes the usual Sobolev space on the Interval $(x,x+1)$ with the usual Sobolev norm, see also \cite[Sec.~8.3.1]{schneider2017a}. In this setting, the amplitude equations \eqref{eq:full-amplitude-eq} can also be rigorously justified by following the strategy in \cite{schneider1994a,schneider1997}, see also \cite[Chap.~10]{schneider2017a}. For completeness, we give an outline of the proof in Appendix \ref{app:justification}. Therefore, the following result holds.
	
	\begin{theorem}\label{thm:approximation-full-amplitude-eq}
		Let $\eps \in (0,1) \mapsto (A,B) \in C([0,T_0], (H^5_{l,u})^2)$ be a family of solutions to \eqref{eq:full-amplitude-eq}, which satisfies $\sup_{\eps \in (0,1)}\sup_{T \in [0,T_0]} \norm{(A,B)}_{(H^5_{l,u})^2} < \infty$. Then there exist $\eps_0 > 0$ and $C > 0$ such that for all $\eps \in (0,\eps_0)$ there is a solution $(u,v)(t,x,\eps)$ to \eqref{eq:toy-model} such that
		\begin{equation*}
			\sup_{t \in [0,T_0/\eps^2]} \norm{(u,v) - \Psi_\mathrm{GL}(A,B)}_{(H^1_{l,u})^2} < C \eps^2.
		\end{equation*}
	\end{theorem}

    \begin{remark}
        Since $H^1_{l,u}$ embeds into the space of bounded and continuous functions, we also have that the error is small in the supremum norm, i.e.
        \begin{equation*}
			\sup_{t \in [0,T_0/\eps^2]} \norm{(u,v) - \Psi_\mathrm{GL}(A,B)}_{\infty} < C \eps^2.
		\end{equation*}
    \end{remark}
	
	Next, we show that the terms of \eqref{eq:full-amplitude-eq}, which contain the fast oscillatory term $\ee^{\ii c_p T/\eps^2}$ can be moved to higher orders of $\eps$. This is a consequence of the following result.
	
	\begin{lemma}\label{lem:averaging-amplitude-equations}
		Let $\eps \in (0,1) \mapsto (\Aav,\Bav) \in C([0,T_0],(H^7_{l,u})^2)$ be a family of solutions to the system
		\begin{equation}\label{eq:amplitude-eq-wohot}
			\begin{split}
				\partial_T \Aav &= 4\partial_{X_1}^2 \Aav + \alpha_u \Aav + \gamma_1 \Aav |\Aav|^2 + \gamma_2 \Aav |\Bav|^2, \\
				\partial_T \Bav &= (4+3ic_v) \partial_{X_2}^2 \Bav + \alpha_v \Bav + \gamma_7 \Bav |\Aav|^2 + \gamma_8 \Bav |\Bav|^2,
			\end{split}
		\end{equation}
		which satisfies $\sup_{\eps \in (0,1)}\sup_{T \in [0,T_0]} \norm{(\Aav,\Bav))}_{(H^8_{l,u})^2} \leq C_1 < \infty$. Then there exist $\eps_0 > 0$ and $C > 0$ such that for all $\eps \in (0,\eps_0)$ exist solutions $(A,B)$ to \eqref{eq:full-amplitude-eq} such that
		\begin{equation*}
			\sup_{T \in [0,T_0]} \norm{(\Aav,\Bav) - (A,B)}_{(H^5_{l,u})^2} \leq C \eps
		\end{equation*}
	\end{lemma}
	\begin{proof}
		We follow the proof of \cite[Theorem 3.1]{schneider1997}. Let $Z = (H^5_{l,u})^2$ and note that all nonlinear terms of \eqref{eq:full-amplitude-eq} are cubic. Therefore, we can write the nonlinearity as a symmetric trilinear mapping $\Ncal$ which satisfies $\norm{\Ncal(u,v,w)}_Z \leq C \norm{u}_Z \norm{v}_Z \norm{w}_Z$ for a constant $C > 0$. Then, we make the ansatz $(A,B) = \Acalav + \eps \Rcal$ with $\Acalav = (\Aav,\Bav)$. Plugging this into the system \eqref{eq:full-amplitude-eq} yields that the error $\Rcal$ satisfies the system
		\begin{equation}\label{eq:error-eq}
			\partial_T\Rcal = \Lcal \Rcal + 3 \Ncal(\Acalav,\Acalav,\Rcal) + 3 \eps \Ncal(\Acalav,\Rcal,\Rcal) + \eps^2 \Ncal(\Rcal,\Rcal,\Rcal) + \eps^{-1} \mathrm{Res}(\Acalav)
		\end{equation}
		with the linear operator $\Lcal \Rcal = (4\partial_{X_1}^2 \Rcal_1 + \alpha_u \Rcal_1, (4+3ic_v) \partial_{X_2}^2 \Rcal_2 + \alpha_v \Rcal_2)^T$. Additionally, $\mathrm{Res}(\Acalav)$ is the residual given by
		\begin{equation*}
			\mathrm{Res}(\Acalav) = \begin{pmatrix*}
										\gamma_3 \Aav^2 \bar{\Bav}\ee^{\ii c_p T/\eps^2} + (\gamma_4 \Bav \abs{\Aav}^2 + \gamma_5 \Bav\abs{\Bav}^2) \ee^{-\ii c_p T/\eps^2} + \gamma_6 B^2 \bar{\Aav} \ee^{-2\ii c_p T/\eps^2} \\
										(\gamma_9 A\abs{\Aav}^2 + \gamma_{10} A \abs{\Bav}^2) \ee^{\ii c_p T/\eps^2} + \gamma_{11} \Aav^2 \bar{\Bav} \ee^{2\ii c_p T/\eps^2} + \gamma_{12} \Bav^2 \bar{\Aav} \ee^{-\ii c_p T/\eps^2}
									\end{pmatrix*}.
		\end{equation*}
		Since \eqref{eq:full-amplitude-eq} has a unique solution on $[0,T_0]$ due to \cite[Lemma 2.5]{schneider1997}, and $(\Aav,\Bav)$ exists by assumption, we find that \eqref{eq:error-eq} has a unique solution on $[0,T_0]$. Then, using that $\Lcal$ generates an analytic semigroup in $(L^2_{l,u})^2$, which satisfies $\norm{\ee^{T\Lcal}}_Z \leq C_2$ for all $T \in [0,T_0]$, we can write the equation for the error $\Rcal$ in integral form as
		\begin{equation}\label{eq:integral-eq-error}
			\begin{split}
				\Rcal(T) &= \int_0^T \ee^{(T-\tau)\Lcal} \left(3 \Ncal(\Acalav(\tau),\Acalav(\tau),\Rcal(\tau)) + 3 \eps \Ncal(\Acalav(\tau), \Rcal(\tau),\Rcal(\tau))\right) \dd \tau \\
                &\quad + \int_0^T \ee^{(T-\tau)\Lcal} \eps^2 \Ncal(\Rcal(\tau),\Rcal(\tau),\Rcal(\tau)) \,\mathrm{d}\tau + \int_0^T \eps^{-1} \ee^{(T-\tau)\Lcal} \mathrm{Res}(\Acalav)(\tau) \,\mathrm{d}\tau.
			\end{split}
		\end{equation}
		Integrating by parts shows that the last term is bounded in $Z$ uniformly for all $T \in [0,T_0]$ and $\eps \in (0,1)$, see \cite[Lemma 3.2]{schneider1997}. Here, we use that $\Acalav(T) \in (H^7_{l,u})^2$ for all $T \in [0,T_0]$. 
		
		We now show that there exist $\eps_0 > 0$ and $D > 0$ such that the solution to the error equation \eqref{eq:error-eq} satisfies $\sup_{T \in [0,T_0]} \norm{\Rcal}_Z \leq D$ for all $\eps \in (0,\eps_0)$ with a constant $D$ to independent of $\eps$. Since $\Rcal(0) = 0$, we can assume that $\norm{\Rcal(T)}_Z \leq D$ at least for $T$ sufficiently small. With this, we estimate
		\begin{equation*}
			\norm{\Tcal(\Rcal)}_Z \leq \left(\int_0^T 3 C_2 C C_1^2 \norm{\Rcal(\tau)}_Z \,\mathrm{d}\tau\right) + T_0 C_2 C (\eps C_1 D^2 + \eps^2 D^3) + C_\mathrm{Res}.
		\end{equation*}
		By choosing $\eps_0$ sufficiently small we can guarantee that $\eps C_1 D^2 + \eps^2 D^3 \leq 1$ holds for all $\eps \in (0,\eps_0)$. Then, using Gronwall's inequality gives that
		\begin{equation}
			\norm{\Rcal(T)}_Z \leq (T_0 C_2 C  + C_\mathrm{Res}) \ee^{T_0 C_2 C C_2^2} =: D,
		\end{equation}
		which defines the constant $D$.	In particular, this shows by a bootstrapping argument that $\sup_{T \in [0,T_0]}\norm{\Rcal(T)} \leq D$. This completes the proof.
	\end{proof}
	
	To make the fast shifting dynamics in the nonlinearities of \eqref{eq:amplitude-eq-wohot} more explicit we write $\Bav$ as a function of $T$ and $X_1$ instead of $X_2$. For this we recall that $X_2 = X_1 + c_g T/\eps$. This yields the system
	\begin{equation}\label{eq:final-amplitude-eq}
		\begin{split}
			\partial_T \Aav &= 4\partial_{X_1}^2 \Aav + \alpha_u \Aav + \gamma_1 \Aav |\Aav|^2 + \gamma_2 \Aav |\Bav|^2, \\
			\partial_T \Bav &= (4+3ic_v) \partial_{X_1}^2 \Bav - \frac{c_g}{\eps} \partial_{X_1} \Bav + \alpha_v \Bav + \gamma_7 \Bav |\Aav|^2 + \gamma_8 \Bav |\Bav|^2,
		\end{split}
	\end{equation}
	where now the nonlinearities have no fast shifting dynamics, but the equation for $\Bav$ has a singular advection term. By using that \eqref{eq:final-amplitude-eq} is still equivalent to \eqref{eq:amplitude-eq-wohot} and combining the results of Theorem \ref{thm:approximation-full-amplitude-eq} and Lemma \ref{lem:averaging-amplitude-equations} we obtain the following approximation result.
	
	\begin{theorem}\label{thm:full-approx-result}
		Let $\eps \in (0,1) \mapsto (\Aav,\Bav) \in C([0,T_0],(H^7_{l,u})^2)$ be a family of solutions to \eqref{eq:final-amplitude-eq}, which satisfies 
		\begin{equation*}
			\sup_{\eps \in (0,1)} \sup_{T \in [0,T_0]} \norm{(\Aav,\Bav)}_{(H^7_{l,u})^2} < \infty.
		\end{equation*} 
		Then there exist $\eps_0 > 0$ and $C > 0$ such that for all $\eps \in (0,\eps_0)$ there is a solution $(u,v)(t,x,\eps)$ to \eqref{eq:toy-model} such that
		\begin{equation*}
			\sup_{t \in [0,T_0/\eps^2]} \norm{(u,v) - \Psi_\mathrm{GL}(\Aav,\Bav)}_{(H^1_{l,u})^2} \leq C\eps^2.
		\end{equation*}
	\end{theorem}

	\section{Modulation dynamics -- time periodic solutions}\label{sec:time-periodic-solutions}
	
	After having shown that the dynamics of the model \eqref{eq:toy-model} is well-approximated on a long time interval by the dynamics of \eqref{eq:final-amplitude-eq}, we now discuss the dynamics of the amplitude system. In this section, we consider the existence of space-time periodic solutions to \eqref{eq:final-amplitude-eq}.
	
	\subsection{Time-periodic solutions}
	
	First, we look for time-periodic solutions to \eqref{eq:final-amplitude-eq}. For this we make the ansatz $(\Aav,\Bav)(T,X) = (r_A \ee^{\ii \omega_A T}, r_B \ee^{\ii \omega_B T})$ with $r_A, r_B > 0$ and $\omega_A, \omega_B \in \R$ to obtain the system
	\begin{equation*}
		\begin{split}
			0 &= (\alpha_u - i\omega_A) r_A + r_A (\gamma_1 r_A^2 + \gamma_2 r_B^2), \\
			0 &= (\alpha_v - i\omega_B) r_B + r_B (\gamma_7 r_A^2 + \gamma_8 r_B^2).
		\end{split}
	\end{equation*}
	We then obtain separate equations for $r_A, r_B$ and $\omega_A, \omega_B$ by equating the real and imaginary parts of the system separately to zero. This yields the system
	\begin{equation}\label{eq:system-time-periodic-sols-amplitude-equation}
		\begin{split}
			0 &= r_A (\alpha_u + \gamma_{1,r} r_A^2 + \gamma_{2,r} r_B^2), \\
			0 &= r_B (\alpha_v + \gamma_{7,r} r_A^2 + \gamma_{8,r} r_B^2), \\
			0 &= r_A (-\omega_A + \gamma_{1,i} r_A^2 + \gamma_{2,i} r_B^2), \\
			0 &= r_B (-\omega_B + \gamma_{7,i} r_A^2 + \gamma_{8,i} r_B^2),
		\end{split}
	\end{equation} 
	where $\gamma_j = \gamma_{j,r} + i\gamma_{j,i}$ with $\gamma_{j,r}, \gamma_{j,i} \in \R$ for $j = 1,2,7,8$. Note that since the nonlinear coefficients $\gamma_1, \gamma_2, \gamma_7, \gamma_8 \in \C$ are, in general, genuinely complex-valued, see also Appendix \ref{app:coefficients}, we cannot expect to find time-independent solutions with $\omega_A = \omega_B = 0$, unless additional structure is imposed on the nonlinearities $f$ and $g$ in \eqref{eq:toy-model}.
	
	\begin{remark}\label{rem:time-periodic-solutions-correspond-to-speed-correction}
		Recalling the ansatz $\Psi_\mathrm{GL}$ and the fact that $\ee^{\ii \omega T} = \ee^{\ii\omega \varepsilon^2 t}$, we point out that time-periodic solutions of the amplitude equations \eqref{eq:final-amplitude-eq} correspond to travelling wavetrains of the full system \eqref{eq:toy-model}, which have a higher-order corrected phase velocity compared to the linear phase velocity.
	\end{remark}
	
	The system \eqref{eq:system-time-periodic-sols-amplitude-equation} has three types of solutions. First, it has a trivial solution at $(r_A,r_B,\omega_A,\omega_B) = (0,0,0,0)$. Second, if either $r_A = 0$ or $r_B = 0$, but not both, we call a solution semi-trivial. Finally, if neither $r_A = 0$ nor $r_B = 0$, we call a solution nontrivial. We now derive conditions on the parameters, which guarantee the existence for semi-trivial and fully nontrivial solutions to \eqref{eq:system-time-periodic-sols-amplitude-equation}. For semi-trivial solutions, we obtain the following result after a straightforward computation.
	
	\begin{proposition}\label{prop:semi-trivial-solutions}
		The system \eqref{eq:system-time-periodic-sols-amplitude-equation} has a semi-trivial solution $(r_A, 0, \omega_A, 0)$ with $r_A > 0$ if and only if $\alpha_u \gamma_{1,r} < 0$. The solution is then given by
		\begin{equation*}
			r_A^2 = - \dfrac{\alpha_u}{\gamma_{1,r}}, \text{ and } \omega_A = \dfrac{\alpha_u \gamma_{1,i}}{\gamma_{1,r}}.
		\end{equation*} 
		Similarly, the system \eqref{eq:system-time-periodic-sols-amplitude-equation} has a semi-trivial solution $(0, r_B, 0, \omega_B)$ with $r_B > 0$ if and only if $\alpha_v \gamma_{8,r} < 0$.
	\end{proposition}
	
	We now turn to fully nontrivial solutions, where both $r_A \neq 0$ and $r_B \neq 0$. If such a solution exists, $r_A^2$ and $r_B^2$ must satisfy the linear system
	\begin{equation*}
		\begin{pmatrix}
			\gamma_{1,r} & \gamma_{2,r} \\
			\gamma_{7,r} & \gamma_{8,r}
		\end{pmatrix} \begin{pmatrix}
			r_A^2 \\ r_B^2
		\end{pmatrix} = \begin{pmatrix}
			-\alpha_u \\ -\alpha_v
		\end{pmatrix},
	\end{equation*}
	which has a nontrivial solution for every $\alpha_u, \alpha_v \in \R$ if and only if $d := \gamma_{1,r} \gamma_{8,r} - \gamma_{2,r} \gamma_{7,r} \neq 0$ and 
	\begin{equation*}
		\begin{split}
			r_A^2 = -\dfrac{\gamma_{8,r} \alpha_u - \gamma_{2,r}\alpha_v}{d} > 0, \\
			r_B^2 = -\dfrac{\gamma_{1,r} \alpha_v - \gamma_{7,r}\alpha_u}{d} > 0.
		\end{split}
	\end{equation*}
	Then, for given $r_A$ and $r_B$, the wave numbers $\omega_A$ and $\omega_B$ are then given by
	\begin{equation*}
		\begin{pmatrix}
			\omega_A \\ \omega_B
		\end{pmatrix} = \begin{pmatrix}
			\gamma_{1,i} & \gamma_{2,i} \\
			\gamma_{7,i} & \gamma_{8,i}
		\end{pmatrix}
		\begin{pmatrix}
			r_A^2 \\ r_B^2
		\end{pmatrix} = -\dfrac{1}{\gamma_{1,r} \gamma_{8,r} - \gamma_{2,r} \gamma_{7,r}} \begin{pmatrix}
			\gamma_{1,i} & \gamma_{2,i} \\
			\gamma_{7,i} & \gamma_{8,i}
		\end{pmatrix} \begin{pmatrix}
			\gamma_{8,r} & -\gamma_{2,r} \\
			-\gamma_{7,r} & \gamma_{1,r}
		\end{pmatrix} \begin{pmatrix}
			\alpha_u \\ \alpha_v
		\end{pmatrix}.
	\end{equation*}
	This yields the following proposition. 
	
	\begin{proposition}\label{prop:fully-nontrivial-solutions}
		The system \eqref{eq:system-time-periodic-sols-amplitude-equation} has a fully nontrivial solution if $d \neq 0$ and 
		\begin{equation*}
			\begin{split}
				\dfrac{\gamma_{8,r} \alpha_u - \gamma_{2,r}\alpha_v}{d} < 0 \text{ and } \dfrac{\gamma_{1,r} \alpha_v - \gamma_{7,r}\alpha_u}{d} < 0.
			\end{split}
		\end{equation*}
	\end{proposition}
	
	\begin{remark}
		Note that the conditions for the existence of semi-trivial and fully nontrivial solutions of \eqref{eq:system-time-periodic-sols-amplitude-equation} only depend on the real parts of the nonlinear coefficient, but not their imaginary part.
	\end{remark}
	
	\begin{remark}
		Due to the invariance of \eqref{eq:final-amplitude-eq} under rigid rotations $(\Aav, \Bav) \mapsto (\Aav \ee^{\ii\phi_A}, \Bav \ee^{\ii\phi_B})$, time-periodic solutions always appear as a one-parameter family $(r_A \ee^{\ii(\omega_A T + \phi_A)}, r_B \ee^{\ii(\omega_B T + \phi_B)})$ with $\phi_A, \phi_B \in [0,2\pi)$.
	\end{remark}

	\subsection{Space-time periodic waves}
	
	Next, we consider space-time periodic solutions to the amplitude equations \eqref{eq:final-amplitude-eq}, that is, we make the ansatz
	\begin{equation}\label{eq:ansatz-space-time-periodic}
		(\Aav, \Bav)(T,X) = (r_A \ee^{\ii(k_A X + \omega_A T)}, r_B \ee^{\ii (k_B X + \omega_B T)})
	\end{equation}
	with amplitudes $r_A, r_B \geq 0$, spatial wave numbers $k_A, k_B \in \R$ and temporal wave numbers $\omega_A, \omega_B \in \R$. Notice that in the case $k_A = k_B = 0$, we obtain spatially constant and purely time-periodic solutions considered in the previous section.
	
	Inserting the ansatz \eqref{eq:ansatz-space-time-periodic} into the amplitude equations \eqref{eq:final-amplitude-eq} we find that
	\begin{equation*}
		\begin{split}
			0 &= r_A \left(i \omega_A - 4 k_A^2 + \alpha_u + \gamma_1 r_A^2 + \gamma_2 r_B^2\right), \\
			0 &= r_B \left(-i\omega_B - (4 + 3ic_v) k_B^2 - \dfrac{c_g}{\varepsilon} i k_B + \alpha_v + \gamma_7 r_A^2 + \gamma_8 r_B^2\right)
		\end{split}
	\end{equation*}
	To compensate for the singular term in the second equation, we restrict to solutions with a long spatial wave-length. That is, we assume that $k_j = \varepsilon \tilde{k}_j$ with $\tilde{k}_j \in \R$ for $j \in\{ A,B\}$. Additionally splitting into real and imaginary parts, this yields
	\begin{equation}\label{eq:system-spatio-temporal-solutions}
		\begin{split}
			0 &= r_A \left(\alpha_u + \gamma_{1,r} r_A^2 + \gamma_{2,r} r_B^2\right) - 4 \varepsilon^2  r_A \tilde{k}_A^2, \\
			0 &= r_A \left(- \omega_A+ \gamma_{1,i} r_A^2 + \gamma_{2,i} r_B^2\right), \\
			0 &= r_B \left(\alpha_v + \gamma_{7,r} r_A^2 + \gamma_{8,r} r_B^2\right) - 4 \varepsilon^2 r_B \tilde{k}_B^2, \\
			0 &= r_B \left(- \omega_B - c_g \tilde{k}_B + \gamma_{7,i} r_A^2 + \gamma_{8,i} r_B^2\right) - 3 \varepsilon^2 c_v r_B \tilde{k}_B^2.
		\end{split}
	\end{equation}
	Using that all equation are homogeneous in either $r_A$ or $r_B$ and recalling the calculations in the purely time-periodic case, we obtain the following result.
	
	\begin{proposition}\label{prop:space-time-periodic-waves}
		For every bounded set of spatial wave numbers $(\tilde{k}_A, \tilde{k}_B)$ exists an $\varepsilon_0 > 0$ such that for all $\varepsilon \in (0,\varepsilon_0)$ there are families of solutions $(r_A,\omega_A, r_B, \omega_B)$ parametrised by $(\tilde{k}_A, \tilde{k}_B)$. Specifically, we obtain
		\begin{itemize}
			\item the trivial solution $r_A = r_B = 0$;
			\item two families of semi-trivial solutions with either $r_A = 0$ or $r_B = 0$ provided that the conditions of Proposition \ref{prop:semi-trivial-solutions} are satisfied;
			\item a family of fully non-trivial solutions with $r_A, r_B \neq 0$ if the conditions of Proposition \ref{prop:fully-nontrivial-solutions} hold.
		\end{itemize}
	\end{proposition}
	
	\begin{remark}
		To leading order, the radius $r_A$ and the corresponding temporal wave number $\omega_A$ as well as the radius $r_B$ are as in the purely time-periodic case in Section \ref{sec:time-periodic-solutions}. In contrast, the temporal wave number $\omega_B$ exhibits a leading order correction depending on the rescaled spatial wave number $\tilde{k}_B$.
	\end{remark}

	\section{Modulation dynamics -- fast-moving fronts connecting time-periodic solutions}\label{sec:fronts}
	
	We now consider front solutions to \eqref{eq:final-amplitude-eq}, which connects different space-time periodic solutions. For this, we fix a speed $c > 0$ and make the ansatz
	\begin{equation*}
		\Aav(T,X_1) = r_A(\xi) \ee^{\ii (\phi_A(\xi) + \omega_A T)}, \quad \Bav(T,X_1) = r_B(\xi) \ee^{\ii (\phi_B(\xi) + \omega_B T)},
	\end{equation*}
	where $\xi = X_1 - c T$ denotes the co-moving frame travelling with speed $c$, cf.~\cite{vansaarloos1992,doelman1993,doelman1996}. Note that these solutions are not pure travelling waves due to their explicit time dependence, but are a special case of defect solutions \cite{sandstede2004}. Inserting this into the amplitude equations \eqref{eq:final-amplitude-eq} and splitting into real and imaginary parts we find that
	\begin{equation}\label{eq:travelling-fronts-ode}
		\begin{split}
			-c r_A' &= 4 r_A'' + \alpha_u r_A - 4 r_A (\phi_A')^2 + \gamma_{1,r} r_A^3 + \gamma_{2,r} r_A r_B^2, \\
			\omega_A- c\phi_A' &= 4 \phi_A'' + 8 \dfrac{r_A'}{r_A} \phi_A' + \gamma_{1,i} r_A^2 + \gamma_{2,i} r_B^2, \\
			-c r_B' &= 4 r_B'' + \alpha_v r_B - \dfrac{c_g}{\varepsilon} r_B' - 4 r_B(\phi_B')^2 - 6 c_v r_B' \phi_B' - 3 c_v r_B \phi_B'' + \gamma_{7,r} r_B r_A^2 + \gamma_{8,r} r_B^3, \\
			\omega_B - c \phi_B' &= 4 \phi_B'' + 8 \dfrac{r_B'}{r_B} \phi_B' + 3 c_v \dfrac{r_B''}{r_B} - 3 c_v (\phi_B')^2 - \dfrac{c_g}{\varepsilon} \phi_B' + \gamma_{7,i} r_A^2 + \gamma_{8,i} r_B^2.
		\end{split}
	\end{equation}
	Here $(\cdot)' = \partial_{\xi}(\cdot)$ and $\gamma_{j,r}$ and $\gamma_{j,i}$ denote the real and imaginary part of $\gamma_{j}$, respectively.
	
	We restrict the analysis to the asymptotic regime of fast-travelling waves, that is, we set $c = \tfrac{c_0}{\varepsilon}$ with $c_0 > 0$, cf.~\cite{scheel1996,gauss2021}. Additionally, we introduce the rescaled time-variable $\tilde{\xi} = \varepsilon \xi$ to obtain
	\begin{equation}\label{eq:fast-travelling-fronts-ode-full}
		\begin{split}
			-c_0 \dot{r}_A &=  \alpha_u r_A + \gamma_{1,r} r_A^3 + \gamma_{2,r} r_A r_B^2 + \varepsilon^2 \left(4 \ddot{r}_A - 4 r_A \dot{\phi}_A^2\right), \\
			\omega_A - c_0 \dot{\phi}_A &= \gamma_{1,i} r_A^2 + \gamma_{2,i} r_B^2 + \varepsilon^2 \left(4 \ddot{\phi}_A + 8 \dfrac{\dot{r}_A}{r_A} \dot{\phi}_A\right), \\
			-c_0 \dot{r}_B &= \alpha_v r_B - c_g \dot{r}_B + \gamma_{7,r} r_B r_A^2 + \gamma_{8,r} r_B^3 + \varepsilon^2 \left(4 \ddot{r}_B - 4 r_B\dot{\phi}_B^2 - 6 c_v \dot{r}_B \dot{\phi}_B - 3 c_v r_B \ddot{\phi}_B\right), \\
			\omega_B - c_0 \dot{\phi}_B &= -c_g \dot{\phi}_B + \gamma_{7,i} r_A^2 + \gamma_{8,i} r_B^2 + \varepsilon^2 \left(4 \ddot{\phi}_B + 8 \dfrac{\dot{r}_B}{r_B} \dot{\phi}_B + 3 c_v  \dfrac{\ddot{r}_B}{r_B} - 3 c_v \dot{\phi}_B^2\right),
		\end{split}
	\end{equation}
	where $\dot{(\cdot)} = \partial_{\tilde{\xi}}$. To study the dynamics of \eqref{eq:fast-travelling-fronts-ode-full} we proceed as follows. First, we observe that $\phi_A$ and $\phi_B$ only appear as derivatives, which is a consequence of the $\mathbb{S}^1$-symmetry of the amplitude system \eqref{eq:final-amplitude-eq}. Therefore, we introduce the local wave numbers $\psi_A = \dot{\phi}_A$ and $\psi_B = \dot{\phi}_B$ and study the resulting $(r_A,r_B,\psi_A,\psi_B)$-system. Second, we neglect the terms of order $\varepsilon^2$ and consider the truncated system. Finally, we show that the heteroclinic orbits relevant for the present analysis persist using geometric singular perturbation theory, see e.g.~\cite{fenichel1979,kuehn2015}.
	
	\subsection{Analysis of the truncated system}
	
	Neglecting terms of order $\varepsilon^2$ and replacing $\psi_A = \dot{\phi}_A$ and $\psi_B = \dot{\phi}_B$, we find, after reordering, the system
	\begin{equation}\label{eq:slow-dynamics}
		\begin{split}
			-c_0 \dot{r}_A &=  \alpha_u r_A + \gamma_{1,r} r_A^3 + \gamma_{2,r} r_A r_B^2, \\
			(c_g - c_0) \dot{r}_B &= \alpha_v r_B + \gamma_{7,r} r_B r_A^2 + \gamma_{8,r} r_B^3 \\
			- c_0 \psi_A &= - \omega_A + \gamma_{1,i} r_A^2 + \gamma_{2,i} r_B^2, \\
			(c_g - c_0) \psi_B &= -\omega_B + \gamma_{7,i} r_A^2 + \gamma_{8,i} r_B^2.
		\end{split}
	\end{equation}
	Note that the equations for $r_A$ and $r_B$ do not depend on the local wave numbers $\psi_A$ and $\psi_B$. Therefore, we can study the $(r_A,r_B)$-dynamics in isolation and consider the system
	\begin{equation}\label{eq:radii-dynamics}
		\begin{split}
			-c_0 \dot{r}_A &=  \alpha_u r_A + \gamma_{1,r} r_A^3 + \gamma_{2,r} r_A r_B^2, \\
			(c_g - c_0) \dot{r}_B &= \alpha_v r_B + \gamma_{7,r} r_B r_A^2 + \gamma_{8,r} r_B^3.
		\end{split}
	\end{equation}
	This is a two-dimensional first order system of ordinary differential equations (ODEs) and can thus be studied using phase plane analysis. Yet, the system~\eqref{eq:radii-dynamics} is a quite general polynomial planar system, which makes this task already challenging.
	
	To simplify the following analysis, we reduce the number of free parameters through appropriate rescalings of "time" and the variables $r_A$ and $r_B$. For this, we assume that both semi-trivial equilibrium points exist and therefore, we have the parameter condition $\gamma_{1,r} < 0$ and $\gamma_{8,r} < 0$ since $\alpha_u, \alpha_v > 0$ by assumption. This results in the rescaled system
	\begin{equation}\label{eq:radii-dynamics-rescaled}
		\begin{split}
			\dot{r}_A &= -r_A + r_A^3 - \tilde{\gamma}_A r_A r_B^2 =: r_A F_A(r_A, r_B), \\
			\dot{r}_B &= \dfrac{1}{\tilde{c}}\left(r_B - r_B^3 + \tilde{\gamma}_B r_B r_A^2\right) =: r_B F_B(r_A,r_B)
		\end{split}
	\end{equation}
	with $\tilde{c}, \tilde{\gamma}_A, \tilde{\gamma}_B \in \R$ given by
	\begin{equation*}
		\begin{split}
			\tilde{c} &= \dfrac{(c_g - c_0)\alpha_u}{c_0 \alpha_v}, \qquad \tilde{\gamma}_A = \dfrac{\alpha_v \gamma_{2,r}}{\alpha_u |\gamma_{8,r}|}, \qquad \tilde{\gamma}_B = \dfrac{\alpha_u \gamma_{7,r}}{\alpha_v |\gamma_{1,r}|}.
		\end{split}
	\end{equation*}
	We point out that the sign of $\tilde{c}$ is still given by the sign of $c_g - c_0$. Additionally, note that, to keep the notation simple, we still use the same notation for the rescaled variables and the rescaled "time".
	
	\subsubsection{Equilibrium points and their stability}
	To start the analysis of \eqref{eq:radii-dynamics}, we consider its equilibrium points and their stability. The system \eqref{eq:radii-dynamics-rescaled} has the trivial equilibrium point $T = (0,0)$ and two semi-trivial equilibrium points
	\begin{equation*}
		ST_A = (1,0) \text{ and } ST_B = (0,1)
	\end{equation*}
	similar to Proposition \ref{prop:semi-trivial-solutions}. Additionally, under the conditions of Proposition \ref{prop:fully-nontrivial-solutions}, it also has a fully nontrivial equilibrium point $NT$ with $r_A > 0$ and $r_B > 0$.
	
	We now linearise about these equilibrium points. Starting with $T$, we find that the corresponding linearisation has eigenvalues $\lambda_{T,1} = -1$ and $\lambda_{T,2} = \tilde{c}^{-1}$ with eigenvectors $(1,0)$ and $(0,1)$, respectively. In particular, $\lambda_{T,1}$ is negative and $\mathrm{sign}(\lambda_{T,2}) = \mathrm{sign}(c_g - c_0)$. Therefore, if $c_g < c_0$ the trivial equilibrium point is stable and if $c_g > c_0$ it is a saddle. 
	
	Next, we linearise about the semi-trivial equilibrium point $ST_A$. The resulting linearisation is again a diagonal matrix with eigenvalues $\lambda_{A,1} = 2$ and $\lambda_{A,2} = \tilde{c}^{-1} (1+\tilde{\gamma}_B)$ with eigenvectors $(1,0)$ and $(0,1)$, respectively. Similarly, linearising about the other semi-trivial equilibrium point $ST_B$ we find that the linearisation has the eigenvalues $\lambda_{B,1} = -1-\tilde{\gamma}_A$ and $\lambda_{B,2} = -2 \tilde{c}^{-1}$ with eigenvectors $(1,0)$ and $(0,1)$, respectively. 
	
	Finally, we linearise about $NT$, denoted by $L_{NT}$. Instead of explicitly calculating the eigenvalues, we consider the determinant and trace of the linearisation given by
	\begin{equation}\label{eq:LNT-det-tr}
		\det(L_{NT}) = \dfrac{4 (1+\tilde{\gamma}_A)(1+\tilde{\gamma}_B)}{-\tilde{c} (1-\tilde{\gamma}_A \tilde{\gamma}_B)} \qquad \operatorname{tr}(L_{NT}) = \dfrac{2 (1 + \tilde{\gamma}_B - \tilde{c} (1+ \tilde{\gamma}_A))}{-\tilde{c} (1-\tilde{\gamma}_A \tilde{\gamma}_B)}.
	\end{equation}
	Given the sign of the determinant and trace, we obtain the sign of the real parts of the eigenvalues of $L_{NT}$, since $L_{NT}$ is a real $2\times 2$-matrix, and thus the stability of $NT$.
	
	\begin{proposition}\label{prop:stability-NT}
		Assume that the parameter assumptions \eqref{eq:non-triv-fp-invertibility}--\eqref{eq:non-triv-fp-pos2} hold. If $\tilde{d} < 0$ then $NT$ is a
		\begin{itemize}
			\item saddle if $\tilde{c} < 0$,
			\item stable equilibrium point if $0 < \tilde{c} < \tilde{c}^\ast := \tfrac{1+\tilde{\gamma}_B}{1+\tilde{\gamma}_A}$,
			\item unstable equilibrium point if $\tilde{c}^\ast < \tilde{c}$.
		\end{itemize}
		If $\tilde{d} > 0$, then $NT$ is
		\begin{itemize}
			\item an unstable equilibrium point if $\tilde{c} < 0$, 
			\item a saddle if $\tilde{c} > 0$.
		\end{itemize}
		
		Additionally, for every $\tilde{\gamma}_A, \tilde{\gamma}_B$ such that $\tilde{d} < 0$ exists a pair $0 < \tilde{c}_{-,1} < \tilde{c}_{-,2}$ such that $\tilde{c}^* \in (\tilde{c}_{-,1},\tilde{c}_{-,2})$ and the linearisation about $NT$ has a pair of complex conjugated eigenvalues if $\tilde{c} \in (\tilde{c}_{-,1},\tilde{c}_{-,2})$ and two real eigenvalues otherwise. In particular, $NT$ is a centre at $\tilde{c} = \tilde{c}^*$.
		
		Furthermore, for every $\tilde{\gamma}_A, \tilde{\gamma}_B$ such that $\tilde{d} > 1$ exists a pair $\tilde{c}_{+,1} < \tilde{c}_{+,2} < 0$ such that the linearisation about $NT$ has a pair of complex conjugated eigenvalues if $\tilde{c} \in (\tilde{c}_{+,1},\tilde{c}_{+,2})$ and two real eigenvalues otherwise.
	\end{proposition}
	\begin{proof}
		The statements can be obtained from a discussion of the signs of $\det(L_{NT})$ and $\operatorname{tr}(L_{NT})$, which are given in \eqref{eq:LNT-det-tr}. First, in the case $\tilde{d} < 0$, we find that $1+\tilde{\gamma}_A < 0$ and $1+ \tilde{\gamma}_B < 0$ due to \eqref{eq:non-triv-fp-pos1} and \eqref{eq:non-triv-fp-pos2}, respectively. Therefore, the sign of $\det(L_{NT})$ is given by the sign of $\tilde{c}$. Hence, for $\tilde{c} < 0$, it follows that $NT$ is a saddle. For $\tilde{c} > 0$, it is either a stable or unstable equilibrium point and we need to consider the sign of $\operatorname{tr}(L_{NT})$. Writing
		\begin{equation*}
			\operatorname{tr}(L_{NT}) = 2 (NT_A)^2 - \dfrac{2 (NT_B)^2}{\tilde{c}}
		\end{equation*}
		with $NT = (NT_A, NT_B)$, we find that $\operatorname{tr}(L_{NT})$ is strictly increasing in $\tilde{c}$ and that $\operatorname{tr}(L_{NT}) = 0$ for $\tilde{c}^\ast$. Together with $\operatorname{tr}(L_{NT}) < 0$ for $0 < \tilde{c} \ll 1$ and $\operatorname{tr}(L_{NT}) > 0$ for $\tilde{c} \gg 1$, the statement for $\tilde{d}< 0$ follows.
		For $\tilde{d} > 0$ we find that $1+\tilde{\gamma}_A > 0$ and $1+\tilde{\gamma}_B > 0$. A similar discussion as in the case $d < 0$ then gives the stability properties of $NT$ for $\tilde{d} > 0$.
		
		Since $L_{NT}$ is a real-valued $2\times 2$ matrix, two eigenvalues can either be real or a pair of complex conjugated eigenvalues. In particular, they are real if 
		\begin{equation*}
			\operatorname{tr}(L_{NT})^2 - 4 \det(L_{NT}) = \dfrac{4}{\tilde{c}^2} \left((NT_B)^4 - 2 \tilde{c} (1-2 \tilde{d}) (NT_A)^2 (NT_B)^2 + \tilde{c}^2 (NT_A)^4\right) \geq 0
		\end{equation*}
		and have a non-trivial imaginary part otherwise. To prove the remainder of the statement, we thus consider the roots of the polynomial
		\begin{equation*}
			p(\tilde{c}) = (NT_B)^4 - 2 \tilde{c} (1-2\tilde{d}) (NT_A)^2 (NT_B)^2 + \tilde{c}^2 (NT_A)^4.
		\end{equation*}
		First, we note that $p$ is a  polynomial of degree $2$ and satisfies $\lim_{\tilde{c} \rightarrow \pm \infty} p(\tilde{c}) = \infty$ and $p(0) = (NT_B)^4 > 0$. Therefore, if $p$ has two real roots, they are either both negative or both positive. Additionally, we point out that for $\tilde{c} < 0$ and $\tilde{d} < 0$ or $\tilde{c} > 0$ and $\tilde{d} \in (0,1/2]$, the linear coefficient of $p$ is non-negative and therefore, $p$ cannot have two real roots. The roots of $p$ are explicitly given by
		\begin{equation*}
			\tilde{c}_\pm = \dfrac{(NT_B)^2}{(NT_A)^2}\left((1-2\tilde{d})\pm 2 \sqrt{(\tilde{d}-1)\tilde{d}}\right)
		\end{equation*}
		They are real if and only if $(\tilde{d}-1)\tilde{d} \geq 0$ that is $d \notin (0,1)$.
		
		Let now $\tilde{d} < 0$. Then $1-2\tilde{d} > 0$ and thus $\tilde{c}_+ > 0$. Since the roots have the same sign this also shows that $\tilde{c}_- > 0$. Therefore, we find that $L_{NT}$ has a pair of complex conjugated eigenvalues if and only if $\tilde{c} \in (c_-,c_+)$. It remains to show that $\tilde{c}^* \in (c_-,c_+)$. For this we note that
		\begin{equation*}
			\tilde{c}^* = \dfrac{1+\tilde{\gamma}_B}{1+\tilde{\gamma}_A} = \dfrac{(NT_B)^2}{(NT_A)^2}.
		\end{equation*}
		Then, $\tilde{c} \in (c_-,c_+)$ follows from the inequality
		\begin{equation*}
			1 - 2 \tilde{d} - 2 \sqrt{(\tilde{d} - 1) \tilde{d}} \leq 1 + 2 |\tilde{d}| - 2 \sqrt{\tilde{d}^2} = 1 \leq 1 - 2 \tilde{d} + 2 \sqrt{\tilde{d}^2} \leq 1 - 2 \tilde{d} + 2 \sqrt{(\tilde{d} - 1) \tilde{d}},
		\end{equation*}
		where we use that $\tilde{d} < 0$ and thus $\tilde{d}^2 - \tilde{d} \geq \tilde{d}^2$.
		
		Finally, let $\tilde{d} > 1$. Then $(1-2\tilde{d}) < 0$ and thus, both roots are negative and $L_{NT}$ has a pair of complex conjugated eigenvalues if and only if $\tilde{c}_- < \tilde{c} < \tilde{c}_+ < 0$. This completes the proof.
	\end{proof}

	\subsubsection{Dynamics on the invariant coordinate axes}
	
	We first analyse the dynamics of the rescaled system \eqref{eq:radii-dynamics-rescaled} on the invariant coordinate axis $\{r_B = 0\}$ and $\{r_A = 0\}$. On $\{r_B = 0\}$, we have two equilibrium points $T$ and $ST_A$, which, restricted to the invariant coordinate axis, are stable and unstable respectively. Therefore, we find a heteroclinic orbit from $ST_A$ to $T$.
	
	To understand the dynamics on $\{r_A = 0\}$, we have to distinguish two cases, namely $\tilde{c} > 0$ and $\tilde{c} < 0$. In the first case, we find that the trivial equilibrium point $T$ restricted to $\{r_A = 0\}$ is unstable, whereas $ST_B$ is stable. Therefore, there is a heteroclinic orbit from $T$ to $ST_B$. In the second case, $T$ is stable and $ST_B$ is unstable and there is a heteroclinic orbit from $ST_B$ to $T$.
	
	\begin{remark}
		Heteroclinic orbits in the ODE system \eqref{eq:radii-dynamics-rescaled} connecting $\mathrm{ST}_B$ and the origin corresponds to a solution to the pattern-forming system \eqref{eq:toy-model}, which, to leading order, connects a travelling wave with group velocity $c_g$ with the trivial steady state through a spatial front with speed $c_0$. Therefore, recalling that $\mathrm{sign}(\tilde{c}) = \mathrm{sign}(c_g - c_0)$, the distinction between $\tilde{c} > 0$ and $\tilde{c} < 0$ ensures that the group velocity cannot point towards the front interface in a co-moving frame with speed $c_0$. This condition also appeared in the construction of modulating fronts close to a Turing–Hopf instability, see \cite{hilder2022}, and seems necessary to obtain nonlinear stability of pattern-forming fronts, cf.~e.g.~\cite{avery2025a}.
	\end{remark}

	\subsubsection{Dynamics outside of invariant subsets}
	
	We now study the dynamics outside of the invariant subsets $\{r_A = 0\}$ and $\{r_B = 0\}$. First, we consider the case when the fully nontrivial equilibrium point $\mathrm{NT}$ exists. From Proposition \ref{prop:fully-nontrivial-solutions} we know that the parameter-assumptions
	\begin{subequations}
		\begin{align}
			\tilde{d} := 1 - \tilde{\gamma}_A \tilde{\gamma}_B &\neq 0, \label{eq:non-triv-fp-invertibility}\\
			\dfrac{1 + \tilde{\gamma}_A}{ 1 - \tilde{\gamma}_A \tilde{\gamma}_B} &>  0, \label{eq:non-triv-fp-pos1}\\
			\dfrac{1 + \tilde{\gamma}_B}{ 1 - \tilde{\gamma}_A \tilde{\gamma}_B} &> 0 \label{eq:non-triv-fp-pos2}
		\end{align}
	\end{subequations}
	hold. From these conditions, we can determine the stability of the semi-trivial equilibrium points $ST_A$ and $ST_B$ from the signs of $\tilde{c}$ and $\tilde{d}$, which is summarised in Table \ref{tab:d-c-diagram}.
	
	\begin{table}
		\begin{tabular}{|c|c|c|}
			\hline
			& $\tilde{d} > 0$ & $\tilde{d} < 0$ \\
			\hline
			$\tilde{c} > 0$ &
			\begin{tikzpicture}
				\begin{axis}[
					axis lines=middle,
					xlabel={$r_A$},
					ylabel={$r_B$},
					xmin=0, xmax=1.5,
					ymin=0, ymax=1.5,
					xtick=\empty,
					ytick=\empty,
					width=5cm,
					height=5cm,
					axis line style={->},
					clip=false,
					every axis x label/.style={at={(ticklabel* cs:1)}, anchor=north west},
					every axis y label/.style={at={(ticklabel* cs:1)}, anchor=south west},
					]
					
					\addplot[only marks, mark=*, mark size=1.5pt] coordinates {(1, 0)};
					\addplot[only marks, mark=*, mark size=1.5pt] coordinates {(0, 1)};
					\node[below] at (axis cs:1, 0) {$ST_A$};
					\node[left] at (axis cs:0, 1) {$ST_B$};
					
					\addplot[only marks, mark=*, mark size=1.5pt] coordinates {(0,0)};
					\node[below left] at (axis cs:0,0) {$T$};
					
					\draw[very thick,->] (axis cs:1,0) -- (axis cs: 0.7,0);
					\draw[very thick,->] (axis cs:1,0) -- (axis cs: 1.3,0);
					\draw[very thick,->] (axis cs:1,0) -- (axis cs: 1,0.3);
					
					\draw[very thick,->] (axis cs:0,1.3) -- (axis cs: 0,1.03);
					\draw[very thick,->] (axis cs:0,0.7) -- (axis cs: 0,0.97);
					\draw[very thick,->] (axis cs:0.3,1) -- (axis cs: 0.03,1);
					
				\end{axis}
			\end{tikzpicture}
			&
			\begin{tikzpicture}
				\begin{axis}[
					axis lines=middle,
					xlabel={$r_A$},
					ylabel={$r_B$},
					xmin=0, xmax=1.5,
					ymin=0, ymax=1.5,
					xtick=\empty,
					ytick=\empty,
					width=5cm,
					height=5cm,
					axis line style={->},
					clip=false,
					every axis x label/.style={at={(ticklabel* cs:1)}, anchor=north west},
					every axis y label/.style={at={(ticklabel* cs:1)}, anchor=south west},
					]
					
					\addplot[only marks, mark=*, mark size=1.5pt] coordinates {(1, 0)};
					\addplot[only marks, mark=*, mark size=1.5pt] coordinates {(0, 1)};
					\node[below] at (axis cs:1, 0) {$ST_A$};
					\node[left] at (axis cs:0, 1) {$ST_B$};
					
					\addplot[only marks, mark=*, mark size=1.5pt] coordinates {(0,0)};
					\node[below left] at (axis cs:0,0) {$T$};
					
					\draw[very thick,->] (axis cs:1,0) -- (axis cs: 0.7,0);
					\draw[very thick,->] (axis cs:1,0) -- (axis cs: 1.3,0);
					\draw[very thick,->] (axis cs:1,0.3) -- (axis cs: 1,0.03);
					
					\draw[very thick,->] (axis cs:0,1.3) -- (axis cs: 0,1.03);
					\draw[very thick,->] (axis cs:0,0.7) -- (axis cs: 0,0.97);
					\draw[very thick,->] (axis cs:0,1) -- (axis cs: 0.3,1);
					
				\end{axis}
			\end{tikzpicture}
			\\
			\hline
			$\tilde{c} <0$ &
			\begin{tikzpicture}
				\begin{axis}[
					axis lines=middle,
					xlabel={$r_A$},
					ylabel={$r_B$},
					xmin=0, xmax=1.5,
					ymin=0, ymax=1.5,
					xtick=\empty,
					ytick=\empty,
					width=5cm,
					height=5cm,
					axis line style={->},
					clip=false,
					every axis x label/.style={at={(ticklabel* cs:1)}, anchor=north west},
					every axis y label/.style={at={(ticklabel* cs:1)}, anchor=south west},
					]
					
					\addplot[only marks, mark=*, mark size=1.5pt] coordinates {(1, 0)};
					\addplot[only marks, mark=*, mark size=1.5pt] coordinates {(0, 1)};
					\node[below] at (axis cs:1, 0) {$ST_A$};
					\node[left] at (axis cs:0, 1) {$ST_B$};
					
					\addplot[only marks, mark=*, mark size=1.5pt] coordinates {(0,0)};
					\node[below left] at (axis cs:0,0) {$T$};
					
					\draw[very thick,->] (axis cs:1,0) -- (axis cs: 0.7,0);
					\draw[very thick,->] (axis cs:1,0) -- (axis cs: 1.3,0);
					\draw[very thick,->] (axis cs:1,0.3) -- (axis cs: 1,0.03);
					
					\draw[very thick,->] (axis cs:0,1) -- (axis cs: 0,1.3);
					\draw[very thick,->] (axis cs:0,1) -- (axis cs: 0,0.7);
					\draw[very thick,->] (axis cs:0.3,1) -- (axis cs: 0.03,1);
					
				\end{axis}
			\end{tikzpicture}
			&
			\begin{tikzpicture}
				\begin{axis}[
					axis lines=middle,
					xlabel={$r_A$},
					ylabel={$r_B$},
					xmin=0, xmax=1.5,
					ymin=0, ymax=1.5,
					xtick=\empty,
					ytick=\empty,
					width=5cm,
					height=5cm,
					axis line style={->},
					clip=false,
					every axis x label/.style={at={(ticklabel* cs:1)}, anchor=north west},
					every axis y label/.style={at={(ticklabel* cs:1)}, anchor=south west},
					]
					
					\addplot[only marks, mark=*, mark size=1.5pt] coordinates {(1, 0)};
					\addplot[only marks, mark=*, mark size=1.5pt] coordinates {(0, 1)};
					\node[below] at (axis cs:1, 0) {$ST_A$};
					\node[left] at (axis cs:0, 1) {$ST_B$};
					
					\addplot[only marks, mark=*, mark size=1.5pt] coordinates {(0,0)};
					\node[below left] at (axis cs:0,0) {$T$};
					
					\draw[very thick,->] (axis cs:1,0) -- (axis cs: 0.7,0);
					\draw[very thick,->] (axis cs:1,0) -- (axis cs: 1.3,0);
					\draw[very thick,->] (axis cs:1,0) -- (axis cs: 1,0.3);
					
					\draw[very thick,->] (axis cs:0,1) -- (axis cs: 0,1.3);
					\draw[very thick,->] (axis cs:0,1) -- (axis cs: 0,0.7);
					\draw[very thick,->] (axis cs:0,1) -- (axis cs: 0.3,1);
					
				\end{axis}
			\end{tikzpicture}
			\\
			\hline
		\end{tabular}
		\caption{Stability of the semi-trivial equilibrium points in \eqref{eq:radii-dynamics-rescaled} depending on the signs of $\tilde{c}$ and $\tilde{d}$ assuming the the existence conditions \eqref{eq:non-triv-fp-invertibility}--\eqref{eq:non-triv-fp-pos2} for the fully nontrivial equilibrium point $NT$ are satisfied.}
		\label{tab:d-c-diagram}
	\end{table}
	
	\paragraph{The case $\tilde{d} < 0 $}
	
	We now discuss the dynamics of \eqref{eq:radii-dynamics-rescaled} depending on the sign of $\tilde{d}$, starting with the case that $\tilde{d} < 0$. This implies that $\tilde{\gamma}_A \tilde{\gamma}_B > 1$. Additionally, the positivity conditions \eqref{eq:non-triv-fp-pos1} and \eqref{eq:non-triv-fp-pos2} imply that $1 + \tilde{\gamma}_A < 0$ and $1 + \tilde{\gamma}_B < 0$. Combining these, we find that $\tilde{\gamma}_A < -1$ and $\tilde{\gamma}_B < -1$.
	
	Numerical calculations indicate the following: if $\tilde{c} < 0$, the equilibrium point $NT$ is a saddle, see Proposition \ref{prop:stability-NT}, and there are heteroclinic orbits $ST_A \longrightarrow NT$, $ST_B \longrightarrow NT$ and $NT \longrightarrow T$. If $\tilde{c} >0$, the situation seems to be more subtle. At $\tilde{c} = 0$ (which is a singularity), $NT$ transitions from a saddle to a stable equilibrium point. For $0 < \tilde{c} \ll 1$, the system \eqref{eq:radii-dynamics-rescaled} has a fast-slow structure, which can be used to prove the existence of heteroclinic orbit $ST_B \longrightarrow NT$, see Lemma \ref{lem:het-orbit-d<0-c<<1}, but the orbit $ST_A \longrightarrow NT$ ceases to exist since $ST_A$ is now a saddle point with a one-dimensional unstable manifold explicitly given by $\{r_B = 0\}$. Increasing $\tilde{c} > 0$ beyond a small neighbourhood of $0$, we find that $NT$ becomes a stable centre and the orbits $ST_B \longrightarrow NT$ behave like an inward spiral close to $NT$. Further increasing $\tilde{c}$ there is a critical value $\tilde{c}^* > 0$, see Proposition \ref{prop:stability-NT}, at which $NT$ transitions from a stable to an unstable centre and the orbit $ST_A \longrightarrow NT$ breaks up. At this critical value, $NT$ is a centre, and there is a heteroclinic connection $ST_B \longrightarrow ST_A$, which encloses a neighbourhood of $T$ filled with periodic orbits around $NT$. For $\tilde{c} > \tilde{c}^\ast$, there is only a heteroclinic orbit $NT \longrightarrow ST_A$. If $\tilde{c} - \tilde{c}^\ast > 0$ is sufficiently small, $NT$ is an unstable centre and therefore the heteroclinic orbit behaves like an outward spiral close to $NT$. For $\tilde{c} > 0$ sufficiently large, the eigenvalues of $NT$ become real again, see Proposition \ref{prop:stability-NT}, and the heteroclinic $NT \longrightarrow ST_A$ seems to be monotone again. The different scenarios are depicted in Figure \ref{fig:phase-planes-d<0}.
	
	\begin{figure}
        \begin{overpic}[width=0.33\textwidth]
		      {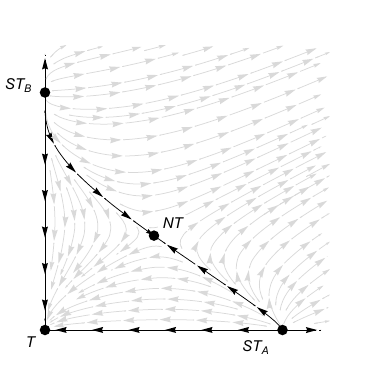}
              \put(50,-2){\makebox[0pt]{\small$\tilde{c}=-2$}}
        \end{overpic}
        \begin{overpic}[width=0.33\textwidth]
		      {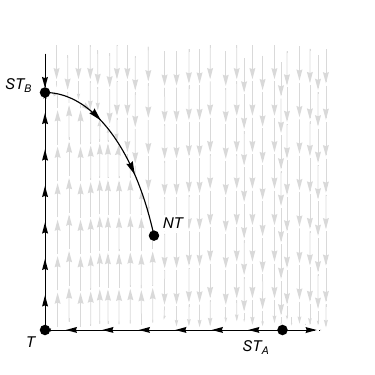}
              \put(50,-2){\makebox[0pt]{\small$\tilde{c}=0.0001$}}
        \end{overpic}
        \begin{overpic}[width=0.33\textwidth]
		      {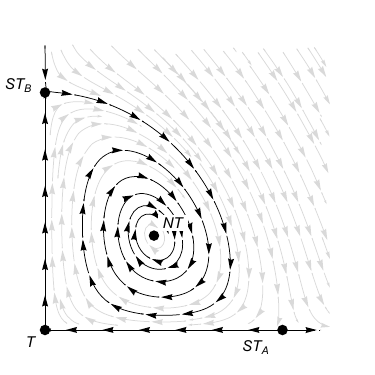}
              \put(50,-2){\makebox[0pt]{\small$\tilde{c}=0.5$}}
        \end{overpic}

        \begin{overpic}[width=0.33\textwidth]
		      {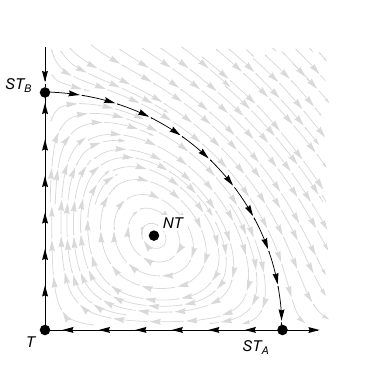}
              \put(50,-2){\makebox[0pt]{\small$\tilde{c}= \tilde{c}^\ast = 0.75$}}
        \end{overpic}
        \begin{overpic}[width=0.33\textwidth]
		      {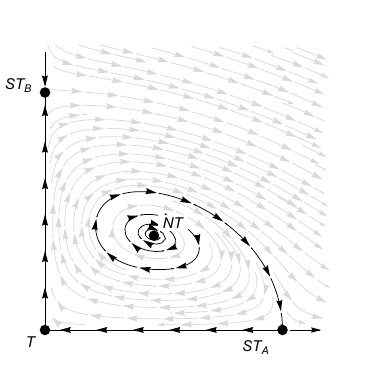}
              \put(50,-2){\makebox[0pt]{\small$\tilde{c}=2$}}
        \end{overpic}
        \begin{overpic}[width=0.33\textwidth]
		      {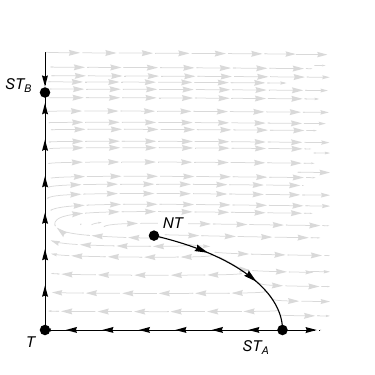}
              \put(50,-2){\makebox[0pt]{\small$\tilde{c}=100$}}
        \end{overpic}
		\caption{Phase planes for the case $\tilde{d} < 0$ and increasing $\tilde{c}$. In these plots $\tilde{\gamma}_A = -5$ and $\tilde{\gamma}_B = -4$ and therefore $\tilde{d} = -19$.}
		\label{fig:phase-planes-d<0}
	\end{figure}
	
	\begin{remark}
		If $\tilde{\gamma}_A = \tilde{\gamma}_B$, we obtain $\tilde{c}^\ast = 1$ and the dynamics and symmetries resemble the heteroclinic cycles studied in \cite{guckenheimer1988}.
	\end{remark}
	
	\paragraph{The case $\tilde{d} > 0$}
	
	Next, we consider the case $\tilde{d} > 0$, which implies that $\tilde{\gamma}_A \tilde{\gamma}_B < 1$. In addition, the positivity conditions \eqref{eq:non-triv-fp-pos1} and \eqref{eq:non-triv-fp-pos2} yield that $\tilde{\gamma}_A > -1$ and $\tilde{\gamma}_B > -1$.
	
	It turns out that the main decider for the dynamics of \eqref{eq:radii-dynamics-rescaled} is the sign of $\tilde{c}$. As Proposition \ref{prop:stability-NT} shows, if $d > 0$ then $NT$ is unstable for $\tilde{c} < 0$ and a saddle for $\tilde{c} > 0$. We now discuss both regimes separately. For $\tilde{c} < 0$, both $ST_A$ and $ST_B$ are stable in the directions off the invariant coordinate axes and unstable on the coordinate axes, see Table \ref{tab:d-c-diagram}. In particular, the trivial equilibrium point $T$ is stable in this regime. A numerical investigation then indicates that there are heteroclinic orbits $NT \longrightarrow T$, $NT \longrightarrow ST_A$ as well as $NT \longrightarrow ST_B$. For $\tilde{c} > 0$, $ST_B$ becomes stable whereas $ST_A$ becomes unstable, see Table \ref{tab:d-c-diagram}. A numerical phase plane analysis then suggests that there are heteroclinic orbits $ST_A \longrightarrow NT$ and $NT \longrightarrow ST_B$. These orbits form the boundary of an invariant, compact subset of the phase plane, which is filled with heteroclinic orbits $ST_A \longrightarrow ST_B$. The numerical findings in the different regimes are also depicted in Figure \ref{fig:phase-plane-d>0}.
	
	\begin{figure}
		\centering
		\begin{overpic}[width=0.33\textwidth]
		      {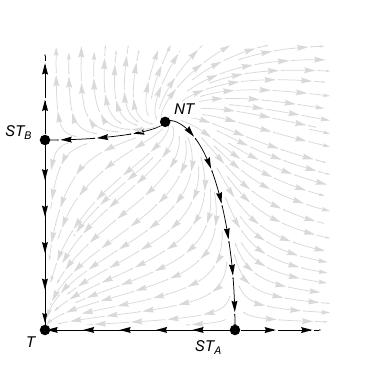}
              \put(50,-2){\makebox[0pt]{\small$\tilde{c}=-2$}}
        \end{overpic}
        \hspace{0.5cm}
        \begin{overpic}[width=0.33\textwidth]
		      {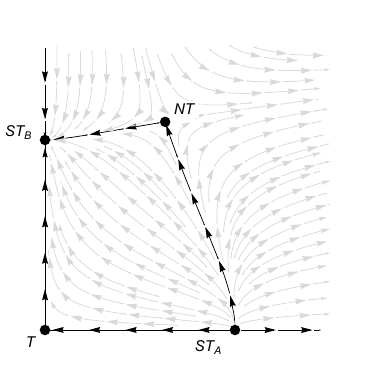}
              \put(50,-2){\makebox[0pt]{\small$\tilde{c}=2$}}
        \end{overpic}
		\caption{Phase planes for the case $\tilde{d} > 0$ for different values of $\tilde{c}$. In these plots $\tilde{\gamma}_A = -0.5$ and $\tilde{\gamma}_B = 0.5$ and therefore, $\tilde{d} = 1.25$.}
		\label{fig:phase-plane-d>0}
	\end{figure}

	\paragraph{Dynamics if $NT$ does not exist}
	
	Finally, we discuss the dynamics in the case that the fully nontrivial equilibrium point $NT$ does not exist, that is, if the assumptions \eqref{eq:non-triv-fp-invertibility}, \eqref{eq:non-triv-fp-pos1} or \eqref{eq:non-triv-fp-pos2} fail. Since the case that $\tilde{d} = 0$ is non-generic, we restrict the discussion to parameter regimes, where $\tilde{d} \neq 0$, but the positivity conditions \eqref{eq:non-triv-fp-pos1} or \eqref{eq:non-triv-fp-pos2} are violated. Since \eqref{eq:non-triv-fp-pos1} and \eqref{eq:non-triv-fp-pos2} can no longer be used to determine the stability of the equilibrium points $ST_A$ and $ST_B$, there is no direct relation between the eigenvalues and the sign of $\tilde{d}$ as in Table \ref{tab:d-c-diagram}. Instead, the stability of $ST_A$ is determined by the sign of $\tilde{c}$ and the sign of $1+\tilde{\gamma}_B$. Similarly, the stability of $ST_B$ is determined by the sign of $\tilde{c}$ and $1+\tilde{\gamma}_A$.
	
	We first discuss the case that $\tilde{c} < 0$. In this setting, $T$ is a stable equilibrium point. Recall that the eigenvalues of the semi-trivial equilibrium points are given by
	\begin{equation*}
		\lambda_{A,1} = 2, \quad \lambda_{A,2} = \dfrac{1+\tilde{\gamma}_B}{\tilde{c}}, \quad \lambda_{B,1} = - (1 + \tilde{\gamma}_A), \quad \lambda_{B,2} = -\dfrac{2}{\tilde{c}}.
	\end{equation*}
	If $\tilde{\gamma}_A < -1$ then $\tilde{\gamma}_B > -1$ since $NT$ does not exist. Therefore, $ST_A$ is a saddle and $ST_B$ is unstable. Numerical phase plane analysis then suggests that there are heteroclinic connections $ST_B \longrightarrow T$ and $ST_B \longrightarrow ST_A$. If $\tilde{\gamma}_B < -1$, then $\tilde{\gamma}_A > -1$ and $ST_A$ is unstable whereas $ST_B$ is a saddle. In this case, numerical phase plane analysis suggests that there are heteroclinic orbits $ST_A \longrightarrow T$ and $ST_A \longrightarrow ST_B$. Finally, we consider the case that both $\tilde{\gamma}_A > -1$ and $\tilde{\gamma}_B > -1$. In this situation, both $ST_A$ and $ST_B$ are saddles, however, they are stable in the direction transverse to the invariant coordinate axes. Since $T$ is also stable, there are no heteroclinic orbits outside of the invariant coordinate axes.
	
	Next, we consider the case $\tilde{c} > 0$. If $\tilde{\gamma}_A < -1$ and thus $\tilde{\gamma}_B > -1$, both $ST_A$ and $ST_B$ are unstable in the direction transverse to the respective coordinate axis. Since $T$ is a saddle with stable manifold given by $\{r_B = 0\}$, we again find that there are no heteroclinic orbits outside of the invariant coordinate axes. If $\tilde{\gamma}_A > -1$ and $\tilde{\gamma}_B < -1$, both $ST_A$ and $ST_B$ are stable in the transverse direction and thus, again no heteroclinic orbits outside of the invariant coordinate axes exist. Finally, we consider that both $\tilde{\gamma}_A > -1$ and $\tilde{\gamma}_B > -1$. Then, $ST_A$ is unstable whereas $ST_B$ is a stable equilibrium point. Numerical phase plane analysis then suggest a heteroclinic orbit $ST_A \longrightarrow ST_B$.
	
	The phase planes for the different situations are displayed in Figure \ref{fig:phase-plane-NT-non-ex}.
	
	\begin{figure}
		\begin{overpic}[width=0.33\textwidth]
		      {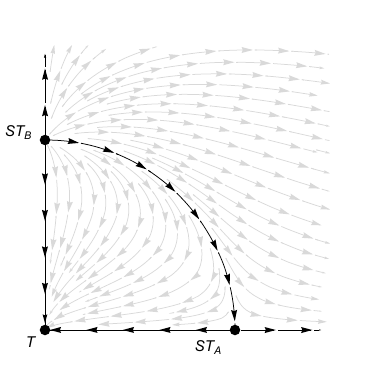}
              \put(50,-2){\makebox[0pt]{\small$\tilde{c}=-2$,\; $\tilde{\gamma}_A=-2$,\; $\tilde{\gamma}_B = 1$}}
        \end{overpic}
        \begin{overpic}[width=0.33\textwidth]
		      {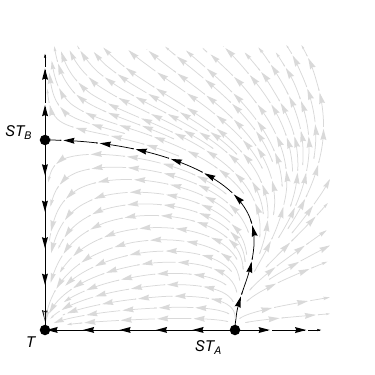}
              \put(50,-2){\makebox[0pt]{\small$\tilde{c}=-2$,\; $\tilde{\gamma}_A=1$,\; $\tilde{\gamma}_B = -2$}}
        \end{overpic}
        \begin{overpic}[width=0.33\textwidth]
		      {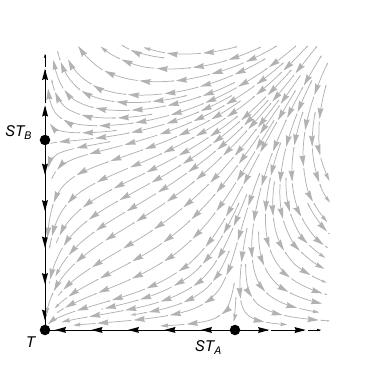}
              \put(50,-2){\makebox[0pt]{\small$\tilde{c}=-2$,\; $\tilde{\gamma}_A=1$,\; $\tilde{\gamma}_B = 2$}}
        \end{overpic}

        \begin{overpic}[width=0.33\textwidth]
		      {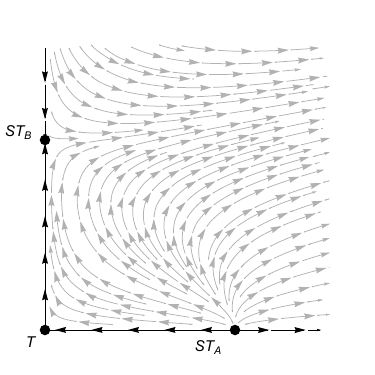}
              \put(50,-2){\makebox[0pt]{\small$\tilde{c}=2$,\; $\tilde{\gamma}_A=-2$,\; $\tilde{\gamma}_B = 1$}}
        \end{overpic}
        \begin{overpic}[width=0.33\textwidth]
		      {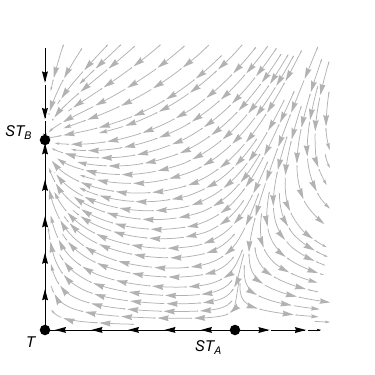}
              \put(50,-2){\makebox[0pt]{\small$\tilde{c}=2$,\; $\tilde{\gamma}_A=1$,\; $\tilde{\gamma}_B = -2$}}
        \end{overpic}
        \begin{overpic}[width=0.33\textwidth]
		      {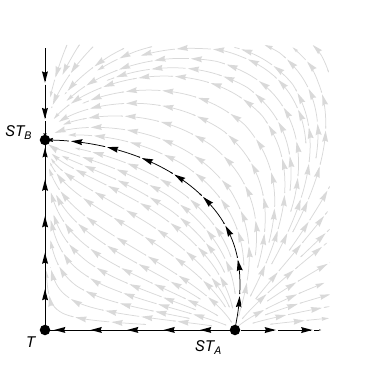}
              \put(50,-2){\makebox[0pt]{\small$\tilde{c}=2$,\; $\tilde{\gamma}_A=1$,\; $\tilde{\gamma}_B = 2$}}
        \end{overpic}
		\caption{Phase planes for the case that $NT$ does not exist due to a violation of the positive conditions \eqref{eq:non-triv-fp-pos1} or \eqref{eq:non-triv-fp-pos2}.}
		\label{fig:phase-plane-NT-non-ex}
	\end{figure}

	\subsection{Persistence of heteroclinic orbits} 
	
	It remains to prove that the heteroclinic orbits persist under the higher-order terms in \eqref{eq:fast-travelling-fronts-ode-full}. However, since using polar coordinates yields singularities in the radii at higher orders, we revert back to euclidean coordinates and look for travelling wave solutions to \eqref{eq:final-amplitude-eq} of the form
	\begin{equation}\label{eq:front-ansatz}
		\Aav(T,X) = A(\xi) \ee^{\ii\omega_A T}, \quad \Bav(T,X) = B(\xi) \ee^{\ii\omega_B T},
	\end{equation}
	with $\xi = X - c T$ and frequencies $\omega_A, \omega_B \in \R$ given by Proposition \ref{prop:semi-trivial-solutions}.
	
	\begin{remark}\label{rem:different-omegas}
		Notice that the specific choice of $\omega_A$ and $\omega_B$ is not strictly necessary. In fact, the following proofs can naturally be adapted to other choices. However, the constructed solutions will be slightly different as the time-periodic asymptotic states will then in general also involve spatial oscillations. To illustrate this, we consider $\Bav = 0$ and write $A(\xi) = r_A(\xi) \ee^{\ii\phi_A(\xi)}$. Then, after moving to $\tilde{\xi} = \varepsilon \xi$ and $c = \tfrac{c_0}{\varepsilon}$, we obtain to leading order
		\begin{equation*}
			\begin{split}
				-c_0 \dot{r}_A &= \alpha_u + \gamma_{1,r} r_A^3, \\
				-c_0 \dot{\phi}_A &= -\omega_A + \gamma_{1,i} r_A^2.
			\end{split}
		\end{equation*}
		If $\omega_A$ is chosen as in Proposition \ref{prop:semi-trivial-solutions}, the equation has a nontrivial equilibrium point. If not, the $r_A$-equation still has the same equilibrium point, but the solution will be $\tilde{\xi}$-periodic with a $\phi_A = \psi_A \tilde{\xi}$, where $\psi_A$ is the constant local wave number. Therefore, the phase of $\Aav$ is given by
		\begin{equation*}
			\phi_A(\tilde{\xi}) + \omega_A T = \psi_A (\varepsilon X - c_0 T) + \omega_A T = \varepsilon \psi_A X + \gamma_{1,i} r_A^2 T.
		\end{equation*}
		Recalling that $r_A^2 = \tfrac{\alpha_u}{\gamma_{1,r}}$, we thus obtain a space-time periodic solution with temporal wave number as in Proposition \ref{prop:semi-trivial-solutions}, but with an additional small spatial wave number.
	\end{remark}
	
	Inserting the ansatz \eqref{eq:front-ansatz} into \eqref{eq:final-amplitude-eq} yields the system
	\begin{equation*}
		\begin{split}
			- c A' + i\omega_A A &= 4 A'' + \alpha_u A + \gamma_1 A |A|^2 + \gamma_2 A |B|^2, \\
			- c B' + i\omega_B B &= (4 + 3ic_v) B'' - \dfrac{c_g}{\varepsilon} B' + \alpha_v B + \gamma_7 B|A|^2 + \gamma_8 B|B|^2.
		\end{split}
	\end{equation*}
	As above, we now set $c = \tfrac{c_0}{\varepsilon}$ with $c_0 \neq 0$. Then, in the frame $\tilde{\xi} = \varepsilon \xi$ we have
	\begin{equation}\label{eq:fast-fronts-eucledean-coords}
		\begin{split}
			-c_0 \dot{A} + i\omega_A A &= 4 \varepsilon^2 \ddot{A} + \alpha_u A + \gamma_1 A |A|^2 + \gamma_2 A |B|^2, \\
			- c_0 \dot{B} + i\omega_B B &= \varepsilon^2 (4 + 3 ic_v) \ddot{B} - c_g B + \alpha_v B + \gamma_7 B|A|^2 + \gamma_8 B|B|^2.
		\end{split}
	\end{equation}
	Finally, splitting $A = A_r + i A_i$ and $B = B_r + i B_i$ into real and imaginary parts, and rewriting the resulting equations as a first-order system in $\tilde{\xi}$, we obtain
	\begin{equation}\label{eq:full-fast-slow-system}
		\begin{split}
			\dot{A}_r &= \tilde{A}_r, \\
			\dot{A}_i &= \tilde{A}_i, \\
			\dot{B}_r &= \tilde{B}_r, \\
			\dot{B}_i &= \tilde{B}_i, \\
			4 \varepsilon^2 \dot{\tilde{A}}_r &= - c_0 \tilde{A}_r - \alpha_u A_r - \omega_A A_i - \gamma_{1,r} A_r |A|^2 + \gamma_{1,i} A_i |A|^2 - \gamma_{2,r} A_r |B|^2 + \gamma_{2,i} A_i |B|^2, \\
			4 \varepsilon^2 \dot{\tilde{A}}_i &= - c_0 \tilde{A}_i + \omega_A A_i - \alpha_u A_i - \gamma_{1,r} A_i |A|^2 - \gamma_{1,i} A_r |A|^2 - \gamma_{2,r} A_i |B|^2 - \gamma_{2,i} A_r |B|^2, \\
			\varepsilon^2 (4 \dot{\tilde{B}}_r - 3c_v \dot{\tilde{B}}_i) &= (c_g - c_0) \tilde{B}_r - \alpha_r B_r - \omega_B B_i - \gamma_{7,r} B_r |A|^2 + \gamma_{7,i} B_i |A|^2 - \gamma_{8,r} B_r |B|^2 + \gamma_{8,i} B_i |B|^2, \\
			\varepsilon^2 (4 \dot{\tilde{B}}_i + 3 c_v \dot{\tilde{B}}_r) &= (c_g - c_0) \tilde{B}_i + \omega_B B_r - \alpha_v B_i - \gamma_{7,r} B_i |A|^2 - \gamma_{7,i} B_r |A|^2 - \gamma_{8,r} B_i |B|^2 - \gamma_{8,i} B_r |B|^2.
		\end{split}
	\end{equation}
	Note that for notational convenience, we write $|A|^2 = A_r^2 + A_i^2$ and $|B|^2 = B_r^2 + B_i^2$.
	Finally, we can obtain equations for $\dot{\tilde{B}}_r$ and $\dot{\tilde{B}}_i$ since
	\begin{equation*}
		\begin{pmatrix}
			4 & -3c_v \\ 3 c_v & 4
		\end{pmatrix}
	\end{equation*}
	with determinant $16 + 9 c_v^2 \neq 0$ is invertible. Therefore, we obtain a fast-slow system of the form
	\begin{equation}\label{eq:full-fast-slow-system-abbrev}
		\begin{split}
			\dot{Y} &= X, \\
			\varepsilon^2 \dot{X} &= H_0(X,Y)
		\end{split}
	\end{equation}
	with $Y = (A_r, A_i, B_r, B_i)$ and $X = (\tilde{A}_r, \tilde{A}_i, \tilde{B}_r, \tilde{B}_i)$.
	
	The critical manifold $C_0$ of the fast-slow system \eqref{eq:full-fast-slow-system-abbrev} is then given by
	\begin{equation*}
		C_0 := \{(X,Y) \in \R^8 \,:\, H_0(X,Y) = 0\}.
	\end{equation*}
	Next, we calculate
	\begin{equation*}
		D_X H_0 (X,Y) \vert_{(X,Y) \in C_0} = \begin{pmatrix}
			\Acal & 0 \\ 0 & \Bcal
		\end{pmatrix}
	\end{equation*}
	with matrices
	\begin{equation*}
		\Acal = \begin{pmatrix}
			-\tfrac{c_0}{4} & 0 \\
			0 & - \tfrac{c_0}{4}
		\end{pmatrix}, \qquad \Bcal = \begin{pmatrix}
			4 & -3 c_v \\ 3c_v & 4
		\end{pmatrix}^{-1} \begin{pmatrix}
			c_g - c_0 & 0 \\ 0 & c_g - c_0
		\end{pmatrix}.
	\end{equation*}
	The eigenvalues of $\Bcal$ are given by $\tfrac{i (c_g - c_0)}{4 i \pm 3 c_v}$ with real part $-\tfrac{4 (c_g - c_0)}{16 + 9c_v^2}$. Hence, if $c_0 \neq 0$ and $c_g - c_0 \neq 0$ the critical manifold $C_0$ is normally hyperbolic. Therefore, any compact submanifold $S_0$ of the critical manifold $C_0$ perturbs to a slow manifold $S_\varepsilon$ for $0 <  \varepsilon \ll 1$, see \cite{fenichel1979}.  In addition, the flow on the critical manifold is given by
	\begin{equation}\label{eq:slow-flow}
		\begin{split}
			c_0 \dot{A}_r &= - \alpha_u A_r - \omega_A A_i - \gamma_{1,r} A_r |A|^2 + \gamma_{1,i} A_i |A|^2 - \gamma_{2,r} A_r |B|^2 + \gamma_{2,i} A_i |B|^2, \\
			c_0 \dot{A}_i &= \omega_A A_r - \alpha_u A_i - \gamma_{1,r} A_i |A|^2 - \gamma_{1,i} A_r |A|^2 - \gamma_{2,r} A_i |B|^2 - \gamma_{2,i} A_r |B|^2, \\
			-(c_g - c_0) \dot{B}_r &= - \alpha_v B_r - \omega_B B_i - \gamma_{7,r} B_r |A|^2 + \gamma_{7,i} B_i |A|^2 - \gamma_{8,r} B_r |B|^2 + \gamma_{8,i} B_i |B|^2, \\
			- (c_g - c_0) \dot{B}_i &= \omega_B B_r - \alpha_v B_i - \gamma_{7,r} B_r |A|^2 - \gamma_{7,i} B_i |A|^2 - \gamma_{8,r} B_r |B|^2 - \gamma_{8,i} B_i |B|^2,
		\end{split}.
	\end{equation}
	Note that any solution to \eqref{eq:slow-flow} gives a solution to \eqref{eq:radii-dynamics} with $(r_A,r_B) = (|A|,|B|)$. In particular, we have that
	\begin{itemize}
		\item the trivial equilibrium point $T$ in \eqref{eq:radii-dynamics} corresponds to the trivial equilibrium point $(A_r, A_i, B_r, B_i) = (0,0,0,0)$;
		\item the equilibrium points $ST_A$ and $ST_B$ in \eqref{eq:radii-dynamics} correspond to the circles of steady states
		\begin{equation*}
			\begin{split}
				(A_r, A_i, B_r, B_i) &= \left(\sqrt{-\dfrac{\alpha_u}{\gamma_{1,r}}} \cos(\theta_A), \sqrt{-\dfrac{\alpha_u}{\gamma_{1,r}}} \sin(\theta_A), 0,0\right), \\
				(A_r, A_i, B_r, B_i) &= \left(0,0, \sqrt{-\dfrac{\alpha_v}{\gamma_{8,r}}} \cos(\theta_B), \sqrt{-\dfrac{\alpha_v}{\gamma_{8,r}}} \sin(\theta_B)\right)
			\end{split}
		\end{equation*}
		with fixed phases $\theta_A, \theta_B \in [0,2\pi)$;
		\item the equilibrium point $NT$ in \eqref{eq:radii-dynamics} corresponds to the family of time-periodic orbits
		\begin{equation*}
			(A_r, A_i, B_r, B_i) = (NT_A \cos(\psi_A T + \theta_A), NT_A \sin(\psi_A T + \theta_A), NT_B \cos(\psi_B T + \theta_B), NT_B \sin(\psi_B T + \theta_B))
		\end{equation*}
		with fixed phase shifts $\theta_A, \theta_B \in [0,2\pi)$ and local wave numbers $\psi_A, \psi_B$ given as solutions to \eqref{eq:slow-dynamics}.
	\end{itemize}
	These structures also persist in the full fast-slow system \eqref{eq:full-fast-slow-system-abbrev} and, in particular, are contained in all slow manifolds as the following result shows.
	
	\begin{lemma}\label{lem:fixed-points-fast-slow-system}
		The fast-slow system \eqref{eq:full-fast-slow-system-abbrev} has a trivial steady state $\Ical_T = 0 \in \R^8$. Assuming that the radii dynamics \eqref{eq:radii-dynamics} has equilibrium points $ST_A$ and $ST_B$, the fast-slow system has two one-dimensional, invariant circles of equilibrium points
		\begin{equation*}
			\begin{split}
				\Ical_{ST_A} &= \left\{(0,Y) \in \R^8 \,:\, |A| = \sqrt{-\tfrac{\alpha_u}{\gamma_{1,r}}}, |B| = 0,\right\}, \\
				\Ical_{ST_B} &= \left\{(0,Y) \in \R^8 \,:\, |A| = 0, |B| = \sqrt{-\tfrac{\alpha_v}{\gamma_{8,r}}}\right\}.
			\end{split}
		\end{equation*}
		Finally, assume that $c_0 \neq 0$, $c_g - c_0 \neq 0$ and the fully nontrivial equilibrium point $NT$ exists in \eqref{eq:radii-dynamics} and is hyperbolic. Then there exists an $\varepsilon_0 > 0$ such that for every $\varepsilon \in (0,\varepsilon_0)$ the fast-slows system has a two-dimensional invariant set $\Ical_{NT,\varepsilon}$ filled with periodic orbits.
		
		The invariant sets $\Ical_0$, $\Ical_{ST_A}$, $\Ical_{ST_B}$ and $\Ical_{NT,\varepsilon}$ lie on all slow manifolds $S_\varepsilon$ for $0 < \varepsilon < \varepsilon_0$.
	\end{lemma}
	\begin{proof}
		We prove that the structures persist in the full fast-slow system \eqref{eq:full-fast-slow-system-abbrev} and remain close to the critical manifold. Since the structures are bounded, they lie in a small neighbourhood of a compact submanifold of the critical manifold. Additionally, this compact submanifold is uniformly hyperbolic, that is, $D_X H_0(X,Y)$ has a uniform spectral gap about the imaginary axis for $(X,Y)$ in the compact submanifold. Therefore, using that slow manifolds are centre manifolds of the fast system extended with $\varepsilon' = 0$, we can lift the typical statement in centre manifold theory that solutions, which remain in a small neighbourhood for all times, lie on all centre manifolds, see e.g.~\cite{sijbrand1985}, by using a finite covering of the compact submanifold. Hence, if the structures perturb in the full system and remain bounded, they lie on all slow manifolds. In particular, this applies to equilibrium points and periodic orbits bifurcating from the critical manifold.
		
		For the persistence, we note that the steady states of the slow flow \eqref{eq:slow-flow} and the full fast-slow system \eqref{eq:full-fast-slow-system-abbrev} are identical. Therefore, the only remaining non-trival part of the lemma is the existence of $\Ical_{NT,\varepsilon}$. Since at $NT$ both $r_A \neq 0$ and $r_B \neq 0$ we can rewrite \eqref{eq:fast-fronts-eucledean-coords} into polar coordinates to obtain \eqref{eq:fast-travelling-fronts-ode-full}. We now look for solutions with constant radii $r_A, r_B$ and constant local wave numbers $\psi_A, \psi_B$. These satisfy
		\begin{equation*}
			\begin{split}
				0 &= \alpha_u r_A + \gamma_{1,r} r_A^3 + \gamma_{2,r} r_A r_B^2 - 4 \varepsilon^2 r_A \psi_A^2, \\
				0 &= c_0 \psi_A - \omega_A + \gamma_{1,i} r_A^2 + \gamma_{2,i} r_B^2, \\
				0 &= \alpha_v r_B + \gamma_{7,r} r_B r_A^2 + \gamma_{8,r} r_B^3 - 4 \varepsilon^2 r_B \psi_B^2, \\
				0 &= (c_g - c_0) \psi_B - \omega_B + \gamma_{7,i} r_A^2 + \gamma_{8,i} r_B^2 - 3 \varepsilon^2 c_v \psi_B^2.
			\end{split}
		\end{equation*}
		For $\varepsilon = 0$, there is a solution since we assume that the fully non-trivial equilibrium point $NT$ in the radii dynamics \eqref{eq:radii-dynamics} exists. Linearising about this point and setting $\varepsilon = 0$ yields a lower triangular block matrix, where the upper diagonal block is a $2\times2$ matrix corresponding to the linearisation about $NT$ in the radii dynamics \eqref{eq:radii-dynamics} and the lower diagonal block is a diagonal $2\times 2$ matrix with $c_0$ and $c_g - c_0$ on the diagonal. By assumption $NT$ is a hyperbolic equilibrium point in \eqref{eq:radii-dynamics} and $c_0 \neq 0$ and $c_g - c_0 \neq 0$. Therefore, the persistence of the periodic orbit follows by the implicit function theorem. Furthermore, since the phases of $A$ and $B$ in the periodic orbit are only given up to additive constants, the periodic orbit appear as a two-parameter family, which all lie on a two-dimensional sphere in $\R^4$.
	\end{proof}
	
	Given the persistence of the equilibrium points in the radii dynamics \eqref{eq:radii-dynamics} to the fast-slow system \eqref{eq:full-fast-slow-system-abbrev}, we can now discuss the persistence of the heteroclinic orbits in the radii dynamics. For this, we observe that the heteroclinic orbits obtained in \eqref{eq:radii-dynamics} or equivalently \eqref{eq:radii-dynamics-rescaled} can be classified into three disjoint groups:
	\begin{enumerate}[label=(\alph*)]
		\item Heteroclinic orbits on the invariant coordinate axes $\{r_A = 0\}$ and $\{r_B = 0\}$;\label{het-orbit-case1}
		\item Heteroclinic orbits connecting an unstable equilibrium point to a stable equilibrium point, a saddle to a stable equilibrium point or an unstable equilibrium point to a saddle;\label{het-orbit-case2}
		\item The heteroclinic orbit for $\tilde{d} < 0$ and $\tilde{c} = c^*$, see Figure \ref{fig:phase-planes-d<0}, connecting two saddles. \label{het-orbit-case3}
	\end{enumerate}
	We now show that the orbits in cases \ref{het-orbit-case1} and \ref{het-orbit-case2} generally persist in the fast-slow system \eqref{eq:full-fast-slow-system-abbrev} and thus result in fast-moving front solutions to the amplitude equations \eqref{eq:final-amplitude-eq}, which connect space-time periodic solutions, see Figure \ref{fig:amplitude-front}.
	
	\begin{figure}
		\includegraphics[width=0.45\textwidth]{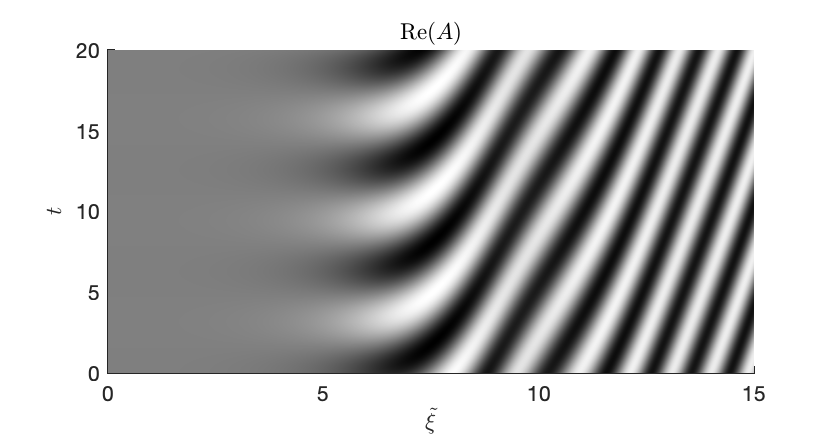} 
		\hspace{1cm}
		\includegraphics[width=0.45\textwidth]{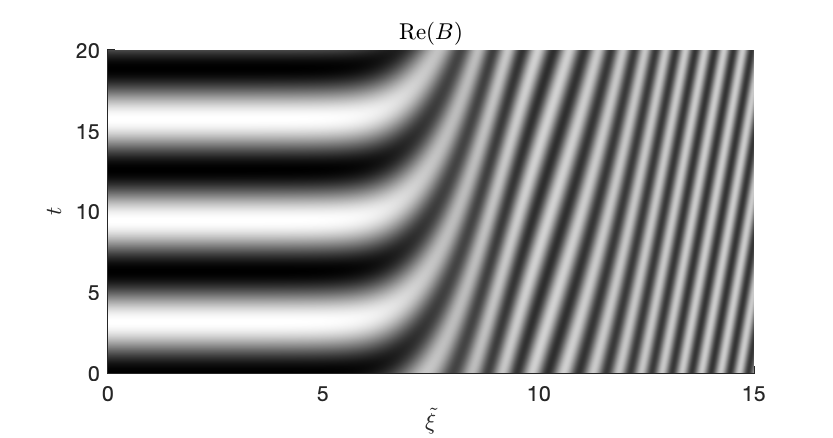} 
		\caption{Numerical simulation of a fast-moving front in the amplitude equations \eqref{eq:final-amplitude-eq}. Here, the front connects a semi-trivial, time-periodic solution with $A = 0$ to a space-time periodic solution. Since the solutions are plotted in the co-moving frame $\tilde{\xi}$, the front interface appears stationary.}
        \label{fig:amplitude-front}
	\end{figure}
	
	\begin{theorem}\label{thm:persistence-heteroclinic-orbits}
		Let $c_0 \neq 0$ and $c_g - c_0 \neq 0$. Then, there exists an $\varepsilon_0 > 0$, depending on the coefficients in \eqref{eq:final-amplitude-eq}, such that for any $\varepsilon \in (0,\varepsilon_0)$ all heteroclinic orbits in the classes \ref{het-orbit-case1} and \ref{het-orbit-case2} persist in the fast-slow system \eqref{eq:full-fast-slow-system-abbrev}.
	\end{theorem}
	\begin{proof}
		We start by showing the persistence of orbits from class \ref{het-orbit-case2}. Take any heteroclinic orbit in \eqref{eq:radii-dynamics-rescaled} from this class. Since the corresponding orbit in \eqref{eq:slow-flow} is bounded in $\R^4$, there is a compact submanifold $S_0$ of $C_0$, which contains this orbit. As $C_0$ is normally hyperbolic, we can apply Fenichel's theorem \cite{fenichel1979}, see also \cite[Thm.~3.1.4]{kuehn2015}, to obtain that this compact submanifold of $C_0$ perturbs to a manifold $S_\varepsilon$, which is diffeomorphic to $S_0$ and locally invariant under the flow of the fast-slow system \eqref{eq:full-fast-slow-system-abbrev}. Additionally, the flow on $S_\varepsilon$ converges to the slow flow \eqref{eq:slow-flow} as $\varepsilon \rightarrow 0$. Therefore, the vector field \eqref{eq:slow-flow} on the critical manifold perturbs to a vector field on $S_\varepsilon$.
		
		It remains to show that heteroclinic orbits in class \ref{het-orbit-case2} are structurally stable, that is, they persist under higher-order perturbations of the vector field. By assumption, heteroclinic orbits in class \ref{het-orbit-case2} arise from an intersection of a unstable and a stable manifold, where at least one manifold is two-dimensional. Since the phase space of the radii dynamics \eqref{eq:radii-dynamics} is two-dimensional, the intersection is transversal and therefore structurally stable. We now show that a similar argument also applies if the heteroclinic orbit is lifted to a corresponding orbit in the slow flow \eqref{eq:slow-flow}. For notational convenience, we restrict the argument to a specific type of heteroclinic orbit in \eqref{eq:radii-dynamics}, which connects $ST_A$ to $NT$ and that $ST_A$ is an unstable equilibrium point while $NT$ is a saddle. The persistence of other heteroclinic orbits can be obtained with the same arguments, the key property is that the intersection of stable and unstable manifolds in the radii dynamics is robust under perturbation.
		
		As discussed above, on the critical manifold $C_0$, the equilibrium point $ST_A$ becomes a one-dimensional circle of equilibrium points, where $B = 0$ and $A$ has constant modulus and the circle is parameterised through the phase of $A$. Therefore, the linearisation about any equilibrium point on the circle has a zero eigenvalue, which can be obtained by differentiating the family of steady states respect to the phase parameter. With the neutral direction tangential to the circle now accounted for, the remaining eigenvalues are given by the radial behaviour in the normal directions of the invariant set. Specifically, since $ST_A$ is unstable in the radii dynamics, we obtain that the remaining three eigenvalues must have positive real part. Therefore, any equilibrium point on the circle on the critical manifold corresponding to $ST_A$ has a four-dimensional centre-unstable manifold. 
		
		As shown in Lemma \ref{lem:fixed-points-fast-slow-system}, the circle of equilibrium points on the critical manifold perturbs to a one-dimensional circle of equilibrium points $\Ical_{ST_A}$, which lies on the slow manifold $\Scal_\varepsilon$. Since the flow on $\Scal_\varepsilon$ is $\varepsilon$ close to the flow on the critical manifold linearisation of the perturbed slow flow on $\Scal_\varepsilon$ about any equilibrium point on $\Ical_{ST_A}$ still has at least a three-dimensional unstable manifold. In addition, the solutions on $\Ical_{ST_A}$ are still parametrised by a phase shift of $A$. Therefore, we also obtain that the zero eigenvalue persists. Hence, the linearisation still has a four-dimensional centre-unstable manifold.
		
		We now apply a similar argument to show that the two-dimensional invariant manifold $\Ical_{NT,\varepsilon}$ on $\Scal_\varepsilon$ corresponding to $NT$ in the fast-slow system \eqref{eq:full-fast-slow-system-abbrev} has a three-dimensional centre-stable manifold with a two-dimensional centre direction. The argument relies on the same ideas. In particular, the two-dimensional centre direction follows from the fact that $\Ical_{NT,\varepsilon}$ is parametrised by two independent parameters for the phase shift of $A$ and $B$, respectively. In addition, the fact that the stable direction is one-dimensional follows from the assumption that $NT$ is a saddle in the radii dynamics and that the eigenvalue with negative real part persists under perturbation.
		
		Since there is a heteroclinic orbit $ST_A \longrightarrow NT$ in the radii dynamics, we obtain a heteroclinic orbit from $\Ical_{ST_A}$ to $\Ical_{NT,0}$. Hence, the four-dimensional centre-unstable manifold of $\Ical_{ST_A}$ and the three-dimensional centre-stable manifold of $\Ical_{NT,\varepsilon}$ intersect for $\varepsilon = 0$. Since the critical manifold is four-dimensional, this intersection must be transversal and therefore, persists under perturbation, using also that the centre directions persist. This establishes a heteroclinic orbit in $\Scal_\varepsilon$ connecting $\Ical_{ST_A}$ to $\Ical_{NT,\varepsilon}$.
		
		As mentioned above, the persistence of other orbits in class \ref{het-orbit-case2} follows with the same arguments. Using that $\{(A_r, A_i, \tilde{A}_r, \tilde{A}_i) = 0\}$ and $\{(B_r, B_i, \tilde{B}_r, \tilde{B}_i) = 0\}$ are invariant in the fast-slow system \eqref{eq:full-fast-slow-system-abbrev}, the strategy can be adapted in a straightforward manner to also show the persistence of heteroclinic orbits in class \ref{het-orbit-case1}, which lie on the invariant sets $\{r_A = 0\}$ or $\{r_B = 0\}$ in \eqref{eq:radii-dynamics}. This concludes the proof.
	\end{proof}	
	
	\begin{figure}
		\includegraphics[width=0.5\textwidth]{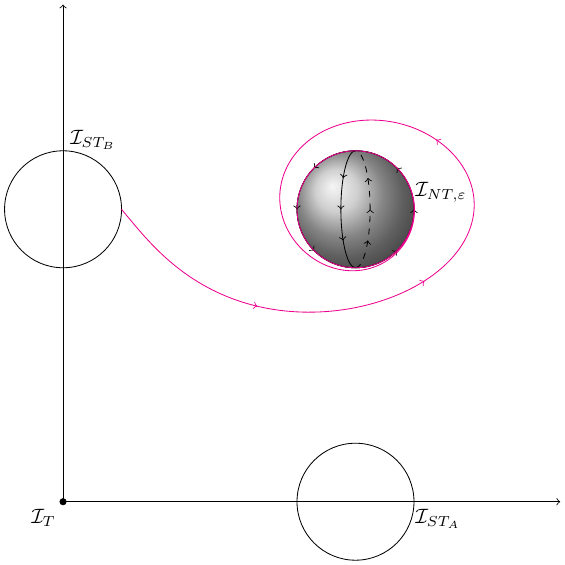}
		\caption{Schematic depiction of the phase space dynamics in the fast-slow system \eqref{eq:full-fast-slow-system-abbrev} displaying the equilibrium point $\Tcal$, the circles of equilibrium points $\Ical_{ST_A}$ and $\Ical_{ST_B}$, as well as the two-dimensional ball of periodic orbits $\Ical_{NT,\varepsilon}$. In addition, the heteroclinic orbit from $\Ical_{ST_A}$ to $\Ical_{NT,\varepsilon}$ constructed in the proof of Theorem \ref{thm:persistence-heteroclinic-orbits} is shown.}
	\end{figure}
	
	\section{Proofs of the existence of heteroclinic orbits}\label{sec:proof-of-het-orbits}
	
	We now rigorously prove that the existence of some numerically observed orbits in the system \eqref{eq:radii-dynamics-rescaled}. Although some orbits cannot be constructed through standard analytical methods, we expect that they can be obtained for example through rigorous numerical methods, see e.g.~\cite{berg2018}.
	
	\subsection{Heteroclinic orbits for $d < 0$}
	We now discuss the dynamics of \eqref{eq:radii-dynamics-rescaled} rigorously in the case $d < 0$, starting with the following result for $\tilde{c} < 0$.
	
	\begin{lemma}\label{lem:het-orbits-d<0-c<0}
		Let $\tilde{\gamma}_A < -1$, $\tilde{\gamma}_B < -1$ and $\tilde{c} < 0$. Then, there are heteroclinic orbits $ST_A \longrightarrow NT$, $ST_B \longrightarrow NT$ and $NT \longrightarrow T$.
	\end{lemma}
	\begin{proof}
		We first recall from Proposition \ref{prop:stability-NT} that $NT$ is a saddle for $\tilde{\gamma}_A < -1$, $\tilde{\gamma}_B < -1$ and $\tilde{c} < 0$. Next, we note that since $\tilde{\gamma}_A < 0$, $\tilde{\gamma}_B < 0$ and $\tilde{c} < 0$, the functions $F_A(r_A,r_B)$ and $F_B(r_A,r_B)$, see \eqref{eq:radii-dynamics-rescaled}, are monotonically increasing in $r_A$ and $r_B$, respectively. Additionally, we explicitly obtain the nullclines $\ell_A := \{F_A = 0\} = \{r_A = \sqrt{1+\tilde{\gamma}_A r_B^2}\}$ and $\ell_B := \{F_B = 0\} = \{r_B = \sqrt{1+\tilde{\gamma}_B r_A^2}\}$. These one-dimensional curves can be written as the graph of concave functions over $r_B$ and $r_A$, respectively. Moreover, they intersect at $NT$. From this, we can decompose the phase space $\{(r_A,r_B) \in [0,\infty)^2\}$ into 
		\begin{equation*}
			\begin{split}
				\Omega_{--} &= \{(r_A, r_B) \,:\, F_A(r_A,r_B) < 0 \text{ and } F_B(r_A,r_B) < 0\}, \\
				\Omega_{+-} &= \{(r_A,r_B) \,:\, F_A(r_A,r_B) F_B(r_A,r_B) < 0\}, \\
				\Omega_{++} &= \{(r_A, r_B) \,:\, F_A(r_A,r_B) > 0 \text{ and } F_B(r_A,r_B) > 0\},
			\end{split}
		\end{equation*}
		see Figure \ref{fig:phase-plane-decomp}. We note that the intersection of the boundaries of these sets contains only the equilibrium point $NT$. Additionally, the set $\Omega_{+-}$ decomposes into two connected components
		\begin{equation*}
			\Omega_{+-}^1 := \Omega_{+-} \cap \{r_A < NT_1 \text{ and } r_B > NT_2\}, \quad \Omega_{+-}^2 := \Omega_{+-} \cap \{r_A > NT_1 \text{ and } r_B < NT_2\},
		\end{equation*}
		see Figure \ref{fig:phase-plane-decomp}. Since $\tilde{\gamma}_A < 0$ and $\tilde{\gamma}_B < 0$ and due to the monotonicity of the vector field in \eqref{eq:radii-dynamics-rescaled}, these sets are non-empty and arranged as in Figure \ref{fig:phase-plane-decomp}.
		
		We now show that $\Omega_{--}$ and $\Omega_{++}$ are inflowing invariant manifolds and $\Omega_{+-}^1$ and $\Omega_{+-}^2$ are overflowing ones. First, consider the interface of $\Omega_{--}$ and $\Omega_{+-}^2$, which is a subset of $\{F_B = 0\}$ and thus a graph of a concave function in $r_A$. Additionally, since $F_A(r_A,r_B) < 0$ for all $(r_A, r_B)$ on the interface, the vector field flows from $\Omega_{+-}^2$ to $\Omega_{--}$. Using the same arguments for the other interfaces establishes the invariance properties of the sets.
		
		Since $NT$ is a saddle, it has one-dimensional stable and unstable manifolds. From the flow of the vector field in the different regions we can conclude that the unstable manifold lies in $\Omega_{++} \cup \Omega_{--}$ and the stable manifold is contained in $\Omega_{+-}$. Since $\Omega_{--}$ is a compact, inflowing invariant manifold, where both $r_A$ and $r_B$ are monotonically decreasing along every orbit, it is easy to see that any orbit starting in $\Omega_{--}$ tends to $T$ and is therefore contained in the two-dimensional stable manifold of $T$. Hence, the unstable manifold of $NT$ and the unstable manifold of $T$ intersect and there is a heteroclinic orbit $NT \longrightarrow T$. Note that in particular, this intersection is transversal since the phase space is two-dimensional.
		
		To establish the heteroclinic orbits $ST_A \longrightarrow NT$ and $ST_B \longrightarrow NT$, we consider the time-reversed system, where $\Omega_{+-}^1$ and $\Omega_{+-}^2$ are inflowing invariant manifolds. Due to the sign condition on the components of the vector field, there is no periodic orbit in either $\Omega_{+-}^1$ or $\Omega_{+-}^2$. Therefore, since both $ST_A$ and $ST_B$ are stable in the time-reversed system, the existence of heteroclinic orbits from $NT$ to $ST_A$ and $ST_B$, respectively, follows from \cite[Thm.~6.41]{meiss2017}. Returning to the original flow of time then concludes the proof.
	\end{proof}
	
	\begin{figure}
        \begin{minipage}[t]{0.47\textwidth}
            \includegraphics[width=\textwidth]{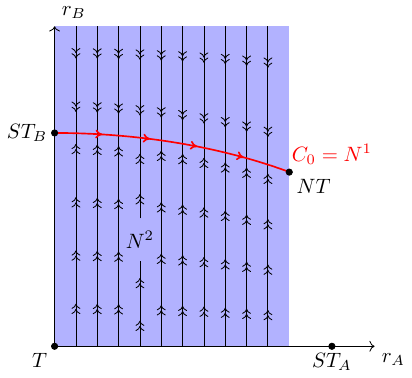}
            \caption{Fast-slow structure of the vector field for $\tilde{c} \ll 1$, as well as the sets $N^1$ and $N^2$ used in the proof of Lemma \ref{lem:het-orbit-d<0-c<<1}.}
            \label{fig:fast-slow-radii-c<<1}
        \end{minipage}
        \qquad
        \begin{minipage}[t]{0.47\textwidth}
    		\includegraphics[width=\textwidth]{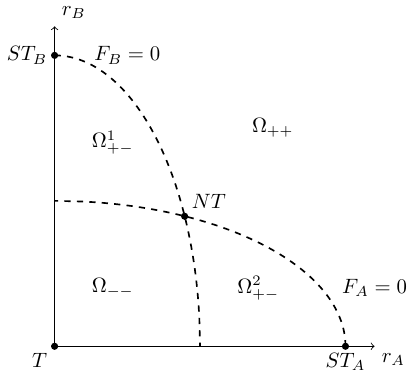}
    		\caption{Decomposition of phase plane in the case $d < 0$ and $\tilde{c} < 0$ used in the proof of Lemma \ref{lem:het-orbits-d<0-c<0}. The dashed lines mark the nullclines of the vector field \eqref{eq:radii-dynamics-rescaled}.}
    		\label{fig:phase-plane-decomp}
        \end{minipage}
	\end{figure}

	As $\tilde{c}$ transitions through zero, the equilibrium point $NT$ becomes stable, see Proposition \ref{prop:stability-NT}. In addition, $T$, $ST_A$ and $ST_B$ become saddles. This results in a break-up of the heteroclinic connection $ST_A \longrightarrow NT$ and $NT \longrightarrow T$. In contrast, the connection $ST_B \longrightarrow NT$ still exists as the following lemma shows.
	
	\begin{lemma}\label{lem:het-orbit-d<0-c<<1}
		Let $\tilde{\gamma}_A < -1$, $\tilde{\gamma}_B < -1$. Then there exists a $\tilde{c}_0 > 0$ such that for all $0 < \tilde{c} < \tilde{c}_0$, the system \eqref{eq:radii-dynamics-rescaled} has a heteroclinic orbit $ST_B \longrightarrow NT$.
	\end{lemma}
	\begin{proof}
		For $0 < \tilde{c} \ll 1$, the ODE system \eqref{eq:radii-dynamics-rescaled} has a fast-slow structure, which allows us to establish the existence of a heteroclinic connection $ST_B \longrightarrow NT$ by using geometric singular perturbation theory, specifically \cite[Thm.~6.1.1]{kuehn2015}, see also \cite{szmolyan1991}. The geometric situation is displayed in Figure \ref{fig:fast-slow-radii-c<<1}.
		
		First, we consider the slow subsystem given by 
		\begin{equation*}
			\begin{split}
				\dot{r}_A &= -r_A (1 - r_A^2 + \tilde{\gamma}_A r_B^2), \\
				0 &= r_B (1 - r_B^2 + \tilde{\gamma}_B r_A^2).
			\end{split}
		\end{equation*}
		Therefore, the critical manifold is given by $\{r_B = 0\} \cup \{r_B^2 = 1 + \tilde{\gamma}_B r_A^2\}$. For this proof, we consider the subset
		\begin{equation*}
			C_0 := \left\{(r_A,r_B) \in [0,\infty)^2 \,:\, r_B^2 = 1 + \tilde{\gamma}_B r_A^2 \text{ and } r_A^2 \in \left[0, \frac{1+\tilde{\gamma}_A}{1 - \tilde{\gamma}_A \tilde{\gamma}_B}\right]\right\},
		\end{equation*}
		which is a one-dimensional curve with endpoints $ST_B$ and $NT$. The dynamics on $C_0$ are given by
		\begin{equation*}
			\dot{r}_A = -(1+\tilde{\gamma}_A) r_A + (1-\tilde{\gamma}_A\tilde{\gamma}_B) r_A^3.
		\end{equation*}
		Using that $1+\tilde{\gamma}_A < 0$ and $1 - \tilde{\gamma}_A \tilde{\gamma}_B < 0$, this equation has a heteroclinic orbit from $r_A  = 0$ to $r_A = \sqrt{\tfrac{1+\tilde{\gamma}_A}{1 - \tilde{\gamma}_A \tilde{\gamma}_B}}$. This corresponds to a heteroclinic orbit $ST_B \longrightarrow NT$ in the slow subsystem. Moreover, we find that $W^u(ST_B) = W^s(NT) = C_0$, where $W^u(ST_B)$ denotes the unstable manifold of $ST_B$ and $W^s(NT)$ denotes the stable manifold of $NT$ both with respect to the slow subsystem on $C_0$.
		
		Next, we consider the fast subsystem
		\begin{equation*}
			\begin{split}
				r_A' &= 0 \\
				r_B' &= r_B (1 - r_B^2 + \tilde{\gamma}_B r_A^2).
			\end{split}
		\end{equation*}
		Linearising the $r_B$-equation about any point on $C_0$, we find the linear equation $R_B' = -2 R_B$ and therefore, $C_0$ is attractive with respect to the fast flow. In particular, this yields that the manifolds
		\begin{equation*}
			\begin{split}
				N^1 &:= \bigcup_{p \in C_0} \left\{q \in (0,\infty)^2 \,:\, \phi^\text{fast}_t(q) \rightarrow p, \text{ for } t \rightarrow -\infty\right\} = C_0 \\
				N^2 &:= \bigcup_{p \in C_0} \left\{q \in (0,\infty)^2 \,:\, \phi^\text{fast}_t(q) \rightarrow p, \text{ for } t \rightarrow -\infty\right\} = \left(0, \sqrt{\frac{1+\tilde{\gamma}_A}{1 - \tilde{\gamma}_A \tilde{\gamma}_B}}\right) \times (0,\infty)
			\end{split}
		\end{equation*}
		intersect transversally and therefore, the conditions of \cite[Thm.~6.1.1]{kuehn2015} are satisfied. Hence, the heteroclinic orbit $ST_B \longrightarrow NT$ in the slow subsystem persists for $0 < \tilde{c} \ll 1$, which proves the desired statement.
	\end{proof}
	
	\begin{remark}
		The heteroclinic orbit $NT \longrightarrow ST_A$ observed for $\tilde{c} \ll 1$ can also be established using geometric singular perturbation theory similar to Lemma \ref{lem:het-orbit-d<0-c<<1}.
	\end{remark}

	\subsection{Heteroclinic orbits for $d > 0$}
	
	Similar to the previous case $d < 0$, we now rigorously establish the heteroclinic orbits, which were numerically conjectured, see also Figure \ref{fig:phase-plane-d>0}. We first consider the case that $\tilde{c} > 0$, which we can understand fully analytically.
	
	\begin{figure}
		\includegraphics[width=0.4\textwidth]{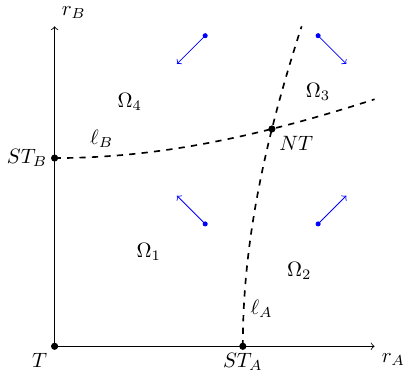}
		\hspace{1cm}
		\includegraphics[width=0.4\textwidth]{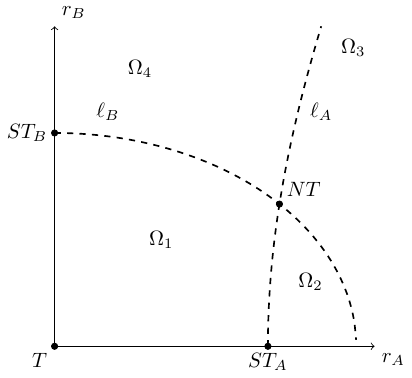}
		
		\includegraphics[width=0.4\textwidth]{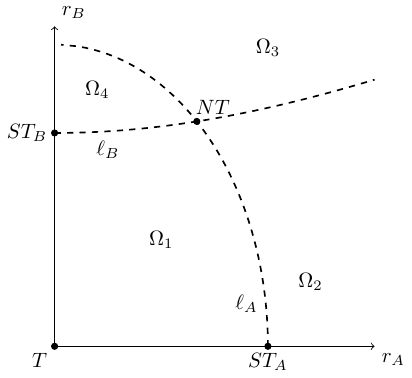}
		\hspace{1cm}
		\includegraphics[width=0.4\textwidth]{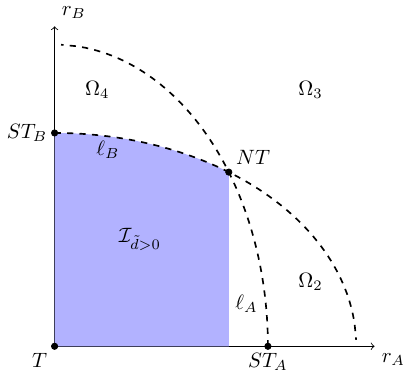}
		\caption{Qualitatively different decompositions of the phase space by the nullclines, depicted as dashed lines, in the case $\tilde{d} > 0$. From top left to bottom right: case $\tilde{\gamma}_A > 0$, $\tilde{\gamma}_B > 0$; case $\tilde{\gamma}_A > 0$, $\tilde{\gamma}_B < 0$; case $\tilde{\gamma}_A < 0$, $\tilde{\gamma}_B > 0$; and case $\tilde{\gamma}_A < 0$, $\tilde{\gamma}_B < 0$. The qualitative behaviour of the flow for $\tilde{c} > 0$ in the different subsets is indicated by blue arrows in the first figure. Note that, the qualitative flow is similar in all other cases. The last plot of the case $\tilde{\gamma}_A < 0$ and $\tilde{\gamma}_B < 0$ also shows the inflowing invariant set $\Ical_{\tilde{d} > 0}$ constructed in the proof of Lemma \ref{lem:het-orbits-d>0}.}
		\label{fig:phase-plane-decomp-d>0-c>0}
	\end{figure}

	\begin{lemma}\label{lem:het-orbits-d>0}
		Let $\tilde{c} > 0$ and $\tilde{\gamma}_A$, $\tilde{\gamma}_B$ such that \eqref{eq:non-triv-fp-invertibility}--\eqref{eq:non-triv-fp-pos2} are satisfied with $\tilde{d} > 0$. Then, the system \eqref{eq:radii-dynamics-rescaled} has heteroclinic orbits $ST_A \longrightarrow NT$ and $NT \longrightarrow ST_B$. Together with the invariant coordinate axes $\{r_A = 0\}$ and $\{r_B = 0\}$, these orbits form the boundary of a set, which is filled with heteroclinic orbits $ST_A \longrightarrow ST_B$.
	\end{lemma}
	\begin{proof}
		The proof works similar to the proof of Lemma \ref{lem:het-orbits-d<0-c<0}, that is, we decompose the phase plane into four subset $\Omega_j$, $j = 1,2,3,4$ separated by the nullclines $\ell_A$ and $\ell_B$ of the vector field \eqref{eq:radii-dynamics-rescaled}. Since for $\tilde{d} > 0$, the cross-coefficients $\tilde{\gamma}_A$ and $\tilde{\gamma}_B$ do not have a determined sign, four scenarios can occur, see Figure \ref{fig:phase-plane-decomp-d>0-c>0}. In all cases, the strategy is the same: we first establish the heteroclinic orbits $NT \longrightarrow ST_B$ and $ST_A \longrightarrow NT$. That these define the borders of a subset filled with heteroclinic orbits $ST_A \longrightarrow ST_B$ then follows from the observation that $NT$ and $T$ are both saddles and $ST_A$ is unstable, while $ST_B$ is stable. Furthermore, from the monotonicity of the vector field, we find that there are no periodic orbits in the compact invariant set bordered by $NT \rightarrow ST_B$, $ST_A \longrightarrow NT$ and the coordinate axes.
		
		Since the construction of the heteroclinic orbits to and from $NT$ works similarly in all cases, we only give the construction of the orbit $NT \longrightarrow ST_B$ in the case $\tilde{\gamma}_A, \tilde{\gamma}_B < 0$ in detail. We define the set $\Ical_{\tilde{d} > 0} := \{(r_A, r_B) : r_A < NT_A\} \cap (\Omega_1 \cup \Omega_4)$, where $NT = (NT_A, NT_B)$, see Figure \ref{fig:phase-plane-decomp-d>0-c>0}. Since $\dot{r}_A < 0$, $\dot{r}_B > 0$ in $\Omega_1$ and $\dot{r}_A < 0$, $\dot{r}_B < 0$ in $\Omega_4$, this set is inflowing invariant. Since $T$ is a saddle with a one-dimensional stable manifold given explicitly given by $\{r_B = 0\}$, and $\dot{r}_A < 0$ in $\Ical_{\tilde{d} > 0}$, the only attracting equilibrium point in $\Ical_{\tilde{d} > 0}$ is $ST_B$. In addition, since any orbit in $\Ical_{\tilde{d}> 0}$ satisfies $\dot{r}_A < 0$, $\Ical_{\tilde{d}> 0}$ cannot contain periodic orbits.
		
		It remains to show that $\Wcal^u(NT)$, the one-dimensional the unstable manifold of $NT$, intersects with $\Ical_{\tilde{d} > 0}$. Assume that this is not the case. Since $\Wcal^u(NT)$ is tangential to the unstable eigenspace at $NT$, it must either intersect with $\Omega_2$ or $\Omega_1 \cap \{r_A \geq NT_A\}$. However, any orbit in these sets satisfies $r_B < NT_B$ and $\dot{r}_B > 0$. Therefore, no orbit can converge to $NT$ as $\tilde{\xi} \rightarrow -\infty$ and therefore, $\Wcal^u(NT)$ cannot intersect $\Omega_2$ or $\Omega_1 \cap \{r_A \geq NT_A\}$. This yields a contradiction and therefore, $\Wcal^u(NT) \cap \Ical_{\tilde{d} > 0} \neq \emptyset$. Thus, take any point on $\Wcal^u(NT) \cap \Ical_{\tilde{d} > 0}$. Then, the associated orbit will converge to $NT$ as $\tilde{\xi} \rightarrow -\infty$. Additionally, since $\Ical_{\tilde{d} > 0}$ is inflowing invariant and does not contain a periodic orbit, the orbit from any point on $\Wcal^u(NT) \cap \Ical_{\tilde{d} > 0}$ is contained in a compact subset of the phase space and its $\omega$-limit set is either an equilibrium or a separatrix cycle by applying \cite[Thm. 6.41]{meiss2017}. However, the latter cannot occur since $ST_B$ is a stable equilibrium point. Since this is also the only equilibrium point in $\Ical_{\tilde{d} > 0}$ which stable manifold intersects with $\Ical_{\tilde{d} > 0}$, the $\omega$-limit set must be $\{ST_B\}$. Therefore, we obtain a heteroclinic orbit $NT \longrightarrow ST_B$ which completes the proof.
	\end{proof}
	
	In the case $\tilde{c} < 0$, we can only obtain partial results since we are not able to exclude the presence of periodic orbits analytically in the cases, where $NT$ can have complex conjugated eigenvalues, see Proposition \ref{prop:stability-NT}. However, if $\tilde{\gamma}_A \tilde{\gamma}_B > 0$, we obtain that $\tilde{d} < 1$ and therefore, the eigenvalues of $NT$ are real. In this case, we can prove the following.
	
	\begin{lemma}
		Let $\tilde{c} < 0$ and $\tilde{\gamma}_A, \tilde{\gamma}_B$ such that \eqref{eq:non-triv-fp-invertibility}--\eqref{eq:non-triv-fp-pos2} are satisfied with $\tilde{d} > 0$. Additionally, assume that $\tilde{\gamma}_A \tilde{\gamma}_B > 0$. Then, the system \eqref{eq:radii-dynamics-rescaled} has two heteroclinic orbits $NT \longrightarrow ST_A$ and $NT \longrightarrow ST_B$. Together with the coordinate axes, these orbits form the boundary of a neighbourhood of $T$, which is filled with heterolinic orbits $NT \longrightarrow T$.
	\end{lemma}
	\begin{proof}
		We again decompose the phase plane into four subsets $\Omega_j$, $j = 1,2,3,4$ separated by the nullclines $\ell_A$ and $\ell_B$, see Figure \ref{fig:phase-plane-d>0}. We note that the qualitative flow in the different subsets is similar to the depiction in Figure \ref{fig:phase-plane-d>0}, but since $\tilde{c} < 0$, the vertical direction of the flow is reversed.
		
		The proof of the orbits $NT \longrightarrow ST_A$ and $NT \longrightarrow ST_B$ follows similar to the proof of Lemma \ref{lem:het-orbits-d>0} by constructing appropriate inflowing and overflowing invariant subsets. After constructing the orbits connecting to $ST_A$ and $ST_B$, we obtain an invariant set with corners $T$, $ST_A$, $ST_B$ and $NT$. Since $NT$ is unstable and the stable manifolds of $ST_A$ and $ST_B$ are one-dimensional, any orbit in the interior of this invariant set must flow from $NT$ to $T$, which completes the proof.
	\end{proof}
	
	\subsection{Heteroclinic orbtis in case that $NT$ does not exist}
	
	We now discuss the existence of heteroclinic orbits in the case that $\tilde{d} \neq 0$ but the positivity conditions \eqref{eq:non-triv-fp-pos1} or \eqref{eq:non-triv-fp-pos2} are violated and thus, $NT$ does not exist.
	
	\begin{lemma}\label{lem:het-orbit-non-NT-1}
		Let $\tilde{c} < 0$, $\tilde{\gamma}_A < -1$ and $\tilde{\gamma}_B > -1$. Then, the system \eqref{eq:radii-dynamics-rescaled} has a heteroclinic orbit $ST_B \longrightarrow ST_A$. Additionally, there exists a family of heteroclinic orbits $ST_B \longrightarrow T$, which do not lie on the invariant coordinate axis $\{r_A = 0\}$. Similarly, if $\tilde{\gamma}_A > -1$ and $\tilde{\gamma}_B < -1$, the system \eqref{eq:radii-dynamics-rescaled} has a heteroclinic orbit $ST_A \longrightarrow ST_B$ and a family of heteroclinic orbits $ST_A \longrightarrow T$, which do no lie on the invariant coordinate axis $\{r_B = 0\}$.
	\end{lemma}
	\begin{proof}
		We restrict the proof to the case $\tilde{\gamma}_A < -1$ and $\tilde{\gamma}_B > -1$ since the second case follows with the same arguments. The main idea of the proof is to show that there is an overflowing invariant set in the phase plane $[0,\infty)$, which is bounded by a subset of both coordinate axis containing $ST_A$ and $ST_B$, respectively, see Figure \ref{fig:overflowing-set-non-NT}. Since the vector field is smooth, continuously deforming the inner boundary close to $T$ into the outer boundary, then shows that there is a curve from a point on $\{r_A = 0\}$ to a point on $\{r_B = 0\}$, where the normal direction of the vector field vanishes at every point on the curve. Therefore, the curve is an orbit of \eqref{eq:radii-dynamics-rescaled} and the endpoints must be $ST_A$ and $ST_B$ respectively. Finally, using that $ST_A$ stable in the transverse direction off the coordinate axis and $ST_B$ is unstable in the assumed parameter regime then shows that there is a heteroclinic orbit $ST_B \longrightarrow ST_A$.
		
		It remains to show that there is such an overflowing invariant subset, see Figure \ref{fig:overflowing-set-non-NT}. First, we consider the normal flow on the circle section
		\begin{equation*}
			\Ccal_i(s) := \{(r_A, r_B) \in [0,\infty)^2 \,:\, r_A^2 + r_B^2 = s^2\}
		\end{equation*}
		for any radius $s > 0$. The normal vector pointing in the direction of $T$ is then given by $n_i = (-r_A, -r_B)$ for every $(r_A,r_B) \in \Ccal_i(s)$ and the vector field in direction $n_i$ is
		\begin{equation*}
			n_i \cdot (r_A F_A(r_A,r_B), r_B F_B(r_A,r_B)) = r_A^2 (1-r_A^2) + \dfrac{1}{|\tilde{c}|} r_B^2 (1 - r_B^2) + r_A^2 r_B^2 \left(\tilde{\gamma}_A + \dfrac{\tilde{\gamma}_B}{|\tilde{c}|}\right),
		\end{equation*}
		where we use that $\tilde{c} <0$. By choosing $s > 0$ sufficiently small, we can guarantee that this is positive. Hence, taking the direction of $n_i$ into account, the vector field flows from $\{r_A^2 + r_B^2 > s^2\}$ into $\{r_A^2 + r_B^2 < s^2\}$.
		
		Next, we consider
		\begin{equation*}
			\begin{split}
				\Ccal_o(S_A,S_B) &:= \Ccal_{o,1}(S_A,S_B) \cup \Ccal_{o,2}(S_A,S_B), \\
				\Ccal_{o,1}(S_A,S_B) &:= \{r_A = S_A \text{ and } r_B \leq S_B\}, \\
				\Ccal_{o,2}(S_A,S_B) &:= \{r_A \leq S_A \text{ and } r_B = S_B\}
			\end{split}
		\end{equation*}
		with $S_A, S_B > 1.$
		First, we consider the case that $\tilde{\gamma}_A, \tilde{\gamma}_B \leq 0$. Then, for every point $(r_A,r_B) \in \Ccal_{o,1}(S_A,S_B)$ holds
		\begin{equation*}
			(1,0) \cdot (r_A F_A(r_A,r_B), r_B F_B(r_A,r_B)) = S_A (- 1 + S_A^2  - \tilde{\gamma}_A r_B^2) > 0
		\end{equation*}
		and similarly for every point $(r_A, r_B) \in \Ccal_{o,2}(S_A,S_B)$ holds
		\begin{equation*}
			(0,1) \cdot (r_A F_A(r_A,r_B), r_B F_B(r_A,r_B)) = \dfrac{1}{|\tilde{c}|} S_B (-1 + S_B^2 - \tilde{\gamma}_B r_A^2) > 0
		\end{equation*}
		using again that $\tilde{c}< 0$. Therefore, the rectangle $\{r_A \leq S_A \text{ and } r_B \leq S_B\}$ is an overflowing invariant set.
		Next, we consider the case that $\tilde{\gamma}_B > 0$. Since $\tilde{\gamma}_A < 0$ we still have that
		\begin{equation*}
			(1,0) \cdot (r_A F_A(r_A,r_B), r_B F_B(r_A,r_B)) = S_A (- 1 + S_A^2  - \tilde{\gamma}_A r_B^2) \geq S_A (- 1 + S_A^2) > 0
		\end{equation*}
		for all $r_B \leq S_B$ by choosing $S_A  > 1$. Note that this choice does not depend on $S_B$. The normal flow over on $\Ccal_{o,2}(S_A,S_B)$ is then given by
		\begin{equation*}
			\dfrac{1}{|\tilde{c}|} S_B (-1 + S_B^2 - \tilde{\gamma}_B r_A^2) \geq \dfrac{1}{|\tilde{c}|} S_B (-1 + S_B^2 - \tilde{\gamma}_B S_A^2)
		\end{equation*}
		using that $\tilde{\gamma}_B > 0$. Choosing $S_B^2 > 1 + \tilde{\gamma}_A S_A^2 > 1$ then ensures that the normal flow is positive. Therefore, the rectangle $\{r_A \leq S_A \text{ and } r_B \leq S_B\}$ is again overflowing.
		
		Combining both parts, there exist $s < 1$, $S_A > 1$ and $S_B > 0$ such that the set $\{r_A^2 + r_B^2 \geq s^2\}  \cap \{r_A \leq S_A \text{ and } r_B \leq S_B\}$ is non-empty and overflowing. Additionally, the boundary contains sections of both coordinate axes and in particular the points $ST_A$ and $ST_B$. This completes the proof.
	\end{proof}
	
	\begin{figure}
		\centering
		\includegraphics[width=0.45\textwidth]{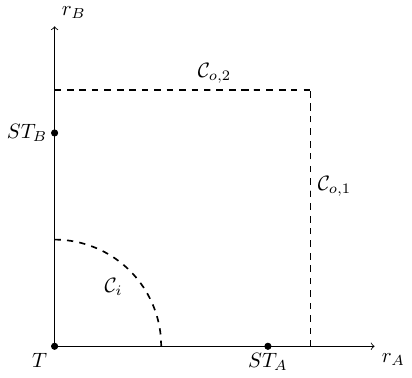}
		\hfill
		\includegraphics[width=0.45\textwidth]{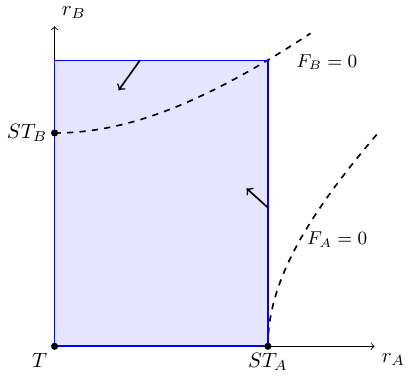}
		\caption{Schematic depiction of the overflowing set constructed in Lemma \ref{lem:het-orbit-non-NT-1} and the inflowing rectangle in the proof of Lemma \ref{lem:het-orbit-non-NT-2}.}
		\label{fig:overflowing-set-non-NT}
	\end{figure}
	
	Next, we discuss the case that $\tilde{c} > 0$ and both $\tilde{\gamma}_A > -1$ and $\tilde{\gamma}_B > -1$. Note that this also implies $\tilde{d} < 0$ since otherwise, the positivity conditions \eqref{eq:non-triv-fp-pos1} and \eqref{eq:non-triv-fp-pos2} would hold.
	
	\begin{lemma}\label{lem:het-orbit-non-NT-2}
		Let $\tilde{c} > 0$ and $\tilde{\gamma}_A > -1$, $\tilde{\gamma}_B > -1$ such that $\tilde{d} < 0$. Then, there is a family of heteroclinic orbits $ST_A \longrightarrow ST_B$.
	\end{lemma}
	\begin{proof}
		We first note that $\tilde{\gamma}_A > -1$ and $\tilde{\gamma}_B > -1$ together with $\tilde{d} < 0$ implies that $\tilde{\gamma}_A \tilde{\gamma}_B > 1$ and thus $\tilde{\gamma}_A > 0$ and $\tilde{\gamma}_B > 0$. We choose a rectangle as depicted in Figure \ref{fig:overflowing-set-non-NT}, which is an inflowing invariant set under the dynamics of \eqref{eq:radii-dynamics-rescaled}. Additionally, it does not contain any periodic orbits since any orbit in the rectangle satisfies $\dot{r}_A < 0$. Using that $T$ is a saddle, $ST_A$ is stable equilibrium point and $ST_B$ is an unstable equilibrium point, as well as the fact that all equilibrium points are contained in the rectangle, an application of \cite[Thm.~6.41]{meiss2017} then shows that the rectangle is filled with heteroclinic orbits $ST_A \longrightarrow ST_B$.
	\end{proof}
	
	\section{Approximate modulated waves in the pattern-forming system \eqref{eq:toy-model}}\label{sec:modulated-waves-in-full-system}
	
	Recalling the approximation result Theorem \ref{thm:full-approx-result}, any sufficiently regular solution $(\Aav, \Bav)$ to the amplitude system \eqref{eq:final-amplitude-eq} gives rise to a solution $(u,v)$ of the full pattern-formation model \eqref{eq:toy-model} on a large time interval $[0,T_0/\varepsilon^2]$ for some $T_0 > 0$ and $0 < \varepsilon \ll 1$ satisfying the error estimate
	\begin{equation*}
		\sup_{t \in [0,T_0/\varepsilon^2]} \|(u,v) - \Psi_\mathrm{GL}(\Aav, \Bav)\| \leq C \varepsilon^2
	\end{equation*}
	for some constant $C < \infty$. Here, $\Psi_\mathrm{GL}$ is given by
	\begin{equation*}
		\Psi_\mathrm{GL}(A,B) = \varepsilon \begin{pmatrix}
			A \ee^{\ii x} \\
			B \ee^{\ii(x-c_p t)}
		\end{pmatrix} + c.c.,
	\end{equation*}
	see \eqref{eq:ansatz-amplitude-eq}. Specifically, writing $\Aav = r_A \ee^{\ii\phi_A}$ and $\Bav = r_B \ee^{\ii\phi_B}$ we find that
	\begin{equation}\label{eq:approx-form-solutions}
		\begin{split}
			u(t,x) &= \varepsilon r_A(T, X) \cos(x + \phi_A(T,X)) + \Ocal(\varepsilon^2), \\
			v(t,x) &= \varepsilon r_B(T,X) \cos(x + c_p t + \phi_B(T,X)) + \Ocal(\varepsilon^2)
		\end{split}
	\end{equation}
	for $t \in [0,T_0/\varepsilon^2]$ and $x \in \R$.
	
	In the previous Sections \ref{sec:time-periodic-solutions} and \ref{sec:fronts}, we discussed three types of solutions to the amplitude equations \eqref{eq:full-amplitude-eq}.
	\begin{enumerate}
		\item Time-periodic solutions;
		\item Space-time periodic solutions with long spatial wave-length;
		\item Fast-moving front solutions connecting space-time-periodic solutions.
	\end{enumerate}
	We now discuss the corresponding solutions in the full pattern-formation model \eqref{eq:toy-model}.
	
	\subsection{Modulated waves corresponding to time-periodic solutions}
	
	Time-periodic solutions established in Section \ref{sec:time-periodic-solutions}, see Propositions \ref{prop:semi-trivial-solutions} and \ref{prop:fully-nontrivial-solutions}, are of the form
	\begin{equation}
		(\Aav, \Bav)(T,X) = (r_A \ee^{\ii\omega_A T}, r_B \ee^{\ii\omega_B T})
	\end{equation}
	with radii $r_A, r_B \geq 0$ and temporal wave numbers $\omega_A, \omega_B \in \R$. As already discussed in Remark \ref{rem:time-periodic-solutions-correspond-to-speed-correction}, the corresponding solutions in \eqref{eq:toy-model} are spatially periodic wavetrains. We summarise the results in the following theorem.
	
	\begin{theorem}\label{thm:pure-patterns}
		Let $T_0 > 0$. Then, there exists a $\varepsilon_0 > 0$ such that for every $\varepsilon \in (0,\varepsilon_0)$, the pattern-forming system \eqref{eq:toy-model} has solutions $(u,v) : [0,T_0/\varepsilon^2] \rightarrow \R^2$ of the form
		\begin{equation}\label{eq:pure-turing-pattern}
			u(t,x) = \varepsilon \sqrt{-\dfrac{\alpha_u}{\gamma_{1,r}}} \cos\left(x + \varepsilon^2 \dfrac{\alpha_u \gamma_{1,i}}{\gamma_{1,r}} t\right) + \Ocal(\varepsilon^2), \qquad v(t,x) = \Ocal(\varepsilon^2)
		\end{equation}
		if $\alpha_u \gamma_{1,r} < 0$, and 
		\begin{equation}\label{eq:pure-turing-hopf-pattern}
			u(t,x) = \Ocal(\varepsilon^2), \qquad v(t,x) = \varepsilon \sqrt{-\dfrac{\alpha_v}{\gamma_{8,r}}} \cos\left(x - c_g t + \varepsilon^2 \dfrac{\alpha_v \gamma_{8,i}}{\gamma_{8,r}} t\right) + \Ocal(\varepsilon^2)
		\end{equation}
		if $\alpha_v \gamma_{8,r} < 0$.
		
		If the parameter assumptions in Proposition \ref{prop:fully-nontrivial-solutions} are satisfied, the pattern-forming system also has solutions of the form
		\begin{equation}\label{eq:pattern-superposition}
			u(t,x) = \varepsilon r_A \cos(x + \varepsilon^2 \omega_A t) + \Ocal(\varepsilon^2), \qquad v(t,x) = \varepsilon r_B \cos(x - c_g t + \varepsilon^2 \omega_B t) + \Ocal(\varepsilon^2)
		\end{equation}
		with $r_A, r_B, \omega_A, \omega_B$ given in Proposition \ref{prop:fully-nontrivial-solutions}.
	\end{theorem}
	
	Theorem \ref{thm:pure-patterns} establishes three types of solutions to the pattern-forming system \eqref{eq:toy-model}. First, \eqref{eq:pure-turing-pattern} corresponds, to leading order, to a pattern bifurcating from a pure Turing bifurcation. Considering that the cubic coefficient $\gamma_1$ in the amplitude system is explicitly given by
	\begin{equation*}
		\gamma_1 = \dfrac{f_{11} g_{20} (16 c_p - 19 i)}{8c_p - 9i} + \dfrac{38 f_{20}^2}{9} + 3 f_{30},
	\end{equation*}
	see \eqref{eq:taylor-expansion-nonlinearities} and Appendix \ref{app:coefficients}, we find that the pattern \eqref{eq:pure-turing-pattern} is stationary to leading order if $f_{11} g_{20} = 0$. Therefore, the small phase velocity originates from the quadratic coupling of the Turing and Turing–Hopf modes in \eqref{eq:toy-model}.
	Similarly, \eqref{eq:pure-turing-hopf-pattern} resembles a travelling wavetrain bifurcating from a pure Turing–Hopf instability with a small perturbation of the linearly determined phase velocity.
	Finally, the solution \eqref{eq:pattern-superposition} describes a superposition of a pure Turing pattern and a pure Turing–Hopf pattern.
	
	\subsection{Modulated waves corresponding to space-time periodic solutions with long spatial wave-length}
	
	Next, we consider the approximated modulated waves in the pattern-forming system \eqref{eq:toy-model}, which correspond to the space-time periodic solutions obtained in Proposition \ref{prop:space-time-periodic-waves}. These are of the form
	\begin{equation*}
		(\Aav, \Bav)(T,X) = (r_A \ee^{\ii\varepsilon \tilde{k}_A X + \omega_AT}, r_B \ee^{\ii\varepsilon \tilde{k}_B X + \omega_B T})
	\end{equation*}
	with amplitudes $r_A, r_B \geq 0$ and spatial and temporal wave numbers $\tilde{k}_j, \omega_j \in \R$ for $j = A, B$. Using the approximation result Theorem \ref{thm:full-approx-result}, we find that the corresponding solutions in the pattern-forming system \eqref{eq:toy-model} are still travelling waves with constant amplitude, albeit with a perturbed spatial and temporal frequency.
	
	\begin{theorem}\label{thm:patterns-space-time}
		Let $T_0  > 0$ and take a bounded set $K \in \R^2$. Then, there exists an $\varepsilon_0 > 0$ such that for every $\varepsilon \in (0,\varepsilon_0)$ and $(\tilde{k}_A, \tilde{k}_B) \in K$, the pattern-forming system \eqref{eq:toy-model} has solutions $(u,v) : [0,T_0/\varepsilon^2] \rightarrow \R^2$ of the form
		\begin{equation}\label{eq:pure-turing-pattern-2}
			u(t,x) = \varepsilon \sqrt{-\dfrac{\alpha_u}{\gamma_{1,r}}} \cos\left(\left(1 + \varepsilon^2 \tilde{k}_A\right) x + \varepsilon^2 \dfrac{\alpha_u \gamma_{1,i}}{\gamma_{1,r}} t\right) + \Ocal(\varepsilon^2), \qquad v(t,x) = \Ocal(\varepsilon^2)
		\end{equation}
		if $\alpha_u \gamma_{1,r} < 0$, and
		\begin{equation}\label{eq:pure-turing-hopf-pattern-2}
			u(t,x) = \Ocal(\varepsilon^2), \qquad v(t,x) = \varepsilon \sqrt{-\dfrac{\alpha_v}{\gamma_{8,r}}} \cos\left((1+\varepsilon^2 \tilde{k}_B) x - c_g t + \varepsilon^2 \left(\dfrac{\alpha_v \gamma_{8,i}}{\gamma_{8,r}} - c_g \tilde{k}_B\right) t \right) + \Ocal(\varepsilon^2)
		\end{equation}
		if $\alpha_v \gamma_{8,r} < 0$.
		
		If the parameter assumptions in Proposition \ref{prop:fully-nontrivial-solutions} are satisfied, the pattern-forming system also has solutions of the form
		\begin{equation}\label{eq:pattern-superposition-2}
			\begin{split}
				u(t,x) &= \varepsilon r_A \cos((1+\varepsilon^2 \tilde{k}_A) x + \varepsilon^2 \omega_A t) + \Ocal(\varepsilon^2), \\
				v(t,x) &= \varepsilon r_B \cos((1+\varepsilon^2 \tilde{k}_B) x - c_g t + \varepsilon^2 (\omega_B - c_g \tilde{k}_B) t) + \Ocal(\varepsilon^2)
			\end{split}
		\end{equation}
		with $r_A, r_B, \omega_A, \omega_B$ given in Proposition \ref{prop:fully-nontrivial-solutions}.
	\end{theorem}

	\begin{remark}
		The main difference of the solutions in Theorem \ref{thm:pure-patterns} and Theorem \ref{thm:patterns-space-time} is the small correction of the spatial frequency. In addition, the temporal frequency of $v$ is also adjusted by an additional perturbation of order $\varepsilon^2$ in \eqref{eq:pure-turing-hopf-pattern-2} and \eqref{eq:pattern-superposition-2}, which is due to the singular advection term in \eqref{eq:final-amplitude-eq}.
	\end{remark}

	\subsection{Modulated waves corresponding to fast-moving front solutions}
	
	Finally, we discuss the solution to the pattern-forming systems \eqref{eq:toy-model}, which arise from the fast-moving travelling front solutions constructed in Section \ref{sec:fronts}. Recall that the front solutions are of the form
	\begin{equation}\label{eq:fast-front-ansatz}
		(\Aav, \Bav)(T,X_1) = (r_A(\tilde{\xi}) \ee^{\ii\phi_A(\tilde{\xi}) + \omega_A T}, r_B(\tilde{\xi}) \ee^{\ii\phi_B(\tilde{\xi}) + \omega_B T)})
	\end{equation}
	with $\tilde{\xi} = \varepsilon X_1 - c_0 T$ for $c_0 \in \R$. Additionally, the temporal frequencies are chosen corresponding to the semi-trivial equilibrium point established in Proposition \ref{prop:semi-trivial-solutions}, that is $\omega_A = \tfrac{\alpha_u \gamma_{1,i}}{\gamma_{1,r}}$ and $\omega_B = \tfrac{\alpha_v \gamma_{8,i}}{\gamma_{8,r}}$. Note that the phases $\phi_A$ and $\phi_B$ are not necessarily bounded, but rather, the local wave numbers $\psi_A = \dot{\phi}_A$ and $\psi_B = \dot{\phi}_B$ approach limits as $\tilde{\xi} \rightarrow \pm\infty$. To discuss the properties of the solutions induced in the full pattern-forming system \eqref{eq:toy-model}, we restrict to $c_0 > 0$ since the case $c_0 < 0$ can be recovered by changing $x \mapsto -x$ and $c_d \mapsto -c_d$ in \eqref{eq:toy-model}.
	
	Recalling Theorems \ref{thm:pure-patterns} and \ref{thm:patterns-space-time}, we note that there is a direct correspondence of equilibrium points in the radii system \eqref{eq:radii-dynamics} and solutions in the pattern-forming system \eqref{eq:toy-model}. Specifically, $T$ corresponds to the trivial solution $(u,v) = (0,0)$ and the semi-trivial equilibrium points $ST_A$ and $ST_B$ correspond to a pure Turing pattern \eqref{eq:pure-turing-pattern} and a pure Turing–Hopf pattern \eqref{eq:pure-turing-hopf-pattern}, respectively. For the latter we note that since $\omega_A$ and $\omega_B$ as chosen as in Proposition \ref{prop:semi-trivial-solutions} in \eqref{eq:fast-front-ansatz}, the local wave numbers at $ST_A$ and $ST_B$ vanish, cf.~\eqref{eq:slow-dynamics}. Finally, the nontrivial equilibrium point $NT$ corresponds to a solution of the form \eqref{eq:pattern-superposition-2} with $\tilde{k}_j = \psi_{j,NT}$ for $j = A,B$, where $\psi_{j,NT}$ denotes the (non-zero) local wave numbers at $NT$ given by \eqref{eq:slow-dynamics}.
	
	\begin{theorem}\label{thm:pattern-interface}
		Let $T_0 > 0$ and a bounded set $\mathcal{S} \subset \C^2$ containing $0$. Then there exists an $\varepsilon_0 > 0$ such that for every $\varepsilon \in (0,\varepsilon_0)$ and for every fast-travelling front \eqref{eq:fast-front-ansatz} in the amplitude equations \eqref{eq:final-amplitude-eq}, which is contained in $\mathcal{S}$, the pattern-forming system \eqref{eq:toy-model} has a solution $(u,v) : [0,T_0/\varepsilon] \rightarrow \R^2$ of the form
		\begin{equation*}
			\begin{split}
				u(t,x) &= 2 \varepsilon r_A(\varepsilon^2(x - c_0 t)) \cos\left(x + \phi_A(\varepsilon^2 (x-c_0 t)) + \varepsilon^2 \omega_A t\right) + \Ocal(\varepsilon^2), \\
				v(t,x) &= 2 \varepsilon r_B(\varepsilon^2(x-c_0 t)) \cos\left(x - c_p t + \phi_B(\varepsilon^2 (x-c_0 t)) + \varepsilon^2 \omega_B t\right) + \Ocal(\varepsilon^2).
			\end{split}
		\end{equation*}
		In particular, the following correspondence between solutions to the radii system \eqref{eq:radii-dynamics-rescaled} and the pattern-forming system \eqref{eq:toy-model} holds:
		\begin{enumerate}
			\item an orbit $ST_A \longrightarrow T$ corresponds to an invasion of the trivial solution $(u,v) = (0,0)$ by a pure Turing pattern \eqref{eq:pure-turing-pattern}, see Figure \ref{fig:T-to-trivial};
			\item orbits $ST_B \longrightarrow T$ and $T \longrightarrow ST_B$ correspond to an invasion of the trivial solution $(u,v) = (0,0)$ by a pure Turing–Hopf pattern \eqref{eq:pure-turing-hopf-pattern}, see Figure \ref{fig:TH-to-trivial}, or an invasion of a pure Turing–Hopf pattern by the trivial solution;
			\item orbits $ST_A \longrightarrow ST_B$ and $ST_B \longrightarrow ST_A$ correspond to an invasion of a pure Turing–Hopf pattern \eqref{eq:pure-turing-hopf-pattern} by a pure Turing pattern \eqref{eq:pure-turing-pattern} or an invasion of a pure Turing by a pure Turing–Hopf pattern, see Figure \ref{fig:T-to-TH};
			\item orbits $NT \longrightarrow T$, $NT \longrightarrow ST_A$ and $NT \longrightarrow ST_B$ correspond to an invasion of the trivial solution $(u,v) = (0,0)$, a pure Turing pattern \eqref{eq:pure-turing-pattern} and a pure Turing–Hopf pattern \eqref{eq:pure-turing-hopf-pattern}, respectively, by a solution of the form \eqref{eq:pattern-superposition-2} modelling a superposition of a Turing and a Turing–Hopf pattern, see Figure \ref{fig:superpos-to-T};
			\item orbits $ST_A \longrightarrow NT$ and $ST_B \longrightarrow NT$ correspond to an invasion of a superposition pattern \eqref{eq:pattern-superposition-2} by a pure Turing pattern \eqref{eq:pure-turing-pattern} and a pure Turing–Hopf pattern \eqref{eq:pure-turing-hopf-pattern}, respectively, see Figure \ref{fig:TH-to-superpos}.
		\end{enumerate}
	\end{theorem}
	
	\begin{figure}
		\begin{minipage}{0.48\textwidth}
			\includegraphics[width=\textwidth]{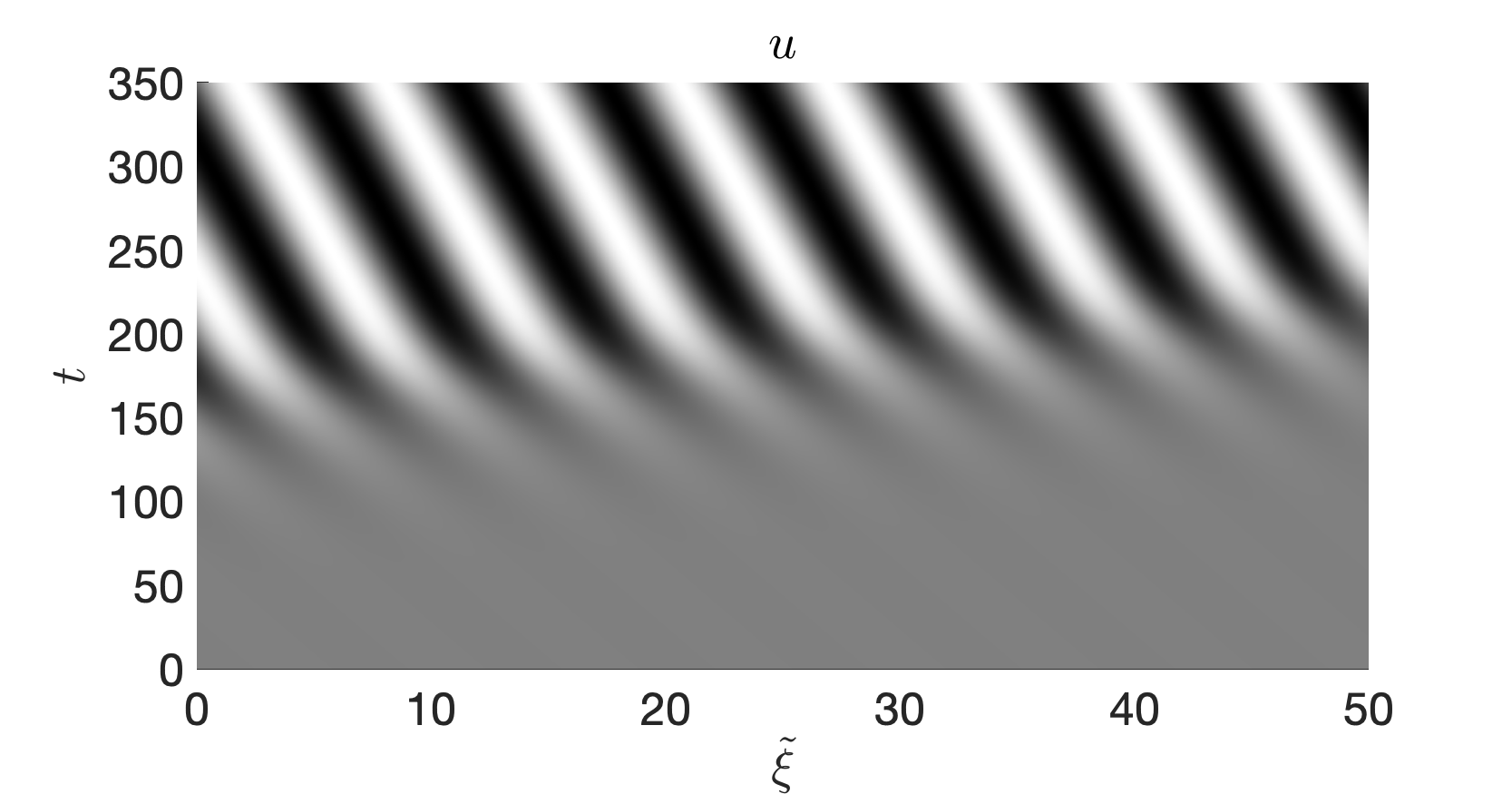}
			\caption{Leading order space-time plot of $u$ for an invasion of the trivial state by a pure Turing pattern \eqref{eq:pure-turing-pattern}. A corresponding video can be found in the supplementary material \cite{hilder2025-supplementary}.}
			\label{fig:T-to-trivial}
		\end{minipage}
		\begin{minipage}{0.48\textwidth}
			\includegraphics[width=\textwidth]{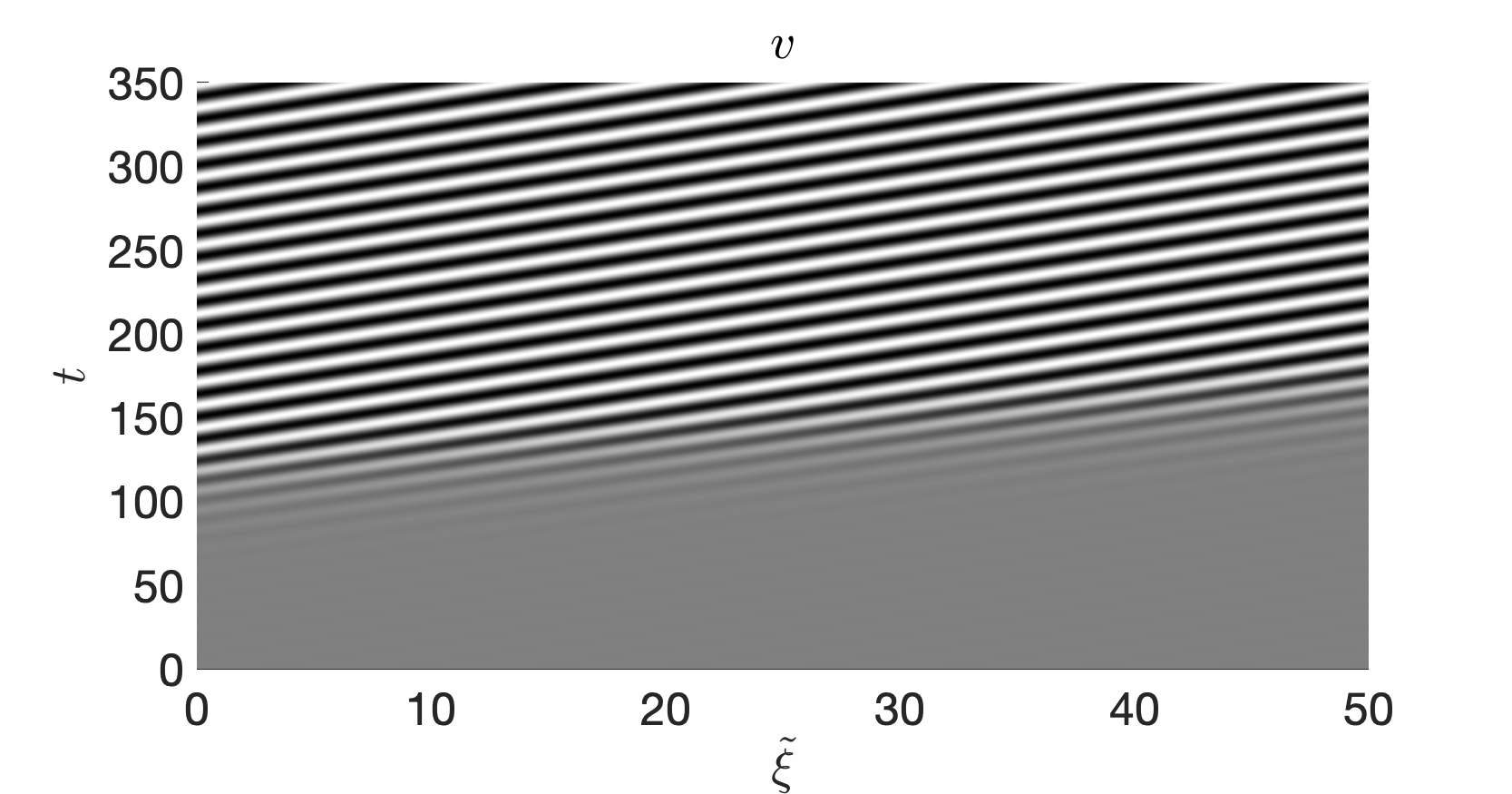}
			\caption{Leading order space-time plot of $v$ for an invasion of the trivial state by a pure Turing–Hopf pattern \eqref{eq:pure-turing-hopf-pattern}. A corresponding video can be found in the supplementary material \cite{hilder2025-supplementary}.}
			\label{fig:TH-to-trivial}
		\end{minipage}
	\end{figure}
	
	\begin{figure}
		\includegraphics[width=0.48\textwidth]{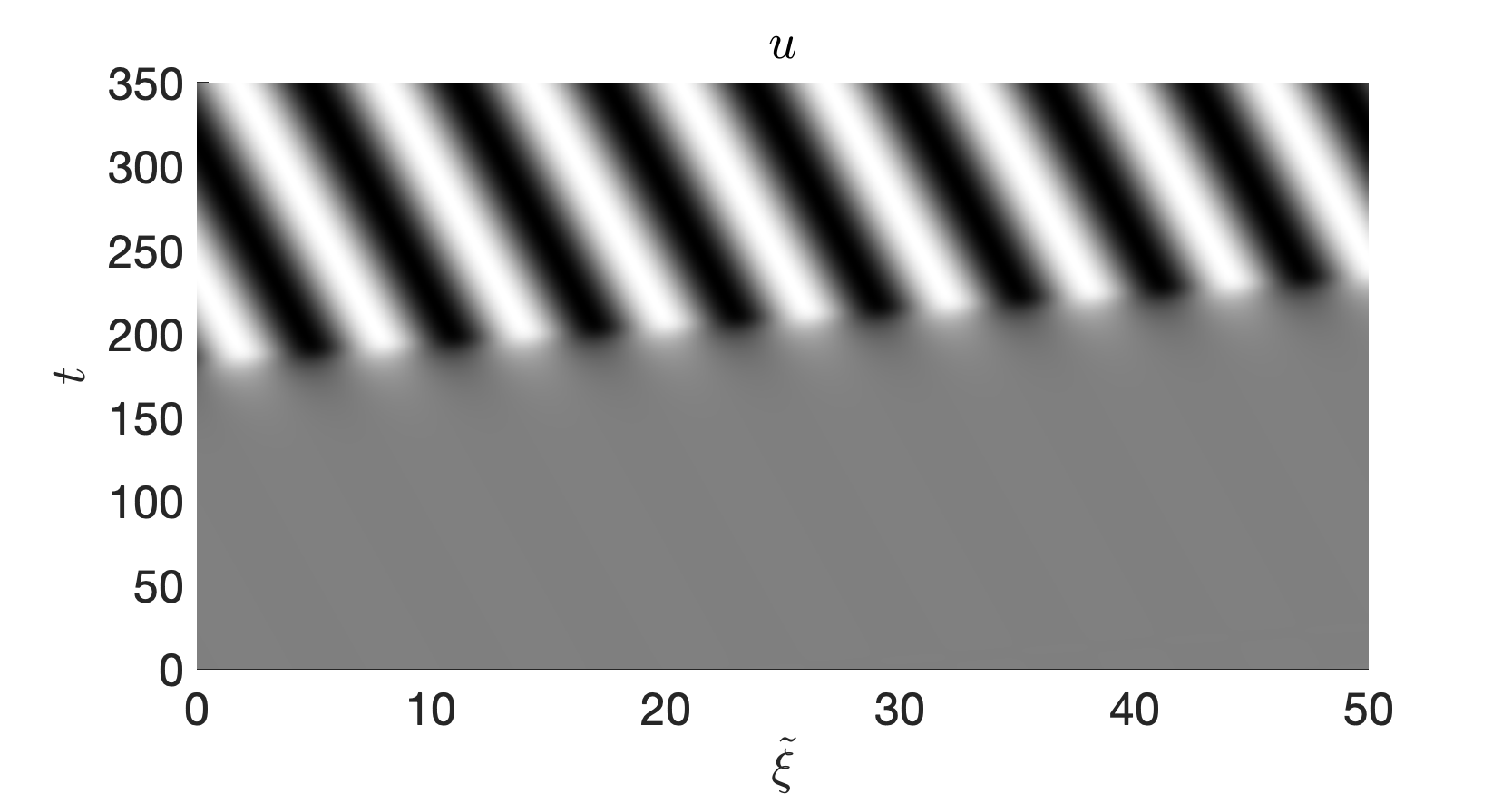}
		\includegraphics[width=0.48\textwidth]{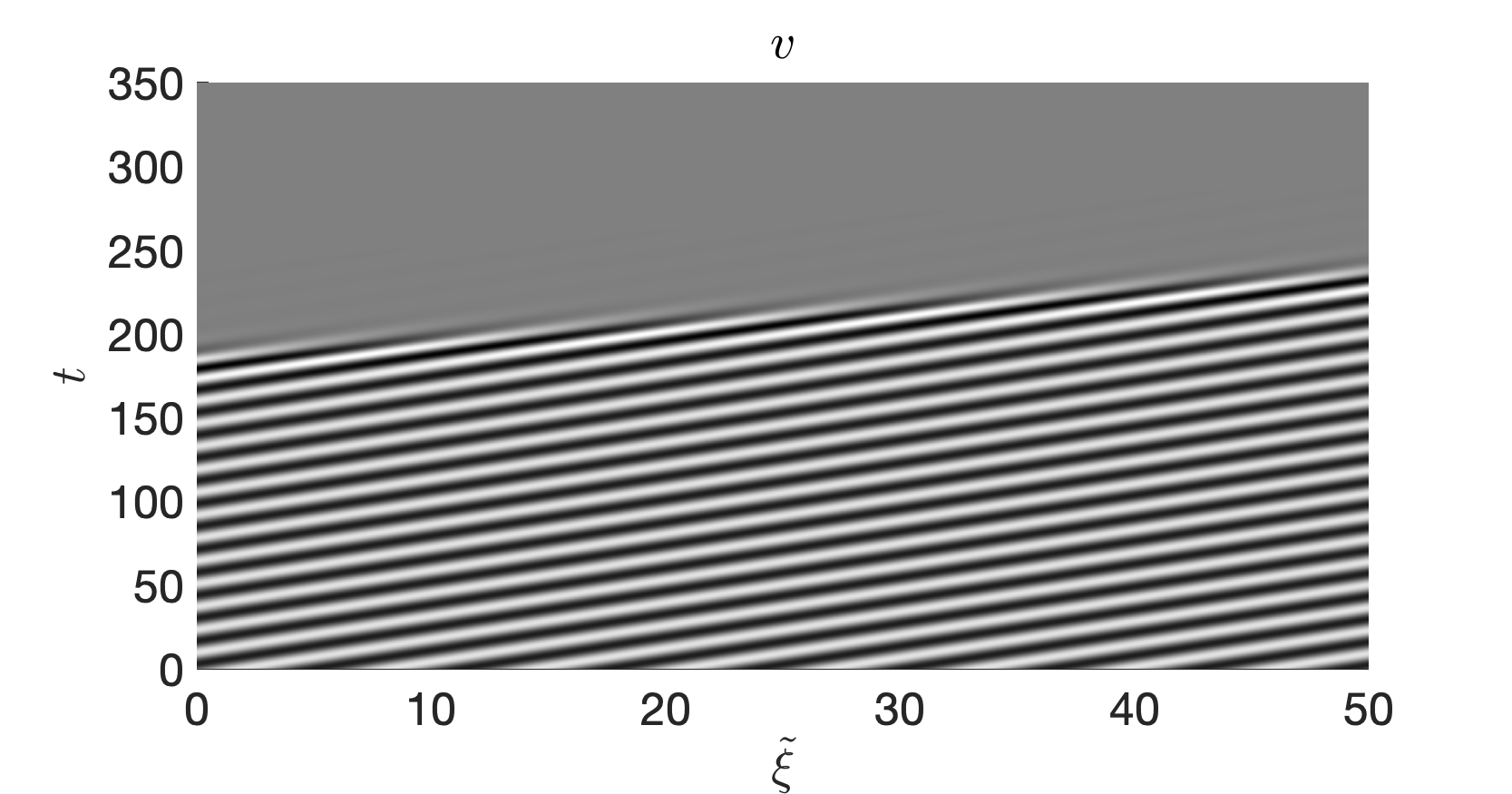}
		\caption{Leading order space-time plot of $u$ and $v$ for an invasion of a pure Turing–Hopf pattern \eqref{eq:pure-turing-hopf-pattern} by a pure Turing pattern \eqref{eq:pure-turing-pattern}. A corresponding video can be found in the supplementary material \cite{hilder2025-supplementary}.}
		\label{fig:T-to-TH}
	\end{figure}
	
	\begin{figure}
		\includegraphics[width=0.48\textwidth]{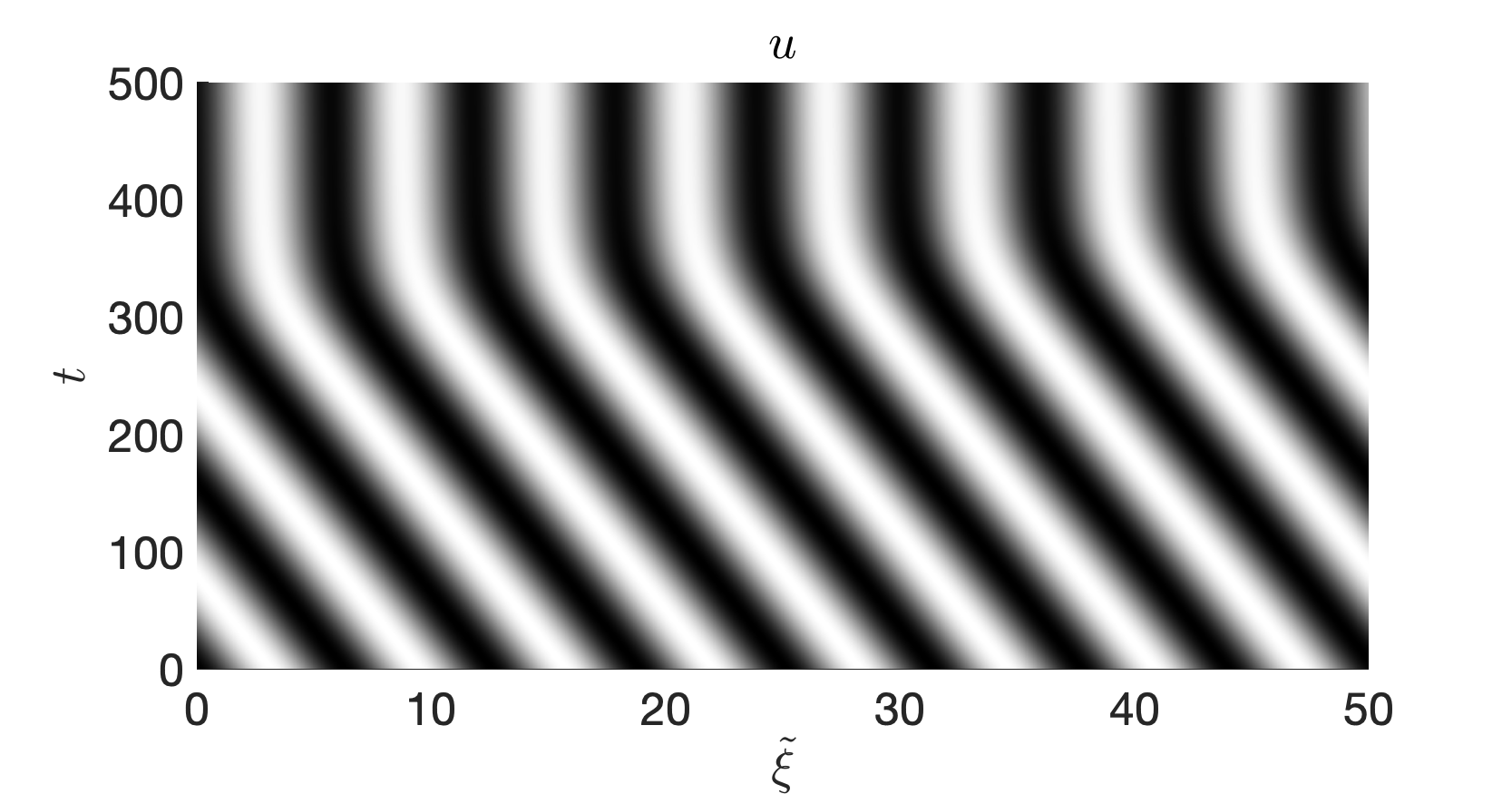}
		\includegraphics[width=0.48\textwidth]{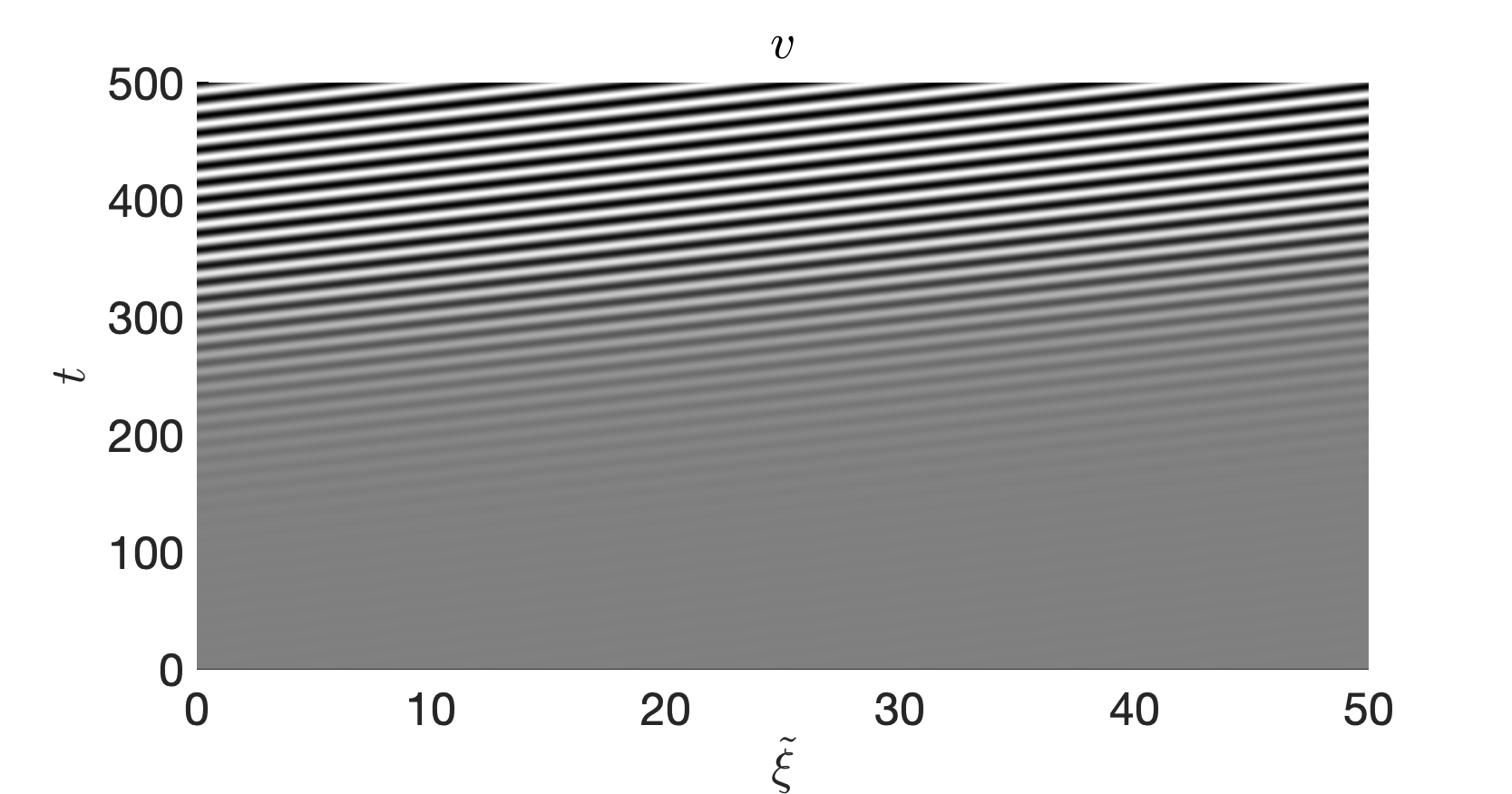}
		\caption{Leading order space-time plot of $u$ and $v$ for an invasion of a pure Turing pattern \eqref{eq:pure-turing-pattern} by a superposition pattern \eqref{eq:pattern-superposition}. A corresponding video can be found in the supplementary material \cite{hilder2025-supplementary}.}
		\label{fig:superpos-to-T}
	\end{figure}
	
	\begin{figure}
		\includegraphics[width=0.48\textwidth]{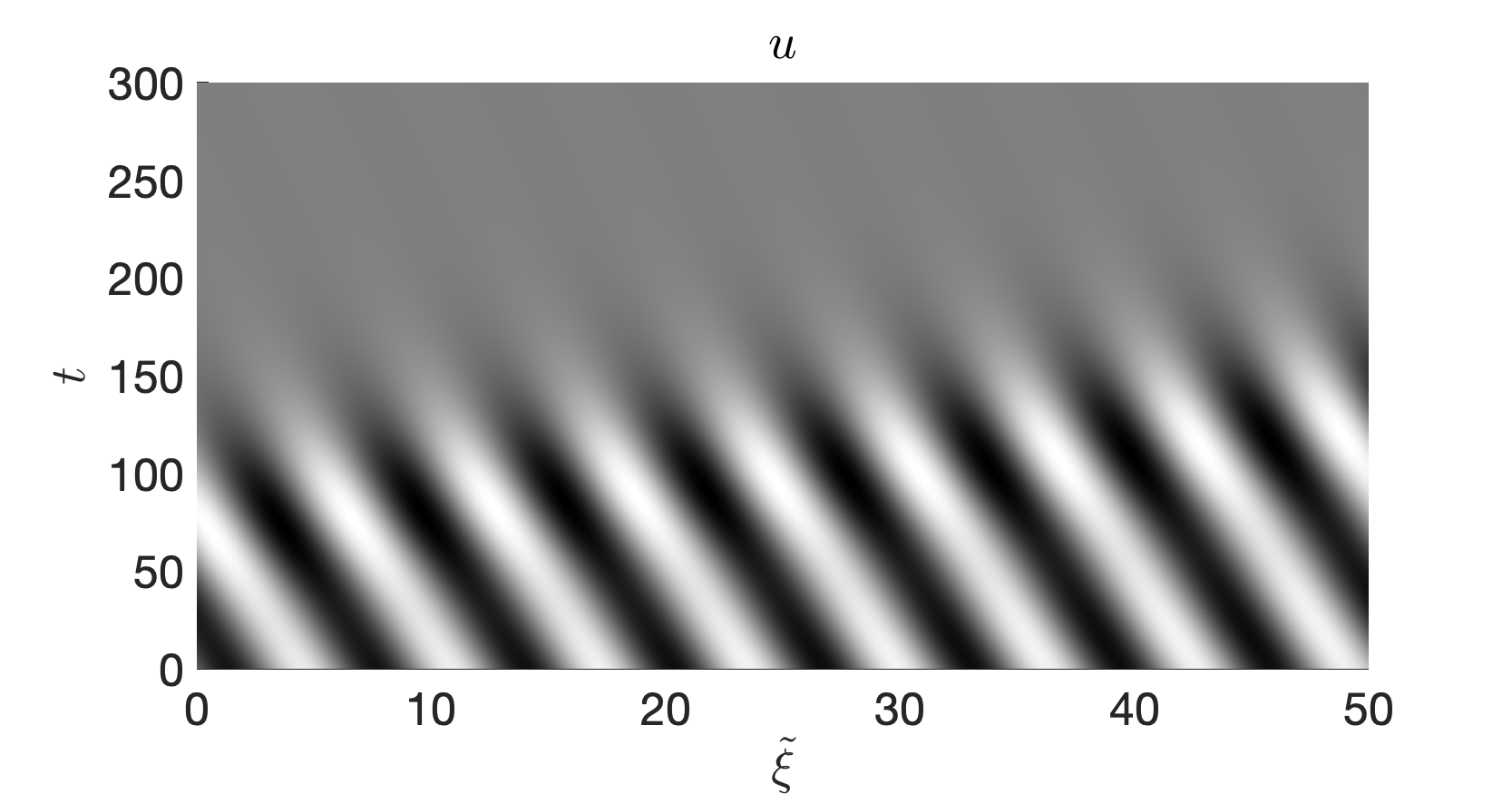}
		\includegraphics[width=0.48\textwidth]{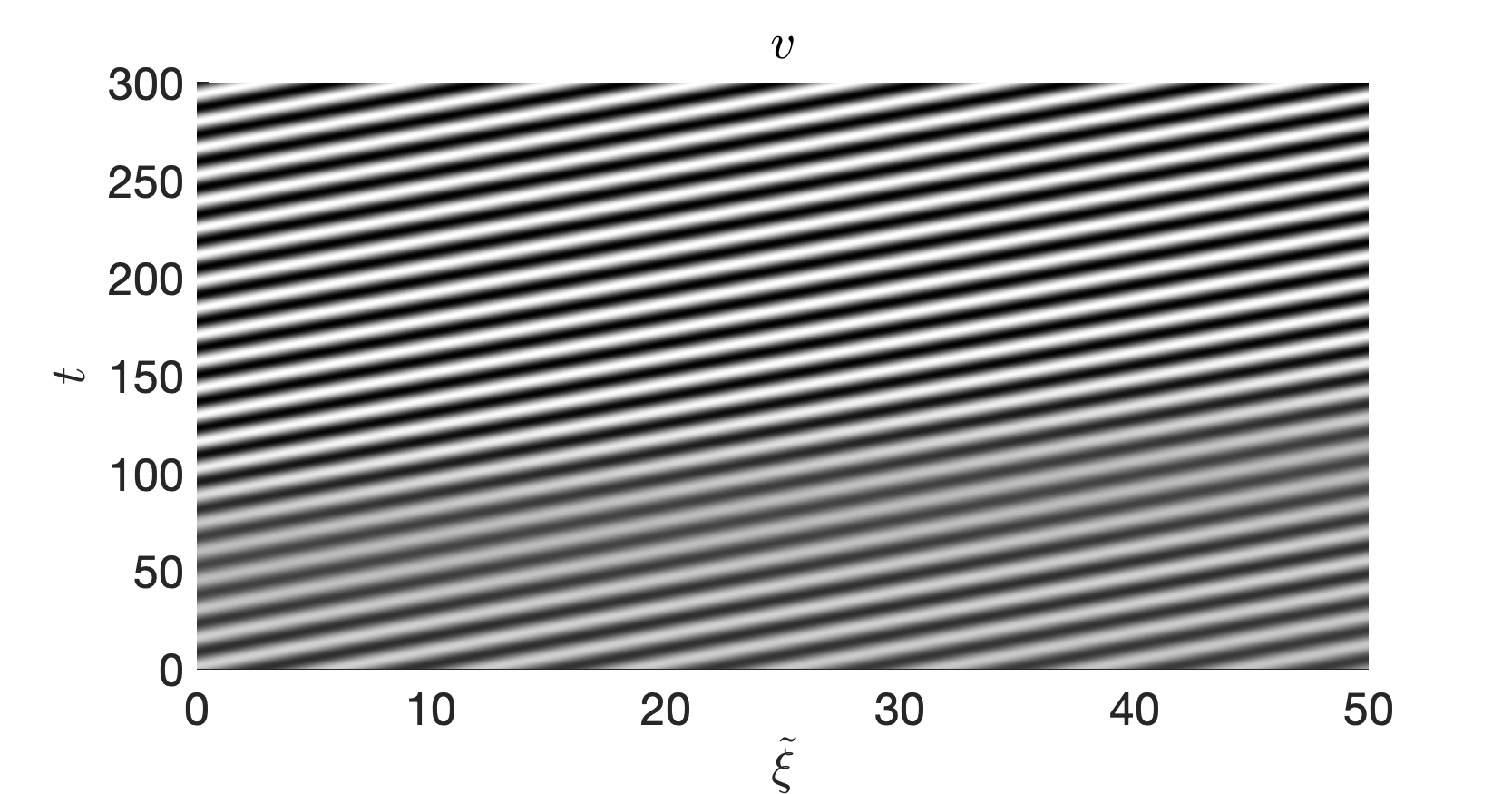}
		\caption{Leading order space-time plot of $u$ and $v$ for an invasion of a superposition pattern \eqref{eq:pattern-superposition} by a pure Turing pattern \eqref{eq:pure-turing-pattern}. A corresponding video can be found in the supplementary material \cite{hilder2025-supplementary}.}
		\label{fig:TH-to-superpos}
	\end{figure}
	
	\begin{remark}
		We point out that the pure Turing pattern in Figures \ref{fig:T-to-trivial}, \ref{fig:T-to-TH} and \ref{fig:superpos-to-T} is slowly moving. This is due to the higher-order corrections of the phase velocity in \eqref{eq:pure-turing-pattern}, which are visible here since the solutions are shown for $\varepsilon = 0.2$ on long time intervals. However, the Turing–Hopf pattern, which has a phase velocity of order one, is still visibly faster.
	\end{remark}

	\section{Bifurcation of spatially periodic solutions}\label{sec:periodic-solutions}
	
	In Theorem \ref{thm:pure-patterns}, we have obtained the existence of solutions to \eqref{eq:toy-model}, which stay close to a spatially periodic solution to the amplitude equations \eqref{eq:amplitude-eq-wohot} on a time-interval of length $\Ocal(\varepsilon^{-2})$. In this section we will see that it is in fact possible to construct global-in-time, spatially periodic solutions to \eqref{eq:toy-model}, which have a bounded amplitude. The proof of their existence is based on centre manifold theory. 
	
	The construction of spatially periodic solutions close to the onset of a finite-wavelength instability such as a Turing or Turing-Hopf instability is well known and can be obtained by various different methods such as bifurcation from a simple eigenvalue, Lyapunov-Schmidt reduction or centre manifold theory. The main challenge in our case is that spatially periodic solutions close to a simultaneous Turing and Turing-Hopf instability are not stationary solutions in any co-moving frame due to the different phase velocities of the two critical Fourier modes. This makes it very difficult to use a direct construction via bifurcation theoretical methods such as Crandall-Rabinowitz and Lyapunov-Schmidt reduction. Nevertheless, it is still possible to construct spatially periodic solutions, which are stationary up to small temporal oscillations, which are of higher order in $\varepsilon$. The idea for the construction is to show the existence of an invariant centre manifold. On this centre manifold, oscillatory modes can be moved to higher-order terms using normal form transformations. Then, up to higher-order perturbations, the normal form supports circles of non-trivial equilibrium points. Finally, using a variant of normally hyperbolic theory, see \cite[Chap.~VII]{hale1969}, we show that the reduced equations on the centre manifold have bounded invariant manifolds close these circles of equilibrium points under higher-order perturbations allowing for fast-oscillating terms. However, the dynamics on these invariant manifold is potentially very complex, see e.g.~Remark \ref{rem:complicated-dynamics}. Nevertheless, this establishes the existence of bounded, spatially periodic solutions to the pattern-forming system \eqref{eq:toy-model}.
	
	\subsection{Centre manifold theorem}
	
	We rewrite the full system \eqref{eq:toy-model} using the notation $U = (u,v)^T$ as
	\begin{equation}\label{eq:cm-formulation}
		\partial_t U = L U + N(U;\varepsilon)
	\end{equation}
	where $L$ is the linear part of \eqref{eq:toy-model} in $(U,\eps)$ and thus reads as
	\begin{equation*}
		L U = \begin{pmatrix*}
					-(1+\partial_x^2)^2 u \\
					-(1+\partial_x^2)^2 v + c_d \partial_x^3 v
				\end{pmatrix*}.
	\end{equation*}
	Here, recall the choice that $\kt = \kth = 1$. Furthermore, the nonlinear part of \eqref{eq:toy-model} is then given by
	\begin{equation*}
		N(U;\varepsilon) = \begin{pmatrix}
			\varepsilon^2 \alpha_u u + f(u,v) \\
			\varepsilon^2 \alpha_v v + g(u,v)
		\end{pmatrix},
	\end{equation*}
    where we interpret $\varepsilon$ as an additional variable by adding an equation $\partial_t \varepsilon =0$. We now focus on spatially periodic solutions to the system \eqref{eq:cm-formulation} and thus interpret the system as a dynamical system on the phase space
	\begin{equation*}
		\Xcal := (H^\ell_\mathrm{per}(\R))^2, 
	\end{equation*}
	for $\ell \geq 0$, where $H^\ell_\mathrm{per}(\R)$ denotes the Sobolev space of $2\pi$-spatially periodic functions from $\R$ to $\R$. Additionally, we also introduce the spaces
	\begin{equation*}
		\begin{split}
			\Zcal &:= (H^{\ell + 4}_\mathrm{per}(\R))^2, \\
			\Ycal &:= \Zcal.
		\end{split}
	\end{equation*}
	Finally, we note that the operator $L$ is a bounded operator from $(H^{\ell+4}_\mathrm{per}(\R))^2$ to $(H^{\ell}_\mathrm{per}(\R))^2$ and has the Fourier symbol
	\begin{equation*}
		\hat{L}(k) = \begin{pmatrix}
			-(1-k^2)^2 & 0 \\
			0 & -(1-k^2)^2 - i c_d k^3
		\end{pmatrix}, \qquad k \in \Z.
	\end{equation*}
	Thus, it is straightforward to see that the spectrum of $L$ splits into a stable part $\sigma_s$ with $\Re(\sigma_s) < 0$ and a set of four (counted with multiplicity) central eigenvalues eigenvalues on the imaginary axis $\sigma_0 = \{0, \pm i c_d\}$, which occur at $k = \pm 1$. In particular, we note that there is a spectral gap between the central and stable part of the spectrum. Finally, to formulate the centre manifold theorem, we introduce projections onto the central part of the spectrum by
	\begin{equation*}
		P_{0,\pm} U := \hat{U}_{\pm 1},
	\end{equation*}
	where $\hat{U}_k$ is the $k$-th Fourier mode of $U \in (H^\ell_\mathrm{per})^2$. Then, the central projection is given by $P_0 U := \ee^{\ii x} P_{0,+} U + \ee^{-\ii x} P_{0,-} U$ and the projection onto the stable part is given by $P_s := (I - P_0)$. Using this, we define $\Zcal_0 := P_0 \Zcal$ and $\Zcal_s := P_s \Zcal$.
	
	Using this notation, we are in the setting of a parameter-dependent centre manifold, see e.g.~\cite[Sec.~2.3.1.]{haragus2011} and the following theorem holds.
	
	\begin{lemma}\label{lem:centre-manifold}
		For all $\alpha_u, \alpha_u > 0$ and $c_d \in \R$ exists a neighbourhood $\Ocal_U \times \Ocal_\varepsilon$ of $(0,0)$ in $\Zcal \times \R$ and a map $\Psi : \Zcal_0 \times \R \rightarrow \Zcal_s$ satisfying
		\begin{equation*}
			\Psi(0;0) = 0 \text{ and } D_U\Psi(0;0) = 0
		\end{equation*}
		such that for all $\varepsilon \in \Ocal_\varepsilon$ the manifold
		\begin{equation*}
			\Mcal_0 := \{U_0 + \Psi(U_0;\varepsilon) \,:\, U_0 \in \Zcal_0\}
		\end{equation*}
		is a locally invariant manifold of the system \eqref{eq:cm-formulation} and $\Mcal_0$ contains all small, bounded solutions of \eqref{eq:cm-formulation}. That is, all solutions $U_0 : [0,T] \rightarrow \Zcal_0$ to the projected system
		\begin{equation}\label{eq:cm-equation-abstract}
			\partial_t U_0 = L_0 U_0 + P_0 N(U_0 + \Psi(U_0;\varepsilon))
		\end{equation}
		with $L_0 = P_0 L P_0$, which satisfy $U_0(t) \in \Ocal_U$ for all $t \in [0,T]$ yield solutions to the full system \eqref{eq:cm-formulation} via $U(t) = U_0(t) + \Psi(U(t);\varepsilon)$. Vice versa, if $U : [0,T] \rightarrow \Zcal$ is a solution to \eqref{eq:cm-formulation} satisfying $P_0 U(t) \in \Ocal_U$ for all $t \in [0,T]$ induce a solution $U_0 = P_0 U$ to \eqref{eq:cm-equation-abstract}. Finally, the reduction function $\Psi$ preserves the symmetries of the full system \eqref{eq:cm-formulation}.
	\end{lemma}
	\begin{proof}
		To obtain the result, we verify that the hypotheses 2.4, 2.7 and 3.1 of \cite{haragus2011} hold. We already established that the spectrum of the linear operator $L$ splits into a stable and central part and thus hypothesis 2.4 holds. Additionally, the mapping properties in hypothesis 3.1. can be checked in a straightforward manner since the nonlinearity $\Rcal$ is smooth. It remains to verify hypothesis 2.7 of \cite{haragus2011} and, since $\Xcal$, $\Ycal$ and $\Zcal$ are Hilbert spaces, it is sufficient to show that there is a constant $\omega_0 > 0$ such that for all $|\omega| > \omega_0$, we have that $i \omega$ belongs to the resolvent set of $L$ and the estimate
		\begin{equation*}
			\norm{(i\omega I - L)^{-1}}_{\Xcal \rightarrow \Xcal} \leq \dfrac{c}{|\omega|}
		\end{equation*}
		holds. This follows directly from the Fourier representation of $L$, in particular, the fact that the spectrum of $L$ is contained in a sector. Then, the desired result follows from an application of \cite[Thm.~3.3]{haragus2011}.
	\end{proof}

	\subsection{Normal form transform and reduced equations}
	
	We now derive the reduced equations on the centre manifold established in Lemma \ref{lem:centre-manifold} via the abstract reduced system \eqref{eq:cm-equation-abstract}. For this, we introduce following coordinates on the centre manifold by writing $U_0 \in \Zcal_0$ as
	\begin{equation*}
		U_0(t) = \tilde{A}(t) \ee^{\ii x} \begin{pmatrix}
			1 \\ 0
		\end{pmatrix} + \tilde{B}(t) \ee^{\ii x} \begin{pmatrix}
			0 \\ 1
		\end{pmatrix} + \bar{\tilde{A}}(t) \ee^{-\ii x} \begin{pmatrix}
		1 \\ 0
		\end{pmatrix} + \bar{\tilde{B}}(t) \ee^{-\ii x} \begin{pmatrix}
		0 \\ 1
		\end{pmatrix}
	\end{equation*} Using these coordinates, the reduction function $\Psi$ in Lemma \ref{lem:centre-manifold} can be written as a function of $(\tilde{A},\tilde{B},\bar{\tilde{A}},\bar{\tilde{B}};\varepsilon)$. Next, we use that $P_0 = P_{0,+} + P_{0,-}$ and thus, we obtain a system for $(\tilde{A},\tilde{B},\bar{\tilde{A}},\bar{\tilde{B}})$, which reads as
	\begin{equation*}
		\begin{split}
			 \partial_t \tilde{A} &= \tilde{\Ncal}_A(\tilde{A},\tilde{B},\bar{\tilde{A}},\bar{\tilde{B}};\varepsilon), \\
			\partial_t \tilde{B} &= -i c_d \tilde{B} +  \tilde{\Ncal}_B(\tilde{A},\tilde{B},\bar{\tilde{A}},\bar{\tilde{B}};\varepsilon), \\
			\partial_t \tilde{\bar{A}} &= \tilde{\Ncal}_{\bar{A}}(\tilde{A},\tilde{B},\bar{\tilde{A}},\bar{\tilde{B}};\varepsilon), \\
			\partial_t \tilde{\bar{B}} &= i c_d \tilde{\bar{B}} +  \tilde{\Ncal}_{\bar{B}}(\tilde{A},\tilde{B},\bar{\tilde{A}},\bar{\tilde{B}};\varepsilon).
		\end{split}
	\end{equation*}
	Here, the nonlinear terms are obtained by taking the components of $P_{0,\pm} N((\tilde{A},\tilde{B},\bar{\tilde{A}},\bar{\tilde{B}})^T+\Psi(\tilde{A},\tilde{B},\bar{\tilde{A}},\bar{\tilde{B}};\varepsilon);\varepsilon)$. Denoting $\tilde{X}_0 = (\tilde{A},\tilde{B},\bar{\tilde{A}},\bar{\tilde{B}})^T$, we abbreviate this four-dimensional system as
	\begin{equation}\label{eq:pre-NF}
		\varepsilon^2 \partial_T \tilde{X}_0 = \Lcal \tilde{X}_0 + \Ncal(\tilde{X}_0;\varepsilon), \qquad \Lcal_\varepsilon = \begin{pmatrix}
			0 & 0 & 0 & 0 \\
			0 & - i c_d & 0 & 0 \\
			0 & 0 & 0 & 0 \\
			0 & 0 & 0 & i c_d
		\end{pmatrix}
	\end{equation}
	Here, we point out that $\Ncal$ is a smooth nonlinear function in $(\tilde{X}_0;\varepsilon)$. We can now simplify this reduced system by bringing it into its normal form using a near identity change of variables. For this, we use the following parameter-dependent normal form theorem, which is a straightforward adaptation of \cite[Theorem 3.2.2]{haragus2011}.
	
	\begin{lemma}\label{lem:NF-result}
		For any integer $p \geq 2$ exist neighbourhoods $\Ocal_{X} \subset \R^4$ and $\Ocal_{\varepsilon,1} \subset \Ocal_\varepsilon \subset \R$ of $0$ such that for any $\varepsilon \in \Ocal_{\varepsilon,1}$ exists a polynomial $\Phi_\varepsilon : \R^4 \rightarrow \R^4$ of degree $p$ with the following properties. The coefficients of $\Phi_\varepsilon$ are $C^p$ in $\varepsilon$ and $\Phi_0(0) = 0$ and $D_{\tilde{X}_0} \Phi_0(0) = 0$. Furthermore, the change of variables
		\begin{equation}\label{eq:NF-trafo}
			\tilde{X}_0 = X_0 + \Phi_\varepsilon(X_0)
		\end{equation}
		transforms \eqref{eq:pre-NF} into
		\begin{equation}\label{eq:NF-abstract}
			\partial_t X_0 = \Lcal X_0 + \Ncal(X_0; \varepsilon) + \rho(X_0; \varepsilon),
		\end{equation}
		where $\Ncal : \R^4 \times \R \rightarrow \R^4$ is a polynomial of degree $p$ with $C^p$-coefficients in $\varepsilon$ satisfying
		\begin{equation}\label{eq:NF-properties}
			\Ncal(0;0) = 0, \qquad D_{X_0} \Ncal(0;0) = 0, \qquad \Ncal \left(\ee^{t \Lcal^*} X_0;\varepsilon\right) = \ee^{t \Lcal^*} \Ncal(X_0;\varepsilon),
		\end{equation}
		for all $T \in \R$ and $(X_0;\varepsilon) \in \Ocal_X \times \Ocal_{\varepsilon,1}$, where $\Lcal^*$ denotes the adjoint of $\Lcal$. Additionally, the remainder terms $\rho$ satisfy the estimate $\norm{\rho(Y_0;\varepsilon)} \leq C \norm{X_0}^{p+1}$.
		Finally, the normal form transform \eqref{eq:NF-trafo} preserves the symmetries and invariances of \eqref{eq:pre-NF}.
	\end{lemma}
	
	We now analyse the structure of the normal form of \eqref{eq:pre-NF} using the properties of the normal form \eqref{eq:NF-properties}. For this, we introduce the notation $X_0(t) = (\varepsilon A(T), \varepsilon B(T), \varepsilon \bar{A}(T), \varepsilon \bar{B}(T))$ with the slow temporal scale $T = \varepsilon^2 t$. Additionally, we note that the full system \eqref{eq:toy-model} is invariant under spatial translations $x \mapsto x + \xi$. Recalling that invariances are preserved both by the centre manifold reduction \ref{lem:centre-manifold} and the normal form transformation, we find that translation invariance is reflected in the normal form through the invariance under the transformation
	\begin{equation*}
		(A,B,\bar{A},\bar{B}) \mapsto \tau_\xi(A,B,\bar{A},\bar{B}) := (\ee^{\ii \xi}A, \ee^{\ii \xi} B, \ee^{-\ii \xi} \bar{A}, \ee^{-\ii \xi} \bar{B}).
	\end{equation*}
	In particular, both $\Ncal$ and $\rho$ in Lemma \ref{lem:NF-result} satisfy
	\begin{equation*}
		\tau_\xi \Ncal(X_0; \varepsilon) = \Ncal(\tau_\xi X_0; \varepsilon), \text{ and } \tau_\xi \rho(X_0;\varepsilon) = \rho(\tau_\xi X_0; \varepsilon).
	\end{equation*}
	
	We first use the invariances of the system and the property \eqref{eq:NF-properties} of the normal form to obtain the form of admissible terms in the polynomial nonlinearity $\Ncal$ in the normal form \eqref{eq:NF-abstract}.	Since the equations for $\bar{A}$ and $\bar{B}$ can be recovered by complex conjugation, we only discuss the nonlinear terms in the equations for $A$ and $B$, respectively. First, consider a monomial in $A^{j_1} \bar{A}^{j_2} B^{j_3} B^{j_4}$ with integers $j_1, j_2, j_3, j_4 \geq 0$. in the equation for $A$. Using invariance under $T_\xi$ coming from translation invariance, we find that
	\begin{equation*}
		\ee^{\ii\xi} A^{j_1} \bar{A}^{j_2} B^{j_3} \bar{B}^{j_4} = (\ee^{\ii \xi})^{j_1 - j_2 + j_3 - j_4} A^{j_1} \bar{A}^{j_2} B^{j_3} \bar{B}^{j_4}
	\end{equation*}
	from which we obtain the condition $j_1 - j_2 + j_3 - j_4 = 1$. Additionally, using the properties \eqref{eq:NF-properties} of the nonlinearity in normal form, we find, after recalling the form the linear operator $\Lcal$, that
	\begin{equation*}
		A^{j_1} \bar{A}^{j_2} B^{j_3} \bar{B}^{j_4} = (\ee^{\ii c_d t})^{j_3 - j_4} A^{j_1} \bar{A}^{j_2} B^{j_3} \bar{B}^{j_4},
	\end{equation*}
	which implies that $j_3 = j_4$. Together with the first condition on the powers, this yields that $j_1 - j_2 = 1$, or $j_1 = j_2 + 1$. Therefore, the monomial must be of the form $A |A|^{2q_1} |B|^{2q_2}$ for integers $q_1, q_2 \geq 0$.
	
	Applying the same arguments for a monomial in the $B$-equation we find that the conditions $j_1 - j_2 + j_3 - j_4 = 1$ and $j_3 - j_4 = 1$. This yields $j_1 = j_2$ and $j_3 = j_4 + 1$ and thus only monomials of the form $B |A|^{2q_3} |B|^{2q_4}$ for integers $q_3, q_4 \geq 0$ can occur in the $B$-equation. Therefore, applying the normal form theorem \ref{lem:NF-result} with $p = 3$, we find that the normal for in the new coordinates $(A,B,\bar{A},\bar{B})$ is given by
	\begin{equation}\label{eq:normal-form}
		\begin{split}
			\partial_T A &= \alpha_u A + A (a_1 |A|^2 + a_2 |B|^2) + \rho_A(A,B,\bar{A},\bar{B};\varepsilon) \\
			\partial_T B &= \alpha_v B - i\dfrac{c_d}{\varepsilon^2} B + B (b_1 |A|^2 + b_2 |B|^2) + \rho_B(A,B,\bar{A},\bar{B};\varepsilon)
		\end{split}
	\end{equation}
	with remainder terms satisfying $\rho_A, \rho_B = \Ocal(\varepsilon^2 |(A,B,\bar{A},\bar{B})|^4)$.
	
	\begin{remark}\label{rem:equivalence-nf-amplitude-eq}
		Note that the leading order part of the normal form \eqref{eq:normal-form} is equivalent to solutions of the amplitude system \eqref{eq:amplitude-eq-wohot}, which are space-independent. In particular, the additional $-i\varepsilon^{-2}c_d B$-term in the $B$-equation can be removed by replacing $B$ with $B \ee^{-\ii\varepsilon^{-2} c_d T} = B \ee^{-\ii c_d t}$, which reflects that typical spatially periodic solutions originating from the Turing–Hopf instability are travelling wavetrains, which move with phase velocity $c_d$.
	\end{remark}
	
	\begin{remark}
		The set $\{A = 0\}$ and $\{B = 0\}$ are not necessarily invariant in the normal form \eqref{eq:normal-form}. The monomials breaking the invariance can be removed up to an arbitrary order using normal form transformation up to higher order, but the remainder terms $\rho_A$ and $\rho_B$ might still contain them. However, in applications such as the Taylor–Couette problem, one sometimes has additional structure in the problem to guarantee the invarince. In the patter-forming system \eqref{eq:toy-model} this is for example the case, if we make the additional assumption $f(u,v) = u\Tilde{f}$ and $g(u,v) = v \Tilde{g}(u,v)$. Then $\{A = 0\}$ and $\{B = 0\}$ are invariant subspaces of \eqref{eq:normal-form} following \cite[Lemma IV.2]{chossat1994}.
	\end{remark}
	
	It remains to identify the coefficients $a_1,a_2,b_1,b_2 \in \C$ in the cubic nonlinearity of the normal form \eqref{eq:normal-form}. For this, we note that the normal form transform $\Phi_\varepsilon : \R^4 \rightarrow \R^4$ can be identified with a map $\tilde{\Phi}_\varepsilon : (H^\ell_\mathrm{per}(\R))^2 \rightarrow (H^\ell_\mathrm{per}(\R))^2$ through the mapping
	\begin{equation*}
		\begin{split}
			A \ee^{\ii x} \begin{pmatrix}
				1 \\ 0
			\end{pmatrix} + B \ee^{\ii x} \begin{pmatrix}
				0 \\ 1
			\end{pmatrix} + &\bar{A} \ee^{-\ii x} \begin{pmatrix}
				1 \\ 0
			\end{pmatrix} + \bar{B} \ee^{-\ii x} \begin{pmatrix}
				0 \\ 1
			\end{pmatrix} \\
			&\mapsto (\Phi_\varepsilon)_1 \ee^{\ii x} \begin{pmatrix}
				1 \\ 0
			\end{pmatrix} + (\Phi_\varepsilon)_2 \ee^{\ii x} \begin{pmatrix}
				0 \\ 1
			\end{pmatrix} + (\Phi_\varepsilon)_3 \ee^{-\ii x} \begin{pmatrix}
				1 \\ 0
			\end{pmatrix} + (\Phi_\varepsilon)_4 \ee^{-\ii x} \begin{pmatrix}
				0 \\ 1
			\end{pmatrix}
		\end{split}
	\end{equation*}
	with $\Phi_\varepsilon = \Phi_\varepsilon((A,B,\bar{A},\bar{B}))$. Thus, combining the normal form transform in Lemma \ref{lem:NF-result} and the centre manifold reduction in Lemma \ref{lem:centre-manifold}, we find that a solution $(A_\mathrm{nf},B_\mathrm{nf},\bar{A}_\mathrm{nf},\bar{B}_\mathrm{nf})$ to the normal form \eqref{eq:normal-form} induces a solution to the full system \eqref{eq:toy-model} through
	\begin{equation*}
		U = U_{0,\mathrm{nf}} + \tilde{\Phi}_\varepsilon(U_{0,\mathrm{nf}}) + \Psi(U_{0,\mathrm{nf}} + \tilde{\Phi}_\varepsilon(U_{0,\mathrm{nf}};\varepsilon) =: U_{0,\mathrm{nf}} + \tilde{\Psi}(U_{0,\mathrm{nf}};\varepsilon)
	\end{equation*}
	with $\Psi$ from Lemma \ref{lem:centre-manifold} and 
	\begin{equation*}
		U_{0,\mathrm{nf}}(t,x) = \varepsilon A_\mathrm{nf}(\varepsilon^2 t) \ee^{\ii x} \begin{pmatrix}
			1 \\ 0
		\end{pmatrix} + \varepsilon B_\mathrm{nf}(\varepsilon^2 t) \ee^{\ii x} \begin{pmatrix}
			0 \\ 1
		\end{pmatrix} + \varepsilon \bar{A}_\mathrm{nf}(\varepsilon^2 t) \ee^{-\ii x} \begin{pmatrix}
			1 \\ 0
		\end{pmatrix} + \varepsilon \bar{B}_\mathrm{nf}(\varepsilon^2 t) \ee^{-\ii x} \begin{pmatrix}
			0 \\ 1
		\end{pmatrix}
	\end{equation*}
	Here, we specifically point out that $(\varepsilon A_\mathrm{nf}, \varepsilon B_\mathrm{nf}, \varepsilon \bar{A}_\mathrm{nf}, \varepsilon \bar{B}_\mathrm{nf}) \in \Ocal_X$ and $U_{0,\mathrm{nf}} \in \Ocal_U$ for $\varepsilon$ sufficiently small and, therefore, both Lemma \ref{lem:NF-result} and Lemma \ref{lem:centre-manifold} indeed are applicable. In particular, we find that this is the same ansatz which has been made for the amplitude equations with the additional assumption that the amplitude modulations only depend on time; cf.~\eqref{eq:ansatz-amplitude-eq}. Hence, following the same calculations as in the formal derivation of the amplitude equations, we can recover the coefficients $a_1 = \gamma_1$, $a_2 = \gamma_2$, $b_1 = \gamma_7$ and $b_2 = \gamma_8$; see Appendix \ref{app:coefficients}.

	\subsection{Dynamics of the normal form} 
	
	We now discuss the dynamics of the normal form \eqref{eq:normal-form}. A particular emphasis lies on solutions, which are close to (time-)periodic solutions to the leading order equations without the higher order terms $\rho_A$ and $\rho_B$. We first follow Remark \ref{rem:equivalence-nf-amplitude-eq} to remove the singular term $-\varepsilon^{-2} c_d B$ in the $B$-equation in \eqref{eq:normal-form}. For this, we replace $B$ with $B_o := B \ee^{-\ii \varepsilon^{-2} c_d T}$, which yields the new system
	\begin{equation*}
		\begin{split}
			\partial_T A &= \alpha_u A + A (\gamma_1 |A|^2 + \gamma_2 |B_o|^2) + \varepsilon^2 \tilde{\rho}_A (A, B_o \ee^{\ii \varepsilon^{-2} c_d T}, \bar{A}, \bar{B}_o \ee^{-\ii \varepsilon^{-2} c_d T}; \varepsilon), \\
			\partial_T B_o &= \alpha_v B_o + B_o (\gamma_7 |A|^2 + \gamma_8 |B_o|^2) +  \varepsilon^2 \tilde{\rho}_B (A, B_o \ee^{\ii \varepsilon^{-2} c_d T}, \bar{A}, \bar{B}_o \ee^{-\ii \varepsilon^{-2} c_d T}; \varepsilon) \ee^{-\ii\varepsilon^{-2} c_d T},
		\end{split}
	\end{equation*}
	where $\rho_j = \varepsilon^2 \tilde{\rho}_j$ for $j = A, B$. Recalling the scaling of $\rho_j$ we find that $\tilde{\rho}_j =  \Ocal(|(A,B_o,\bar{A},\bar{B}_o)|^4)$ for $j = A,B$.
	This is a non-autonomous system. To recover an autonomous system, we introduce the additional variable $\theta(T) = \varepsilon^{-2} c_d T$ to obtain
	\begin{equation*}
		\begin{split}
			\partial_T A &= \alpha_u A + A (\gamma_1 |A|^2 + \gamma_2 |B_o|^2) + \varepsilon^2 \tilde{\rho}_A (A, B_o \ee^{\ii \theta}, \bar{A}, \bar{B}_o \ee^{-\ii \theta}; \varepsilon), \\
			\partial_T B_o &= \alpha_v B_o + B_o (\gamma_7 |A|^2 + \gamma_8 |B_o|^2) +  \varepsilon^2 \tilde{\rho}_B (A, B_o \ee^{\ii \theta}, \bar{A}, \bar{B}_o \ee^{-\ii \theta}; \varepsilon) \ee^{-\ii\theta}, \\
			\partial_T \theta &= \varepsilon^{-2} c_d.
		\end{split}
	\end{equation*}
	To remove the singularity in $\varepsilon$, we recall $T = \varepsilon^2 t$ and obtain the following system in the fast time $t$
	\begin{equation}\label{eq:A-B-theta-fast-time}
		\begin{split}
			\partial_t A &= \varepsilon^2 (\alpha_u A + A (\gamma_1 |A|^2 + \gamma_2 |B_o|^2)) + \varepsilon^4 \tilde{\rho}_A (A, B_o \ee^{\ii \theta}, \bar{A}, \bar{B}_o \ee^{-\ii \theta}; \varepsilon), \\
			\partial_t B_o &= \varepsilon^2 (\alpha_v B_o + B_o (\gamma_7 |A|^2 + \gamma_8 |B_o|^2)) +  \varepsilon^4 \tilde{\rho}_B (A, B_o \ee^{\ii \theta}, \bar{A}, \bar{B}_o \ee^{-\ii \theta}; \varepsilon) \ee^{-\ii\theta}, \\
			\partial_t \theta &= c_d.
		\end{split}
	\end{equation}
	Following the discussion in Section \ref{sec:time-periodic-solutions}, if the $\varepsilon^4$-terms are neglected, these equations have (component-wise) time-periodic solutions $(A,B,\theta) = (r_A \ee^{\ii \varepsilon^2 \omega_A t}, r_B \ee^{\ii \varepsilon^2 \omega_B t},\theta)$ with $r_A, r_B \geq 0$, $\omega_A, \omega_B \in \R$ and $\theta \in \mathbb{T}$. For any parameters $\gamma_1, \gamma_2, \gamma_7, \gamma_8 \in \C$, the system has a trivial solution $(r_A, r_B, \omega_A, \omega_B, \theta) = (0,0,0,0,c_d t)$. If $\alpha_u \gamma_{1,r} < 0$ and $\alpha_v \gamma_{8,r} < 0$ there are also semi-trivial solutions given by Proposition \ref{prop:semi-trivial-solutions} with $r_B = 0$ and $r_A  =0$, respectively. Finally, if the parameters satisfy the assumptions in Proposition \ref{prop:fully-nontrivial-solutions}, there are also fully non-trivial solutions with $r_A > 0$ and $r_B > 0$. The goal of this section is to show that these solutions persist if the higher-order terms $\tilde{\rho}_A$ and $\tilde{\rho}_B$ are added. 
	
	\paragraph{Case 1: Persistence of semi-trivial solutions} We first consider the persistence of the semi-trivial solutions provided by Proposition \ref{prop:semi-trivial-solutions}. Here, we restrict to the solution with $B = 0$ since persistence in the other case can be handled similarly. Note that, as the semi-trivial solutions point appears as a family parameterised by a phase shift, cf.~Lemma \ref{lem:fixed-points-fast-slow-system}. Therefore, there is a neutral direction similar to Theorem \ref{thm:persistence-heteroclinic-orbits}. Again, it turns out that using polar coordinates is a convenient choice of coordinates to isolate this addition neutral direction, which allows for the application of Theorem \ref{app-thm:hale-adapted}.
	
	Let $A^*(t) = r^*_A \ee^{\ii \omega^*_A \varepsilon^2 t}$ with $r_A^*$ and $\omega^*_A$ from Proposition \ref{prop:semi-trivial-solutions}. Then, let
	\begin{equation*}
		A(t) = (r_A^* + r_A(t)) \ee^{\ii (\omega_A^* \varepsilon^2 t + \phi_A(t))}, \quad B_o(t) = B_r(t) + i B_i(t)
	\end{equation*}
	be a solution to \eqref{eq:A-B-theta-fast-time} with $r_A(t) \geq 0$, $\phi_A(t) \in \R$ and $B_r(t), B_i(t) \in \R$ the real and imaginary part of $B_o(t)$. Using the notation $E(t) = \exp(i(\omega_A^* \varepsilon^2 t + \phi_A(t)))$, we find that $(r_A,\phi_A, B_r, B_i, \theta)$ satisfy the system
	\begin{equation*}
		\begin{split}
			\partial_t r_A &= \varepsilon^2 \left( - 2 \alpha_u r_A + \gamma_{1,r} (3r_A^* r_A^2 + r_A^3) + \gamma_{2,r} (r_A^* + r_A)(B_r^2 + B_i^2)\right) + \varepsilon^4 \Re(E^{-1} \tilde{\rho}_A), \\
			\partial_t \phi_A &= \varepsilon^2 \left(2\gamma_{1,i} r_A^* r_A + \gamma_{1,i} r_A^2 + \gamma_{2,i} (r_A^* + r_A) (B_r^2 + B_i^2)\right) + \dfrac{\varepsilon^4}{r_A^* + r_A} \Im(E^{-1} \tilde{\rho}_A), \\
			\partial_t B_r &= \varepsilon^2 \left(\alpha_v B_r + (r_A^* + r_A)^2 (B_r \gamma_{7,r} - B_i \gamma_{7,i}) + (B_r^2 + B_i^2) (B_r \gamma_{8,r} - B_i \gamma_{8,i})\right) + \varepsilon^4 \Re(\ee^{-\ii\theta}\tilde{\rho}_B), \\
			\partial_t B_i &= \varepsilon^2 \left(\alpha_v B_i + (r_A^* + r_A)^2 (B_i \gamma_{7,r} + B_r \gamma_{7,i}) + (B_r^2 + B_i^2) (B_i \gamma_{8,r} + B_r \gamma_{8,i})\right) + \varepsilon^4 \Im (\ee^{-\ii\theta} \tilde{\rho}_B), \\
			\partial_t \theta &= c_d,
		\end{split}
	\end{equation*}
	where $\tilde{\rho}_j = \tilde{\rho}_j((r_A^* + r_A) E, (B_r + i B_i) \ee^{\ii\theta}, (r_A^* + r_A) E^{-1}, (B_r - i B_i) \ee^{-\ii\theta})$ with $j = A, B$. We note in particular, that $\tilde{\rho}_j$ is $2\pi$-periodic in $\theta$ and $\phi_A$. Next, the linear part of the $(B_r,B_i)$-equations is given by
	\begin{equation*}
		\varepsilon^2 \Lcal_B = \varepsilon^2 \begin{pmatrix}
			\alpha_v + \gamma_{7,r} (r_A^*)^2 & - \gamma_{7,i} (r_A^*)^2 \\
			\gamma_{7,i} (r_A^*) & \alpha_v + \gamma_{7,r} (r_A^*)^2
		\end{pmatrix}.
	\end{equation*}
	We assume that $\Lcal_B$ has no eigenvalues on the imaginary axis, which is equivalent to $\alpha_v + \gamma_{7,r} (r_A^*)^2 \neq 0$. Then there are projections $P_s$ and $P_u$, which project onto the stable and unstable eigenspace of $\Lcal_B$, respectively. With this, we introduce $B_j := P_j(B_r,B_i)$ and $\Lcal_{B,j} = \Lcal_B P_j$ for $j = s,u$.
	Using this notation, we can rewrite the system as
	\begin{equation}\label{eq:hale-system-ST_A}
		\begin{split}
			\partial_t \begin{pmatrix}
				\theta \\ \phi_A
			\end{pmatrix} &= \begin{pmatrix}
				c_d \\ 0
			\end{pmatrix} + \varepsilon^2 \Theta(\theta, \phi_A, r_A, B_s, B_u; \varepsilon), \\
			\partial_t \begin{pmatrix}
				r_A \\ B_s
			\end{pmatrix} &= \varepsilon^2\begin{pmatrix}
				-2\alpha_u & 0 \\
				0 & \Lcal_{B,s}
			\end{pmatrix} + \varepsilon^2 F(\theta, \phi_A, r_A, B_s, B_u; \varepsilon), \\
			\partial_t B_u &= \varepsilon^2 \Lcal_{B,u} + \varepsilon^2 G(\theta, \phi_A, r_A, B_s, B_u; \varepsilon).
		\end{split}
	\end{equation}
	Here, the functions $\Theta_1$, $F$, and $G$ are given by
	\begin{equation*}
		\begin{split}
			\Theta(\theta, \phi_A, r_A, B_s, B_u; \varepsilon) &:= \begin{pmatrix}
				0 \\ 2\gamma_{1,i} r_A^* r_A + \gamma_{1,i} r_A^2 + \gamma_{2,i} (r_A^* + r_A) (B_r^2 + B_i^2) + \dfrac{\varepsilon^2}{r_A^* + r_A} \Im(E^{-1} \tilde{\rho}_A)
			\end{pmatrix}, \\
			F(\theta, \phi_A, r_A, B_s, B_u; \varepsilon) &:= \begin{pmatrix}
				\gamma_{1,r} (3r_A^* r_A^2 + r_A^3) + \gamma_{2,r} (r_A^* + r_A)(B_r^2 + B_i^2) + \varepsilon^2 \Re(E^{-1} \tilde{\rho}_A) \\
				\Ncal_{B,s}
			\end{pmatrix}, \\
			G(\theta, \phi_A, r_A, B_s, B_u; \varepsilon) &:= \Ncal_{B,u},
			\end{split}
		\end{equation*}
		where $\Ncal_{B,s}$ and $\Ncal_{B,u}$ are the stable and unstable parts of the nonlinear terms in the $(B_r,B_i)$-equations, respectively, given by
		\begin{equation*}
			\Ncal_{B,j} := P_j \begin{pmatrix}
				(2 r_A^* r_A + r_A^2) (B_r \gamma_{7,r} - B_i \gamma_{7,i}) + (B_r^2 + B_i^2) (B_r \gamma_{8,r} - B_i \gamma_{8,i}) + \varepsilon^2 \Re(\ee^{-\ii\theta}\tilde{\rho}_B) \\
				(2 r_A^* r_A + r_A^2) (B_i \gamma_{7,r} + B_r \gamma_{7,i}) + (B_r^2 + B_i^2) (B_i \gamma_{8,r} + B_r \gamma_{8,i}) + \varepsilon^2 \Im (\ee^{-\ii\theta}\tilde{\rho}_B)
			\end{pmatrix}
		\end{equation*}
		for $j = s,u$. Then, using Theorem \ref{app-thm:hale-adapted}, and reverting the centre manifold reduction Lemma \ref{lem:centre-manifold}, we obtain the following result.
		
		\begin{theorem}\label{thm:pure-patterns-global}
			Assume that $\alpha_v + \gamma_{7,r} (r_A^*)^2 \neq 0$. Then, there exist an $\varepsilon_0 > 0$ and functions $f_{ST_A} = f_{ST_A}(\theta,\phi_A;\varepsilon)$ and $g_{ST_A} = g_{ST_A}(\theta,\phi_A;\varepsilon)$, which are continuous in $\R^2 \times (0,\varepsilon_0)$, bounded and Lipschitz and $2\pi$-periodic in $\theta$ and $\phi_A$ such that
			\begin{equation*}
				\Scal_{\varepsilon}^{ST_A} := \{(\theta,\phi_A,r_A, B_s, B_u \,:\, (r_A, B_s) = f_{ST_A}(\theta, \phi_A; \varepsilon), B_u = g_{ST_A}(\theta, \phi_A;\varepsilon),\, (\theta,\phi_A) \in \R^2\} 
			\end{equation*}
			is an invariant manifold of \eqref{eq:hale-system-ST_A}. In particular, the pattern-forming system \eqref{eq:toy-model} has a bounded, spatially periodic solution of the form
			\begin{equation*}
				\begin{split}
					u(t,x) &= 2\varepsilon (r_A^* + f_{ST_A,1}(c_d t, \phi_A(t);\varepsilon)) \cos(x + \varepsilon^2 \omega_A^* t + \phi_A(t)) + \Ocal(\varepsilon^2), \\
					v(t,x) &= 2 \varepsilon \Re\left(f_{ST_A,2}(c_d t, \phi_A(t);\varepsilon) + g_{ST_A}(c_d t, \phi_A(t);\varepsilon)\right) \cos(x - c_d t) + \Ocal(\varepsilon^2),
				\end{split}
			\end{equation*}
			where $\phi_A$ is a solution to \eqref{eq:hale-system-ST_A} with $(r_A, B_s, B_u) = (f_{ST_A}, g_{ST_A})(\theta, \phi_A;\varepsilon)$.
		\end{theorem}
		\begin{proof}
			We check that the system \eqref{eq:hale-system-ST_A} satisfies the assumptions \ref{app:assumpH1}--\ref{app:assumpH6}. Then, the result follows from an application of Theorem \ref{app-thm:hale-adapted}.
			
			Since $r_A^* > 0$, \ref{app:assumpH1} is satisfied for $r_A$ in a neighbourhood of zero. For \ref{app:assumpH6}, we note that the leading order terms of $\Theta$, $F$ and $G$ do not depend on $(\theta, \phi_A)$. Additionally, as pointed out above, the higher order terms are $2\pi$-periodic in $(\theta,\phi_A)$ and thus, \ref{app:assumpH6} holds. Using the same arguments, we also find that \ref{app:assumpH2} is satisfied.
			
			For \ref{app:assumpH3}, we note that $\Theta$, $F$ and $G$ are locally Lipschitz and at least quadratic in $(r_A, B_s, B_u)$. Hypothesis \ref{app:assumpH4} follows from the observation that $\theta$ and $\phi_A$ only appear as a phase-shift and thus, $F(\theta,\phi_A,0,0,0;\varepsilon) = G(\theta,\phi_A,0,0,0;\varepsilon) = 0$. Finally, for \ref{app:assumpH5}, we point out that the spectral assumptions are satisfied by construction, using the assumption that $\alpha_v + \gamma_{8,r}(r_A^*)^2 \neq 0$ and thus, that $\Lcal_B$ has no spectrum on the imaginary axis. Additionally, since only higher order terms in $\Theta$ depend on $(\theta,\phi_A)$, we obtain that the corresponding Lipschitz constant tends to zero as $\varepsilon \rightarrow 0$. Therefore, \ref{app:assumpH5} holds and the proof is complete.
		\end{proof}
		
	
	\begin{remark}\label{rem:complicated-dynamics}
		In general, we cannot make a definitive statement regarding the temporal dynamics of the amplitudes. Specifically, we consider two numerical simulations for the example system
		\begin{equation}\label{eq:example-system}
			\begin{split}
				\partial_t A &= \varepsilon^2 A(1 - (1+5i) |A|^2 + (1+5i) |B_o|^2) + \varepsilon^4 B_o |A|^4 \ee^{\ii\theta}, \\
				\partial_t B_o &= \varepsilon^2 B_o(1 - |B_o|^2 - (2-i) |A|^2) + \varepsilon^4 A |A|^4 \ee^{-\ii\theta}, \\
				\partial_t \theta &= 1
			\end{split}
		\end{equation} 
		with initial data $(A(0),B_o(0),\theta(0)) = (1,0,0)$ for $\varepsilon_1 = 0.1$ and $\varepsilon_2 = 0.09$. In both cases, the absolute value of $A$ and $B_o$ stays close to $|A| = 1$ and $B = 0$, which is an invariant set without the $\Ocal(\varepsilon^4)$-terms. However, the $(\phi_A,\theta)$-dynamics on the two-dimensional torus changes from periodic for $\varepsilon = \varepsilon_1$ to quasiperiodic for $\varepsilon = \varepsilon_2$, see Figure \ref{fig:periodic-vs-quasiperiodic}.
	\end{remark}
	
	\begin{figure}
		\includegraphics[width=0.48\textwidth]{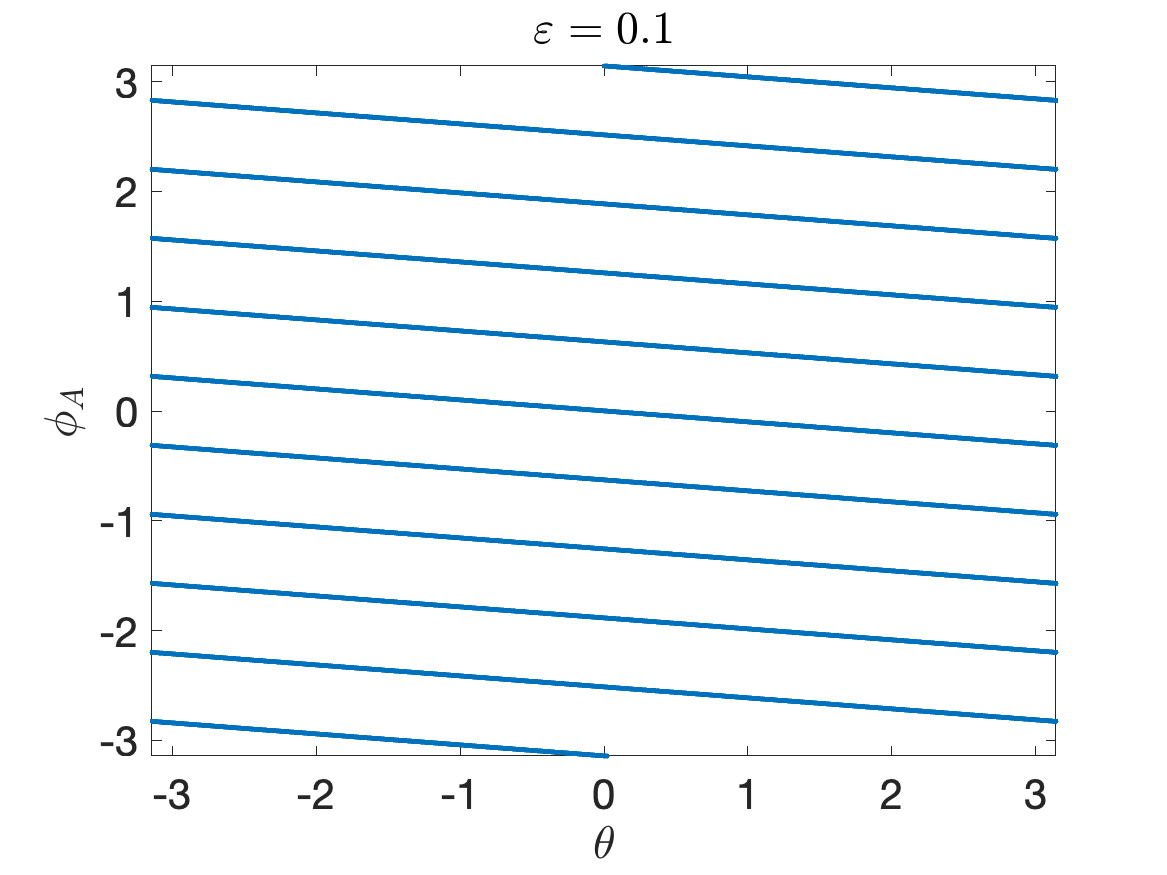}
		\includegraphics[width=0.48\textwidth]{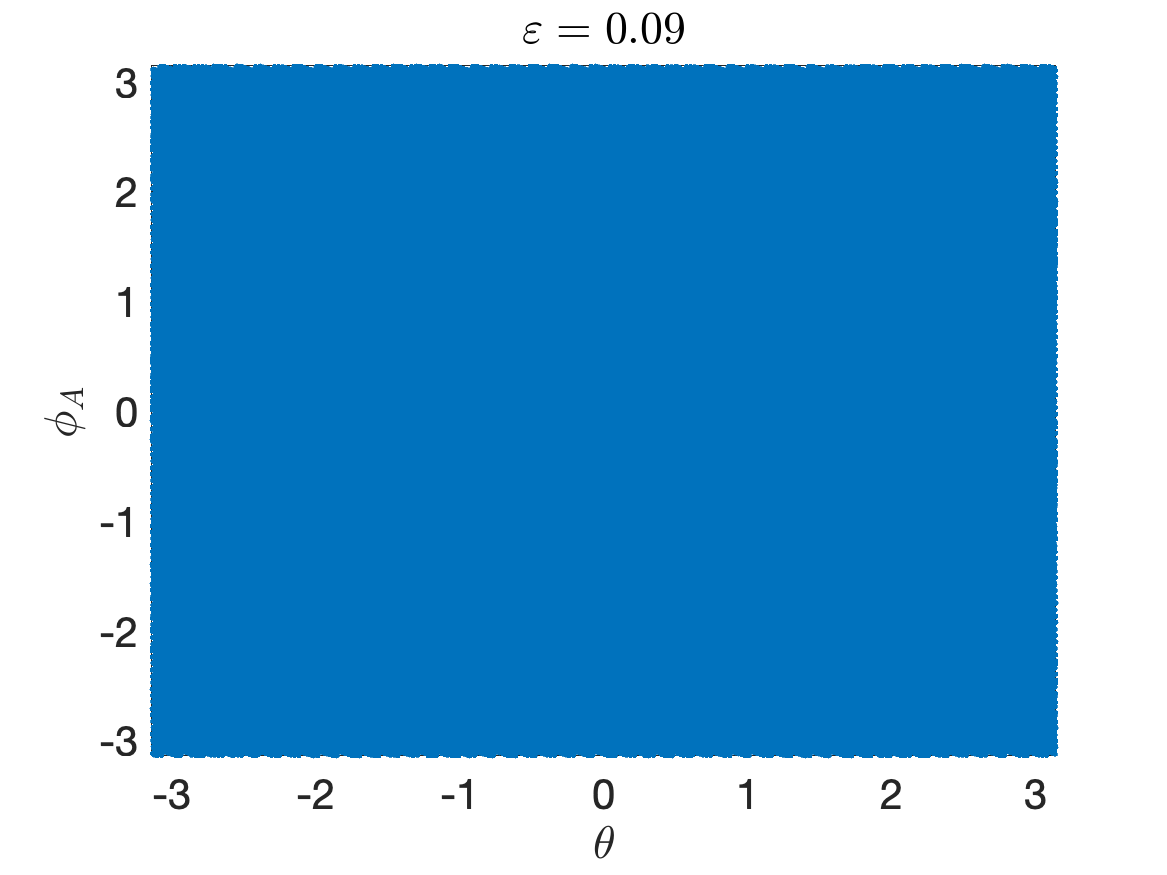}
		\caption{Angle dynamics $(\phi_A, \theta)$ on the two-dimensional torus in the example system \eqref{eq:example-system} for $\varepsilon = 0.1$ (left) and $\varepsilon = 0.09$ (right). Specifically, for $\varepsilon = 0.1$, the $(\phi_A,\theta)$-orbit is periodic on the two-dimensional torus, whereas it lies dense for $\varepsilon = 0.09$.}
		\label{fig:periodic-vs-quasiperiodic}
	\end{figure}
	
	\paragraph{Case 2: Persistence of fully non-trivial solutions}
	We now consider the persistence of the fully non-trivial solutions, that is both $A \neq 0$ and $B \neq 0$, which are obtained in Proposition \ref{prop:fully-nontrivial-solutions}. We proceed similar, however, we now use polar coordinates for both $A$ and $B$. Therefore, let the parameter conditions of Proposition \ref{prop:fully-nontrivial-solutions} be satisfied and let $A^*(t) = r_A^* \ee^{\ii \omega_A^* \varepsilon^2 t}$ and $B^*(t) = r_B^* \ee^{\ii\omega_B^* \varepsilon^2 t}$ be the non-trivial, time-periodic solution to the leading order terms of \eqref{eq:A-B-theta-fast-time}. Then, let
	\begin{equation*}
		A(t) = (r_A^* + r_A(t)) \ee^{\ii (\omega_A^* \varepsilon^2 t + \phi_A(t))}, \quad B(t) = (r_B^* + r_B(t)) \ee^{\ii (\omega_B^* \varepsilon^2 t + \phi_B(t))}
	\end{equation*}
	be a solutions to \eqref{eq:A-B-theta-fast-time} with $r_A(t), r_B(t) \geq 0$ and $\phi_A(t), \phi_B(t) \in \R$. Using the notation $E_A(t) = \ee^{\ii (\omega_A^* \varepsilon^2 t + \phi_A(t))}$ and $E_B(t) = \ee^{\ii (\omega_B^* \varepsilon^2 t + \phi_B(t))}$, we obtain the system for $(r_A,r_B,\phi_A,\phi_B,\theta)$, which reads as
	\begin{equation}\label{eq:hale-system-NT}
		\begin{split}
			\partial_t r_A &= \varepsilon^2 \left((\alpha_u + 3 \gamma_{1,r} (r_A^*)^2 + \gamma_{2,r}(r_B^*)^2) r_A + 2 \gamma_{2,r} r_A^* r_B^* r_B\right) \\
			&\qquad + \varepsilon^2 \left(\gamma_{1,r} \left(3 r_A^* r_A^2 + r_A^3\right) + \gamma_{2,r} \left(2 r_B^* r_A r_B + r_A^* r_B^2 + r_A r_B^2\right)\right) + \varepsilon^4 \Re(E_A^{-1} \tilde{\rho}_A), \\
			\partial_t r_B &= \varepsilon^2 \left(2 \gamma_{7,r}r_A^* r_B^* r_A + (\alpha_v + \gamma_{7,r} (r_A^*)^2 + 3 \gamma_{8,r} (r_B^*)^2) r_B\right) \\
			&\qquad + \varepsilon^2 \left(\gamma_{7,r} \left(r_B^* r_A^2 + 2 r_A^* r_A r_B + r_B r_A^2\right) + \gamma_{8,r} (3 r_B^* r_B^2 + r_B^3)\right) + \varepsilon^4 \Re(\ee^{-\ii\theta} E_B^{-1} \tilde{\rho}_B), \\
			\partial_t \phi_A &= \varepsilon^2\left(\gamma_{1,i} \left(2 r_A^* r_A + r_A^2\right) + \gamma_{2,i} \left(2 r_B^* r_B + r_B^2\right)\right) + \dfrac{\varepsilon^4}{r_A^* + r_A} \Im(E_A^{-1} \tilde{\rho}_A), \\
			\partial_t \phi_B &= \varepsilon^2\left(\gamma_{7,i} \left(2 r_A^* r_A + r_A^2\right) + \gamma_{8,i} \left(2 r_B^* r_B + r_B^2\right)\right) + \dfrac{\varepsilon^4}{r_B^* + r_B} \Im(\ee^{-\ii\theta} E_B^{-1} \tilde{\rho}_B), \\
			\partial_t \theta &= c_d.
		\end{split}
	\end{equation}
	To apply Theorem \ref{app-thm:hale-adapted} to the system \eqref{eq:hale-system-NT}, we have to assume that
	\begin{equation*}
		\Lcal_{NT} := \begin{pmatrix}
			\alpha_u + 3 \gamma_{1,r} (r_A^*)^2 + \gamma_{2,r} (r_B^*)^2 & 2 \gamma_{2,r} r_A^* r_B^* \\
			2 \gamma_{7,r} r_A^* r_B^* & \alpha_v + \gamma_{7,r} (r_A^*)^2 + 3 \gamma_{8,r} (r_B^*)^2
		\end{pmatrix}
	\end{equation*} 
	has no eigenvalues on the imaginary axis. In fact, this is equivalent to assuming that the non-trivial equilibrium point $NT$ in the radii system \eqref{eq:radii-dynamics} is hyperbolic. Then, proceeding as in the first case, we obtain the following result.
	
	\begin{theorem}\label{thm:superposition-pattern-global}
		Assume that $NT$ given by Proposition \ref{prop:fully-nontrivial-solutions} is a hyperbolic equilibrium point of \eqref{eq:radii-dynamics}. Then, there exists a $\varepsilon_0 > 0$ and a function $f_{NT} = f_{NT}(\theta, \phi_A, \phi_B; \varepsilon)$ which is continuous in $\R^3 \times (0,\varepsilon_0)$, bounded and Lipschitz and $2\pi$-periodic in $\theta$, $\phi_A$ and $\phi_B$ such that
		\begin{equation*}
			\Scal^{NT}_\varepsilon := \{(r_A, r_B, \phi_A, \phi_B, \theta) \,:\, (r_A, r_B) = f_{NT}(\theta, \phi_A,\phi_B;\varepsilon), \; (\theta,\phi_A,\phi_B) \in \R^3)\}
		\end{equation*}
		is an invariant manifold of \eqref{eq:hale-system-NT}. In particular, the pattern-forming system \eqref{eq:toy-model} has bounded, spatially periodic solutions of the form
		\begin{equation*}
			\begin{split}
				u(t,x) &= 2 \varepsilon (r_A^* + f_{NT,A}(c_d t,\phi_A(t), \phi_B(t))) \cos(x + \varepsilon^2 \omega_A^* t + \phi_A(t)) + \Ocal(\varepsilon^2), \\
				v(t,x) &= 2 \varepsilon  (r_B^* + f_{NT,B}(c_d t,\phi_A(t), \phi_B(t))) \cos(x - c_d t + \varepsilon^2 \omega_B^* t + \phi_B(t)) + \Ocal(\varepsilon^2),
			\end{split}
		\end{equation*}
		where $\theta_A, \theta_B$ satisfy \eqref{eq:hale-system-NT} with $(r_A,r_B) = f_{NT}(\theta, \phi_A,\phi_B;\varepsilon)$.
	\end{theorem}
	
	\section{Discussion}\label{sec:discussion}
	
	In this paper, we have studied the dynamics of a pattern-forming system close to the onset of a 1:1 resonant Turing and Turing–Hopf instability. We have derived and justified a system of two coupled Ginzburg–Landau equations with a singular advection term as amplitude equations and established the existence of space-time periodic solutions and fast-moving travelling front solutions to the amplitude equations. The front solutions correspond to fast-moving pattern interfaces in the original pattern-forming systems, which exist on a long, but finite, time interval. Finally, we establish the existence of globally bounded, spatially periodic solutions to the pattern-forming system.
	
	We conclude by giving a brief discussion of open and related problems.
	
	\paragraph{Global existence of modulating front solutions.} Using the amplitude system, we found a variety of front solutions connecting different patterns in the full pattern-forming system. However, since the rigorous approximation result is only valid on a long, but finite time interval, it remains an open question whether these pattern interfaces can be established globally in time. A natural approach is to extend the spatial dynamics and centre manifold approach used to construct modulating front solutions, see e.g.~\cite{eckmann1991}. The main issue, however, is the presence of highly oscillatory higher-order terms due to the different phase velocity of the Turing and Turing–Hopf instability. In particular, the typical ansatz of the form $u(t,x) = U(x-ct, x-c_pt)$ does not seem sufficient to capture these fast oscillating terms. Furthermore, it also cannot capture the spatial connection between pattern with different phase speeds. Therefore, a more general ansatz is necessary and the construction of modulating fronts closet the resonant instability remains an open question.
	
	
	\paragraph{Extension to other models and related instabilities.} Another natural question is the extension to other models and related instabilities. We expect that our proof carries over to other model with the same 1:1 resonant Turing and Turing–Hopf instability since we consider a very general class of nonlinearities. Therefore, we discuss the extension to related instabilities here, namely, the extension to resonant Turing and Turing–Hopf instabilities with non-equal critical wave numbers and the extension to resonant Turing and Turing–Hopf instabilities in reflection symmetric systems such as the Taylor–Couette problem.
	
	In the first case, we do not expect any new technical challenges, see also Remark \ref{rem:different-resonances}. However, the resulting dynamics is potentially very different due to the different interaction terms in the amplitude system, see also \cite{porter2001, porter2000} for a discussion of spatially periodic solutions in a 1:2 and 1:3 resonant instability of two Turing modes.
	
	In the later case, the main difference in the instability is that the Turing–Hopf spectral curve appears as complex conjugated pair due to the reflection symmetry. In fact, this is the main motivation in \cite{schneider1997} to consider the interaction of two Turing–Hopf instabilities with opposite phase and group velocities. However, we do not expect any significant difficulties in the justification of the amplitude system, which now consists of three coupled Ginzburg–Landau equations (one for corresponding to the Turing instability, and two to the reflection symmetric Turing–Hopf instability). Nevertheless, constructing coherent structures in the higher-dimensional amplitude system is much more challenging, although the problem might become tractable if restricted to a suitable subspace, e.g. if both Turing–Hopf modes are the same, which is an invariant subspace of the amplitude system due to reflection symmetry.
	
	\paragraph{Temporal dynamics on the centre manifold for spatially periodic solutions.}
	
	Lastly, we discuss the extension of the centre manifold analysis in Section \ref{sec:periodic-solutions} to global temporal dynamics on the centre manifold as in the case of two resonant steady state instabilities, see e.g.~\cite{guckenheimer1986,armbruster1988}. Although a similar analysis can be carried out for the normal form \eqref{eq:normal-form} after neglecting the higher-order terms $\rho_A$ and $\rho_B$, the main challenge is to show the persistence of obtained solutions. Any persistence argument has to handle the highly oscillating higher-order terms arising from the different phase velocities. In fact, Remark \ref{rem:complicated-dynamics} demonstrates that this is a subtle task even for solutions close to simple orbits of the cut-off normal form. Therefore, the construction of global temporal dynamics on the centre manifold remains an open problem.
	
	
	\SkipTocEntry\section*{Acknowledgements}
	
	B.H. was partially supported the Deutsche Forschungsgemeinschaft (DFG, German Research Foundation) – Project-IDs 444753754 and 543917644.
	
	\SkipTocEntry\section*{Data Availability Statement}
	
	The explicit formulas for the coefficients, see Appendix \eqref{app:coefficients}, were computed in Mathematica \cite{Mathematica}. The plots of the phase planes, Figures \ref{fig:phase-planes-d<0}, \ref{fig:phase-plane-d>0} and \ref{fig:phase-plane-NT-non-ex}, were also generated using Mathematica \cite{Mathematica}. The plots of the fronts in Figures \ref{fig:T-to-trivial}, \ref{fig:TH-to-trivial}, \ref{fig:T-to-TH}, \ref{fig:superpos-to-T} and \ref{fig:TH-to-superpos} were generated by numerically calculating the heteroclinic orbits in the reduced system \eqref{eq:slow-dynamics} using Matlab \cite[R2024a]{themathworksinc.} and then inserting into the leading-order expressions \eqref{eq:approx-form-solutions}. Finally, the angle dynamics in Figure \ref{fig:periodic-vs-quasiperiodic} is generated by numerically integrating the system \eqref{eq:example-system} using Matlab \cite[R2024a]{themathworksinc.}. The code used to generate the corresponding data and videos is available at \href{https://github.com/Bastian-Hilder/Resonant-Turing-and-Turing-Hopf}{https://github.com/Bastian-Hilder/Resonant-Turing-and-Turing-Hopf}.
	
	\appendix

    \section{Outline of a proof of Theorem \ref{thm:approximation-full-amplitude-eq}}\label{app:justification}

    In this appendix, we outline the main steps of the justification result Theorem \ref{thm:approximation-full-amplitude-eq}. As discussed above, this follows the standard strategy, see e.g.~\cite{schneider1994a,schneider1997,gauss2021} as well as \cite[Chap.~10]{schneider2017a} for a general overview. For this, we recall $\Psi = \Psi_\mathrm{GL} + \Psi_\mathrm{hot}$ with
    \begin{equation*}
        \begin{split}
    		\Psi_\mathrm{GL}(T,X_1,X_2,t) &= \varepsilon A(T,X_1) \ee^{\ii x} \begin{pmatrix}
    		    1 \\ 0
    		\end{pmatrix} + \varepsilon B(T,X_2) \ee^{\ii (x-c_p t)} \begin{pmatrix}
    		    0 \\ 1
    		\end{pmatrix} + c.c \\
    		\Psi_\mathrm{hot}(T,X_1,X_2,t), &= \varepsilon^2 \left(\dfrac{1}{2}A_0(T,X_1,t) + A_2(T,X_1,t) \ee^{2\ii x} + c.c.\right) \begin{pmatrix}
    		    1 \\ 0
    		\end{pmatrix} \\
            &\qquad+ \varepsilon^2 \left(\dfrac{1}{2}B_0(T,X_2,t) + B_2(T,X_2,t) \ee^{\ii(x-c_p t)} + c.c.\right) \begin{pmatrix}
		    0 \\ 1
		\end{pmatrix}.
        \end{split}
	\end{equation*}
    Formally balancing different powers of $\varepsilon$ at different Fourier modes $e^{ikx}$ then yields that the quadratic correction terms $A_0, B_0, A_2, B_2$ are given by \eqref{eq:correction-terms-0} and $\eqref{eq:correction-terms-2}$, while $A$ and $B$ are solutions of the amplitude equations \eqref{eq:full-amplitude-eq}.

    We now write the pattern-forming system \eqref{eq:toy-model} in abstract form as
    \begin{equation}\label{eq:abstract-system}
        \partial_t w = \Lambda w + N_2(w,w) + N_3(w,w,w),
    \end{equation}
    where $w = (u,v)$, $\Gamma$ contains the linear terms, $N_2$ is a symmetric bilinear form containing the quadratic terms and $N_3$ is a trilinear form containing the cubic terms. To streamline the following sketch, we will neglect the cubic terms from now on. In fact, it is well-known that the main difficulty in justification results such as Theorem \ref{thm:approximation-full-amplitude-eq} comes from quadratic nonlinearities. Therefore, the addition of cubic (or higher-order) nonlinearities poses no additional difficulties.
    
    We write a solution $w = \Psi + \varepsilon^2 R$ with an error term $R$ and $\beta > 0$ to be determined later. Inserting this into \eqref{eq:abstract-system} yields the error equation
    \begin{equation*}
        \partial_t R = \Lambda R + 2 N_2(\Psi,R) + \varepsilon^2 N_2(R,R) + \varepsilon^{-2} \operatorname{Res}(\Psi),
    \end{equation*}
    where $\operatorname{Res}(\Psi)$ denotes the residual terms given by
    \begin{equation*}
        \operatorname{Res}(\Psi) := -\partial_t \Psi + \Lambda\Psi + N_2(\Psi,\Psi).
    \end{equation*}
    The remaining justification follows the following steps. First, one establishes an estimate for the residual using that the ansatz $\Psi$ is chosen such that formally, the lowest orders of $\varepsilon$ are equated to zero if the amplitudes $A$ and $B$ satisfy the amplitude equations \eqref{eq:full-amplitude-eq}. Second, we use the residual estimate to establish the existence of a solution to the error equation for a sufficiently long time interval and obtain uniform bounds on the solution with respect to $\varepsilon$.

    The key idea introduced in \cite{schneider1994a} to handle quadratic terms is to define so-called mode filters
    \begin{equation*}
        E_c w = \Fcal^{-1}(\chi_c \Fcal(w)), \quad E_s w = \Fcal^{-1}(\chi_s \Fcal(w)),
    \end{equation*}
    where $\Fcal$ denotes the Fourier transformation and smooth and compactly supported cut-off functions $\chi_j \in C^\infty_0(\R;[0,1])$ with
    \begin{equation*}
        \chi_c(k) = \begin{dcases}
            1, & \text{if } ||k| - 1| \leq \tfrac{1}{30}, \\
            0, & \text{if } ||k| - 1| \geq \tfrac{1}{15}
        \end{dcases}
    \end{equation*}
    and $\chi_s = 1-\chi_c$. These are bounded maps from $H^\ell_{l,u}$ to $H^\ell_{l,u}$, see e.g.~\cite[Lemma 5]{schneider1994}. The mode filters distinguish the Fourier modes close to the critical wave number pair $k_c = \pm 1$ and the exponentially damped modes. This separation is particularly relevant to handle quadratic nonlinearities since it allows us to exploit the fact that quadratic combinations of the critical Fourier modes $e^{\pm ix}$ are always located at exponentially damped modes. Here, we also exploit that the Turing and Turing–Hopf instabilities occur at the same wave number. However, a similar construction also works, for example, in the case of two Turing instabilities with a spatial 1:2 resonance, see \cite{gauss2021}.

    Under the assumptions of Theorem \ref{thm:approximation-full-amplitude-eq} we obtain that
    \begin{equation}\label{eq:residual-estimate}
        \sup_{t \in [0,T_0/\varepsilon^2]} \left(\varepsilon^{-4} \|E_c \Res(\Psi)\|_{(H^1_{l,u})^2} + \varepsilon^{-3}\|E_s \Res(\Psi)\|_{(H^1_{l,u})^2}\right) \leq C_{\Res}.
    \end{equation}
    This follows from the construction of the ansatz $\Psi = \Psi_{\mathrm{GL}} + \Psi_\mathrm{hot}$ and additionally exploiting that the contributions at $\varepsilon^3 e^{\pm ix}$ vanish since $A$ and $B$ satisfy the amplitude equations \eqref{eq:full-amplitude-eq}. Here, we also use that $A(T)$ and $B(T)$ are bounded in $H_\mathrm{l,u}^5$. This yields a bound of the residual in $H^1_\mathrm{l,u}$ since the residual contains at most four derivatives, see also \cite[Sec.~10.2.2]{schneider2017a}.

    To control the error $R$, we write $R = R_c + \varepsilon R_s$ with $R_j = E_j R$, $j = s,c$. We additionally write $\Psi = \Psi_c + \varepsilon\Psi_s$ with $\Psi_j = E_j \Psi_j$ for $j = s,c$, and we note that by construction of the ansatz and the mode filters, we obtain the error bound
    \begin{equation}\label{eq:bound-psi_c-psi_s}
        \sup_{t \in [0,T/\varepsilon^2]} \left(\|\Psi_c\|_{(C^1_b)^2} + \|\Psi_s\|_{(C^1_b)^2}\right) \leq C_\Psi \varepsilon
    \end{equation}
    for some constant $C_\Psi$. The error equation can then be split into equations for $R_c$ and $R_s$
    \begin{equation*}
        \begin{split}
            \partial_t R_c &= \Lambda R_c + 2 E_c(\varepsilon N_2(\Psi_c, R_s) + \varepsilon N_2(\Psi_s, R_c) + \varepsilon^2 N_2(\Psi_s, R_s)) \\
            &\qquad + E_c(\varepsilon^2 N_2(R,R)) + \varepsilon^{-2} E_c\Res(\Psi) \\
            \partial_t R_s &= \Lambda R_s + 2 E_s(\varepsilon^{-1} N_2(\Psi_c, R_c) + N_2(\Psi_s, R_c) + N_2(\Psi_c, R_s) + \varepsilon N_2(\Psi_s, R_s)) \\
            &\qquad + E_s(\varepsilon N_2(R,R)) + \varepsilon^{-3} E_s\Res(\Psi).
        \end{split}
    \end{equation*}
    As outlined above, here we used that quadratic interactions of critical modes are non-critical, that is $E_c(N_2(\Psi_c,R_c)) = 0$. Now, using \eqref{eq:bound-psi_c-psi_s}, we find that
    \begin{equation*}
        \begin{split}
            \|\varepsilon E_cN_2(\Psi_c,R_s)\|_{(H^1_{l,u})^2} &\lesssim \varepsilon \|\Psi_c\|_{(C^1_b)^2} \|R_s\|_{(H^1_{l,u})^2} \lesssim \varepsilon^2 \|R_s\|_{(H^1_{l,u})^2} \\
            \|\varepsilon E_cN_2(\Psi_s,R_c)\|_{(H^1_{l,u})^2} &\lesssim \varepsilon \|\Psi_s\|_{(C^1_b)^2} \|R_c\|_{(H^1_{l,u})^2} \lesssim \varepsilon^2 \|R_c\|_{(H^1_{l,u})^2} \\
            \|\varepsilon^{-1} E_s N_2(\Psi_c,R_c)\|_{(H^1_{l,u})^2} &\lesssim \varepsilon^{-1} \|\Psi_c\|_{(C^1_b)^2} \|R_c\|_{(H^1_{l,u})^2} \lesssim \|R_c\|_{(H^1_{l,u})^2} \\
            \|E_s N_2(\Psi_s,R_c)\|_{(H^1_{l,u})^2} &\lesssim \|\Psi_s\|_{(C^1_b)^2} \|R_c\|_{(H^1_{l,u})^2} \lesssim \varepsilon\|R_c\|_{(H^1_{l,u})^2} \\
            \|E_s N_2(\Psi_c,R_s)\|_{(H^1_{l,u})^2} &\lesssim \|\Psi_c\|_{(C^1_b)^2} \|R_s\|_{(H^1_{l,u})^2} \lesssim \varepsilon \|R_s\|_{(H^1_{l,u})^2} 
        \end{split}
    \end{equation*}
    Therefore, together with the residual estimate \eqref{eq:residual-estimate}, we obtain that the error equations qualitatively behave like
    \begin{equation*}
        \begin{split}
            \partial_t R_c &= \Lambda R_c + \Ocal\left(\varepsilon^2 \|R_c\|_{(H^1_{l,u})^2} + \varepsilon^2 \|R_s\|_{(H^1_{l,u})^2}\right) + \Ocal(\varepsilon^2) \\
            \partial_t R_s &= \Lambda R_s + \Ocal\left(\|R_c\|_{(H^1_{l,u})^2} + \varepsilon \|R_s\|_{(H^1_{l,u})^2}\right) + \Ocal(\varepsilon) \\
        \end{split}
    \end{equation*}
    as long as $\|R_c\|_{(H^1_{l,u})^2}$ and $\|R_s\|_{(H^1_{l,u})^2}$ remain bounded. As outlined in \cite[Sec.~10.4]{schneider2017a} we can now obtain an estimate for $R_s$ by using that $E_s$ maps onto modes which are exponentially damped by the semigroup generated by $\Lambda$. This yields an estimate of the form
    \begin{equation*}
        \sup_{\tau \in [0,t]} \|R_s(\tau)\|_{(H^1_{l,u})^2} \lesssim \sup_{\tau \in [0,t]} \|R_s(\tau)\|_{(H^1_{l,u})^2} + 1
    \end{equation*}
    for all $t \geq 0$ that $R_s$ and $R_c$ remain bounded. Then, we can use this bound to obtain an estimate of the form
    \begin{equation*}
        \sup_{\tau \in [0,t]} \|R_c(\tau)\|_{(H^1_{l,u})^2} \lesssim 1 + \varepsilon ^2\int_0^t \sup_{\tau \in [0,t]} \|R_c(\tau)\|_{(H^1_{l,u})^2} \dd t
    \end{equation*}
    as long as $R_s$ and $R_c$ remain bounded. An application of Gronwall's inequality then yields the justification result. For more details, we refer to \cite[Thm.~10.4.3]{schneider2017a}.

    \begin{remark}
        We point out that the mode filters $E_j$, $j = s,c$ do not define projections on $H^1_{l,u}$. Therefore, it is necessary to introduce auxiliary mode filters $E^h_j$ such that $E_j^h E_j = E_j$ for $j = s,c$ and then prove the semigroup estimates
        \begin{equation*}
            \begin{split}
                \|\ee^{\Lambda t} E_c^h\|_{(H^1_{l,u})^2 \rightarrow (H^1_{l,u})^2} \lesssim \ee^{\sigma_c \varepsilon^2 t} \\
                \|\ee^{\Lambda t} E_s^h\|_{(H^1_{l,u})^2 \rightarrow (H^1_{l,u})^2} \lesssim \ee^{-\sigma_s t}
            \end{split}
        \end{equation*}
        for constants $\sigma_c, \sigma_s > 0$, cf.~\cite[Thm.~10.4.3, Assumption (A2)]{schneider2017a}.
        It turns out that $E^h_j$ can be chosen similarly to the original mode filters $E_j$, see e.g.~\cite{schneider1994}.
    \end{remark}
	
	\section{Persistence of periodic solutions}\label{app:persistence}
	
	Here, we provide an adaptation of the result from \cite[Theorem VII.2.3]{hale1969} to our setting. For this, we consider the system
	\begin{equation}\label{app-eq:ode-hale}
		\begin{split}
			\dot{\theta} &= \omega + \varepsilon \Theta(\theta,x,y,\varepsilon), \\
			\dot{x} &= \varepsilon \Acal x + \varepsilon F(\theta,x,y,\varepsilon), \\
			\dot{y} &= \varepsilon \Bcal y + \varepsilon G(\theta, x,y,\varepsilon)
		\end{split}
	\end{equation}
	where $\theta = \theta(t) \in \R^k$, $\omega \in \R^k$, and $(x(t),y(t),\varepsilon) \in \Omega(\sigma, \varepsilon_0)$ with
	\begin{equation*}
		\Omega(\sigma,\varepsilon_0) := \{(x,y,\varepsilon) \in \R^{n_1} \times \R^{n_2} \times (0,\infty) \,:\, |x| < \sigma, |y| < \sigma, \varepsilon \in (0,\varepsilon_0)\}, \quad n_1, n_2 \in \N.
	\end{equation*}
	In addition, we assume the following hypotheses.
	\begin{enumerate}[label=(H\arabic*)]
		\item\label{app:assumpH1} The functions $\Theta, F, G$ are continuous and bounded in $\R^k \times \Omega(\sigma, \varepsilon_0)$.
		\item\label{app:assumpH2} The function $\Theta, F, G$ are Lipschitz in $\theta$ with constants $\eta(\rho,\varepsilon_1), \gamma(\rho,\varepsilon_1), \gamma(\rho,\varepsilon_1)$ in $\R^k \times \Omega(\rho,\varepsilon_1)$. Here, $\eta$ and $\gamma$ are continuous and nondecreasing for $0 \leq \rho \leq \sigma$ and $0 \leq \varepsilon_1 \leq \varepsilon_0$. Moreover, $\gamma(0,0) = 0$.
		\item\label{app:assumpH3} The functions $\Theta, F, G$ are Lipschitz in $(x,y)$ with constants $\mu(\rho, \varepsilon_1), \delta(\rho,\varepsilon_1), \delta(\rho,\varepsilon_1)$ in $\R^k \times \Omega(\rho, \varepsilon_1)$. Additionally, $\mu$ and $\delta$ are continuous and nondecreasing for $0 \leq \rho \leq \sigma$ and $0 \leq \varepsilon_1 \leq \varepsilon_0$. Moreover, $\delta(0,0) = 0$.
		\item\label{app:assumpH4} It holds that $|F(\theta,0,0,\varepsilon)|, |G(\theta,0,0,\varepsilon)| \leq N(\varepsilon)$ for all $\theta \in \R^k$ and $0 \leq \varepsilon \leq \varepsilon_0$. Here, $N$ is continuous and nondecreasing for $0 \leq \varepsilon \leq \varepsilon_0$ and $N(0) = 0$.
		\item\label{app:assumpH5} There exists a $\alpha_1 > 0$ such that the spectrum of $\Acal \in \R^{n_1 \times n_1}$ and $\Bcal \in \R^{n_2 \times n_2}$ satisfies
		\begin{equation*}
			\sigma(\Acal) < - \alpha_1 < 0 < \alpha_1 < \sigma(\Bcal).
		\end{equation*}
		Additionally, $\alpha_1 - \limsup_{\varepsilon \rightarrow 0} \eta(0,\varepsilon) > 0$.
		\item\label{app:assumpH6} The functions $\Theta,F,G$ are $\tilde{\omega}$-periodic in $\theta$ with $\tilde{\omega} > 0$.
	\end{enumerate}
	
	\begin{remark}
		The last condition in Assumption \ref{app:assumpH5} is also called normal hyperbolicity at higher order, cf.~\cite{verhulst2023}.
	\end{remark}
	
	\begin{theorem}\label{app-thm:hale-adapted}
		Assume that there exists a $\sigma > 0$ and $\varepsilon_0 > 0$ such that the hypotheses \ref{app:assumpH1}--\ref{app:assumpH6} are satisfied. Then, there exist functions $f = f(\theta, x,y) \in \R^{n_1}$ and $g = g(t,\theta,\varepsilon) \in \R^{n_2}$, which are continuous in $\R^k \times (0,\varepsilon_0)$, bounded, Lipschitz and $\tilde{\omega}$-periodic in $\theta$ such that
		\begin{equation}\label{app-eq:invariant-manifold}
			\Scal_\varepsilon := \{(\theta, x,y) \,:\,  x = f(\theta,\varepsilon), y = g(\theta,\varepsilon), \theta \in \R^k\}
		\end{equation}
		is an invariant manifold of \eqref{app-eq:ode-hale}.
	\end{theorem}
	\begin{proof}
		It is straightforward to see that under the assumptions \ref{app:assumpH1}--\ref{app:assumpH5}, the result \cite[Theorem VII.2.3]{hale1969} applies and provides functions $f = f(t,\theta,\varepsilon)$ and $g = g(t,\theta,\varepsilon)$, which are continuous in $\R \times \R^k \times (0,\varepsilon_0)$, bounded and Lipschitz in $\theta$ such that
		\begin{equation*}
			\tilde{\Scal}_\varepsilon := \{(t,\theta,x,y) \,:\, x = f(t,\theta,\varepsilon), y = g(t,\theta,\varepsilon), (t,\theta) \in \R \times \R^k\}
		\end{equation*}
		is an integral manifold of \eqref{app-eq:ode-hale}, that is, if $P \in \tilde{\Scal}_\varepsilon$ is an arbitrary point and $(\theta(t), x(t), y(t))$ is a solution to \eqref{app-eq:ode-hale} which passes through $P$, then $(t,\theta(t),x(t),y(t)) \in \tilde{\Scal}_\varepsilon$ for all $t \in \R$. In addition, if all functions in \eqref{app-eq:ode-hale} are $\tilde{\omega}$-periodic in $\theta$, then so are $f$ and $g$. Similarly, if all functions in \eqref{app-eq:ode-hale} are $T$-periodic in $t$, then so are $f$ and $g$.
		
		In the case at hand, all functions in \eqref{app-eq:ode-hale} are independent of time $t$ and thus are periodic with any period. Therefore, $f$ and $g$ are periodic in $t$ with any period and hence must be independent of $t$ as well. Using this, we can write $\tilde{\Scal}_\varepsilon = \R \times \Scal_\varepsilon$ with $\Scal_\varepsilon$ in \eqref{app-eq:invariant-manifold} and $\Scal_\varepsilon$ must be invariant. Finally, the periodicity in $\theta$ follows from assumption \ref{app:assumpH6}.
	\end{proof}

	\section{Coefficients}\label{app:coefficients}
	The coefficients for the amplitude equations are explicitly given by
	{\allowdisplaybreaks\small
    \begin{align*}
        \gamma_1 &= \frac{f_{11} g_{20} \left(16 c_p-19 i\right)}{8 c_p-9 i}+\frac{38
            f_{20}^2}{9}+3 f_{30}, \\
        \gamma_2 &= \frac{2 \left(\left(7 c_p-9 i\right) \left(f_{11} \left(5 f_{11}+g_{02}
            \left(9+c_p \left(c_p+8 i\right)\right)\right)+\left(f_{12}+2 f_{02}
            f_{20}\right) \left(9+c_p \left(c_p+8 i\right)\right)\right)\right)}{\left(c_p-i\right)
            \left(c_p+9 i\right) \left(7 c_p-9 i\right)} \\
            &\qquad +\frac{4 f_{02} g_{11}
            \left(9-i c_p\right) \left(4 c_p-5 i\right)}{\left(c_p-i\right)
            \left(c_p+9 i\right) \left(7 c_p-9 i\right)}, \\
        \gamma_3 &= \frac{f_{11} \left(f_{20} \left(c_p-19 i\right)-9 i g_{11}\right)}{9
            \left(c_p-i\right)}+\frac{2 i f_{02} g_{20}}{-8 c_p+9 i}+f_{21}, \\
        \gamma_4 &= 6 \left(27+4 c_p \left(c_p+3 i\right)\right) \left(f_{11} g_{11} \left(5+3 i
        c_p\right) \left(c_p+9 i\right)+\left(c_p+i\right) \left(c_p+9 i\right)
        \left(7 c_p-9 i\right) \left(2 f_{02} g_{20}+f_{21}\right)\right) \\
        &\qquad +6 \left(27+4 c_p \left(c_p+3 i\right)\right)\left(f_{11} f_{20}
        \left(7 c_p-9 i\right) \left(-19+c_p \left(c_p+12 i\right)\right)\right), \\
        \gamma_5 &= \left(c_p+i\right) \left(c_p+9 i\right) \left(7 c_p-9 i\right) \Big(2 f_{02}
        g_{02} \left(24 c_p^2+70 i c_p+171\right) \\&\qquad+9 f_3 \left(27+4 c_p \left(c_p+3
        i\right)\right)+3 f_{02} f_{11} \left(8 c_p^2+26 i c_p+57\right)\Big), \\
        \gamma_6 &= \frac{f_{11} g_{02}}{9+6 i c_p}+\frac{2 i f_{02} g_{11}}{c_p+i}+\frac{i
            f_{11}^2}{c_p+i}+\frac{2 f_{02} f_{20}}{9-2 i c_p}+f_{12}, \\
        \gamma_7 &= 2 \left(\frac{2 i f_{11} g_{20} \left(c_p+5 i\right)}{\left(c_p+i\right)
            \left(c_p+9 i\right)}+\frac{g_{11}^2 \left(5+3 i c_p\right)}{9+c_p
            \left(7 c_p-2 i\right)}+f_{20} g_{11}+2 g_{02} g_{20}+g_{21}\right), \\
        \gamma_8 &= \frac{f_{02} g_{11} \left(4 c_p+19 i\right)}{2 c_p+9 i}+\frac{2 g_{02}^2 \left(12
            c_p-19 i\right)}{6 c_p-9 i}+3 g_{03}, \\
        \gamma_9 &= g_{11} g_{20} \left(2+\frac{1}{9+8 i c_p}\right)+\frac{38 f_{20}
            g_{20}}{9}+3 g_{30}, \\
        \gamma_{10} &= 2 \left(\frac{5 f_{11} g_{11}}{9+c_p \left(c_p+8 i\right)}+g_{02} g_{11}
        \left(-\frac{i}{c_p-i}+\frac{i}{-7 c_p+9 i}+1\right)+2 f_{02}
        g_{20}+g_{12}\right), \\
        \gamma_{11} &= -\frac{2 i f_{11} g_{20}}{c_p-i}-\frac{i g_{11}^2}{c_p-i}+\frac{2 g_{02}
            g_{20}}{9+8 i c_p}+\frac{f_{20} g_{11}}{9}+g_{21}, \\
        \gamma_{12} &= \frac{g_{11} \left(f_{11} \left(9+6 i c_p\right)+g_{02} \left(19+11 i
            c_p\right)\right)}{6 c_p^2-3 i c_p+9}+\frac{2 f_{02} g_{20}}{9-2 i
            c_p}+g_{12}.
		\end{align*}}
	
	\printbibliography
	
\end{document}